\documentclass[10pt]{article}
\usepackage{graphicx}
\usepackage{amsmath,amssymb,amsthm,amsfonts}
\usepackage{amssymb}
\usepackage[title]{appendix}

\newtheorem{thm}{Theorem}[section]
\newtheorem{cor}[thm]{Corollary}
\newtheorem{lem}[thm]{Lemma}
\newtheorem{prop}[thm]{Proposition}
\newtheorem{rem}[thm]{Remark}
\newtheorem{claim}{Claim}
\newcommand{\bremark}{\begin{rem}}
\newcommand{\eremark}{\end{rem} }

\numberwithin{equation}{section}

\newcommand{\equ}[1]{(\ref{#1})}
\newcommand{\C}{{\mathbb C}}
\newcommand{\cuad}{{\sqcap\kern-.68em\sqcup}}

\newcommand{\e}{\epsilon}
\newcommand{\de}{\delta}

\newcommand{\la}{\lambda}
\newcommand{\ptl}{{\partial}}
\newcommand{\Om}{{\Omega}}

\newcommand{\R}{{\mathbb{R}}}
\newcommand{\N}{{\mathbb{N}}}

\newcommand{\lf}{\left}
\newcommand{\rg}{\right}
\newcommand{\ti}{\tilde}
\newcommand{\st}{such that }

\newcommand{\sm}{\setminus}

\newcommand\ds{\displaystyle}

\newcommand\lap{\Delta}
\newcommand{\grad}{\nabla}
\newcommand{\ml}{\mathcal}

\def\qed{{\unskip\nobreak\hfil\penalty50
         \hskip2em\hbox{}\nobreak\hfil\mbox{\rule{1ex}{1ex} \qquad}
           \parfillskip=0pt
           \finalhyphendemerits=0\par }}
\DeclareMathOperator{\re}{Re} \DeclareMathOperator{\im}{Im}

\begin{document}

\title{Non-topological condensates for the self-dual Chern-Simons-Higgs model}

\author{Manuel del Pino\footnote{Departamento de Ingenier\'ia Matem\'atica
and CMM, Universidad de Chile, Casilla 170, Correo 3, Santiago,
Chile. E-mail: delpino@dim.uchile.cl. Author supported by grants
Fondecyt 1070389 and FONDAP (Chile).}\quad Pierpaolo Esposito
\footnote{Dipartimento di Matematica e Fisica, Universit\`a degli
Studi ``Roma Tre", Largo S. Leonardo Murialdo, 1 -- 00146 Roma,
Italy. E-mail: esposito@mat.uniroma3.it. Author supported by the
PRIN project ``Critical Point Theory and Perturbative Methods for
Nonlinear Differential Equations" and the Firb-Ideas project
``Analysis and Beyond".} \quad Pablo
Figueroa\footnote{Departamento de Matem\'atica, Pontificia
Universidad Cat\'olica de Chile, Avenida Vicuna Mackenna 4860,
Macul, Santiago, Chile. E-mail: pfigueros@mat.puc.cl. Author
supported by grants Proyecto Anillo ACT-125 and Fondecyt
Postdoctorado 3120039 (Chile).} \quad Monica Musso
\footnote{Departamento de Matem\'atica, Pontificia Universidad Cat\'olica de Chile, Avenida
Vicuna Mackenna 4860, Macul, Santiago, Chile. E-mail:
mmusso@mat.puc.cl. Author supported by Fondecyt grant 1040936
(Chile), and by PRIN project ``Metodi variazionali e
topologici nello studio di fenomeni non lineari''.}}

\maketitle
%%%%%%%%%%%%%%%%%%%%%%%%%%%%%%%%%%%%%%%%%%%%%%%%%%%%%%%%%%%%%%%%%%%%%%

\begin{abstract}
\noindent For the abelian self-dual Chern-Simons-Higgs model we
address existence issues of periodic vortex configurations -- the
so-called condensates-- of non-topological type as $k \to 0$,
where $k>0$ is the Chern-Simons parameter. We provide a
positive answer to the long-standing problem on the existence of
non-topological condensates with magnetic field concentrated at
some of the vortex points (as a sum of Dirac measures) as $k \to
0$, a question which is of definite physical interest.
\end{abstract}

\vskip 0.2truein

\medskip \noindent {\bf Keywords}:
\medskip \noindent {\bf  AMS subject classification}: %35J25, 35B25, 35B40.

%%%%%%%%%%%%%%%%%%%%%%%%%%%%%%%%%%%%%%%%%%%%%%%%%%%%%%%%%%%%%%%%%%%%%%

%%%%%%%%%%%%%%%%%%%%%%%%%%%%%%%%%%%%%%%%%%%%%%%%%%%%%%%%%%%%%%%%%%%%%%
%%%%%%%%%%%%%%%%%%%%%%%%%%%%%%%%%%%%%%%%%%%%%%%%%%%%%%%%%%%%%%%%%%%%%%

\vskip 0.2truein

\section{Introduction and statement of main results}
The Chern-Simons vortex theory is a planar theory which is
physically relevant in connection with high critical temperature
superconductivity, the quantum Hall effect and anyonic particle
physics, as widely discussed by Dunne \cite{D}. Hong-Kim-Pac
\cite{HKP} and Jackiw-Weinberg \cite{JW} have proposed an abelian
self-dual model where the electrodynamics is governed only by the
Chern-Simons term. Over the Minkowski space
$(\mathbb{R}^{1+2},g)$, with metric tensor $g=\hbox{diag
}(1,-1,-1)$, the model is described by the following Lagrangean
density:
$${\cal L}_({\cal A},\phi)=\frac{k}{4}\e^{\alpha \beta
\gamma}A_\alpha F_{\beta \gamma}+D_\alpha \phi \overline{D^\alpha
\phi}-\frac{1}{k^2}|\phi|^2\left(|\phi|^2-1 \right)^2,$$
where the Chern-Simons coupling parameter $k>0$ measures the
strenght of the Chern-Simons term and the antisymmetric
Levi-Civita tensor $\e^{\alpha \beta \gamma}$ is fixed with $\e^{0
1 2}=1$. The metric tensor $g$ is used to lower and raise indices
in the usual way, and the standard summation convention over
repeated indices is adopted. The gauge potential ${\cal A}=-i
A_\alpha dx^{\alpha}$ is a $1$-form (a connection over the
principal bundle $\mathbb{R}^{1+2}\times U(1)$), $
A_\alpha:\mathbb{R}^{1+2}\to \mathbb{R}$ for $\alpha=0,1,2$, and
the Higgs field $\phi:\mathbb{R}^{1+2} \to \mathbb{C}$ is the
matter field. The gauge field $F_{\cal A}=-\frac{i}{2}F_{\alpha
\beta}dx^\alpha \wedge dx^\beta$ is a $2-$form (the curvature of
${\cal A}$), where $F_{\alpha \beta}=\partial_\alpha
A_\beta-\partial_\beta A_\alpha$, and the Higgs field $\phi$ is
weakly coupled with the gauge potential ${\cal A}$ through the
covariant derivative $D_A$ as follows: $D_A \phi=D_\alpha \phi\,
dx^\alpha$, $D_\alpha \phi=\partial_\alpha \phi-i A_\alpha \phi$
for $\alpha=0,1,2$.

\medskip \noindent The self-dual regime has been identified by Hong-Kim-Pac \cite{HKP} and
Jackiw-Weinberger \cite{JW} through the choice of the
``triple-well" potential $\frac{1}{k^2}|\phi|^2 (|\phi|^2-1)^2$
which yields to a Bogomol'nyi reduction \cite{Bog} for the
Chern-Simons-Higgs model, as we discuss below. Vortices are
time-independent ($x^0$ is the time-variable) configurations
$(\mathcal{A},\phi)$ which solve the Euler-Lagrange equations
\begin{equation} \label{ELequations}
\left\{ \begin{array}{l} \displaystyle  D_\mu D^\mu \phi=-\frac{1}{k^2}(|\phi|^2-1) (3|\phi|^2-1) \phi \\
\displaystyle  \frac{k}{2}\e^{\mu \alpha \beta} F_{\alpha
\beta}=J^\mu:=i \left(\phi \overline{D^\mu
\phi}-\overline{\phi}D^\mu \phi\right) \end{array} \right.
\end{equation}
and have finite energy. In the self-dual regime, for
energy-minimizing vortices (at given magnetic flux) the
second-order Euler-Lagrange equations are equivalent to the
first-order self-dual equations
\begin{equation}\label{CSeqs}
\left\{ \begin{array}{l} D_\pm\phi=0 \\
F_{12} \pm \frac{2}{k^2}|\phi|^2(|\phi|^2-1)=0 \\
kF_{12}+2A_0|\phi|^2=0,
\end{array}\right.
\end{equation}
where $D_{\pm}=D_1 \pm i D_2 $ and the last equation is usually
referred to as the Gauss law. In the sequel, we restrict our
attention to energy-minimizing vortices (at given magnetic flux),
and we will simply refer to them as vortices.

\medskip \noindent In the physical interpretation, the electric field $\vec{E}=(\partial_1 A_0,\partial_2 A_0,0)$ is planar, the magnetic field $\vec{B}=(0,0,F_{1,2})$ is in the orthogonal direction, and $J^0$, $\vec{J}=(J^1,J^2)$ can be identified with the charge density, current density, respectively, as in the classical Maxwell theory. Thanks to the Gauss law, vortices are both electrically and magnetically charged, a physical relevant property which was absent in the abelian Maxwell-Higgs model \cite{JaTa,Taubes}. Notice that ${\cal A}$ and $\phi$ are not observable quantities, as they are defined only up to a gauge transformation,  whereas the
electric and magnetic fields as well as the magnitude $|\phi|$ of
the Higgs field define gauge-independent quantities. The second
and third equations in (\ref{CSeqs}) only involve observable
quantities, whereas the first one $D_+ \phi=0$ (or $D_-\phi=0$) --
a  gauge invariant version of the Cauchy-Riemann equations--
implies holomorphic-type properties for the Higgs field $\phi$ (or
$\bar{\phi}$) in a suitable gauge. Following an approach first
developed by Taubes \cite{Taubes} for the abelian Maxwell-Higgs
model, vortices $(\phi,\mathcal{A})$ can be found in the form:
\begin{equation} \label{1917}
\phi=e^{\frac{u}{2} \pm i\sum_{j=1}^N  Arg(z-p_j)},\quad A_0=\pm
\frac{1}{k}(|\phi|^2-1), \quad A_1\pm iA_2=-i (\partial_1\pm
i\partial_2) \log \phi
\end{equation}
as soon as $u=\log |\phi|^2$ does solve the elliptic problem
\begin{equation}\label{1}
-\Delta u= \frac{1}{\epsilon^2}e^u(1-e^u)-4\pi \sum_{j=1}^N
\delta_{p_j},
\end{equation}
where $\epsilon=\frac{k}{2}$ and $p_1,\dots,p_N$ are the zeroes of
$\phi$ (repeated according to their multiplicities)-- usually
referred to as the vortex points (with the convention $N=0$ if
$\phi \not= 0$). We refer the interested reader to
\cite{Tbook,Ybook} and the references therein for more details and
for an extensive discussion of several gauge field theories.
%The abelian Chern-Simons model is part of a more general theory \cite{LLM} where the %electrodynamics is governed by a Maxwell term perturbed by a small Chern-Simons correction, %and it is then of physical interest the regime $\epsilon \to 0$.

\medskip \noindent For planar vortices, the finite energy condition $\int_{\mathbb{R}^2} e^u(1-e^u)<+\infty$ imposes two possible asymptotic behaviors at infinity. The topological behavior $|\phi|^2=e^u \to 1$ as $|z|\to \infty$ gives the vortex number $N$ the topological meaning of winding number for $\phi$ at infinity (up to a $\pm$ sign, depending on whether $D_+ \phi=0$ or $D_-\phi=0$), yielding to quantization
effects for the energy $E$, the magnetic flux $\Phi$ and the
electric charge $Q$ in the class of topological $N-$vortices:
$E=2\pi N$, $\Phi=\pm 2\pi N$ and $Q=\pm 2\pi kN$. The existence
of planar topological vortices has been addressed in
\cite{H,SY2,RWa}. The non-topological behavior $|\phi|^2=e^u \to
0$ as $|z|\to \infty$ has no counterpart in the abelian
Maxwell-Higgs model, and the possible coexistence of topological
and non-topological $N-$vortices is the main new feature in
Chern-Simons theories. After the seminal work \cite{SY1} in a
radial setting with a single vortex point (see also \cite{CHMY}
for related results), it has been a challenging problem to find
planar non-topological $N-$vortices \cite{ChI,CFL} for an
arbitrary configuration of $p_1,\dots,p_N$. Surprisingly, two
different classes have been found by using different limiting
problems: the singular Liouville equation in \cite{ChI} or the
Chern-Simons equation $-\Delta U=e^U(1-e^U)-4\pi \delta_0$ in
\cite{CFL}. Since the latter problem has no scale-invariance, in \cite{CFL} the points
$p_1,\dots,p_N$ are taken along the vertices of a regular
$N-$polygon in order to glue together $U(\frac{x-p_j}{\epsilon})$, $j=1,\dots,N$,
for there is no freedom to adjust the height at each $p_j$ to
account for the interaction, but the approximating function has
invertible linearized operator.

\medskip \noindent Since the theoretical prediction by Abrikosov \cite{Abr}, the appearance of lattice structure, in the form of spatially periodic vortices, has been experimentally observed. To account for it, the model is formulated on
$$\Omega=\{z=t\omega_1+s \omega_2: \:(t,s) \in (-\frac{1}{2},\frac{1}{2}) \times (-\frac{1}{2},\frac{1}{2})\},$$
where $\omega_1,\: \omega_2 \in \mathbb{C} \setminus \{0\}$
satisfy $\hbox{Im }(\frac{\omega_2}{\omega_1})>0$. Condensates are
time-independent configurations $(\mathcal{A},\phi)$ which solve
the Euler-Lagrange equations \eqref{ELequations}, have finite
energy and satisfy the 't Hooft boundary conditions \cite{tHo}:
\begin{equation}\label{tH}
e^{i\xi_k(z+\omega_k)}\phi(z+\omega_k)=e^{i\xi_k(z)}\phi(z),\quad
A_0(z+\omega_k)=A_0(z), \quad \left(A_j +\partial_j
\xi_k\right)(z+\omega_k)=\left(A_j +\partial_j \xi_k\right)(z)
\end{equation}
for all $z\in \Gamma^1\cup \Gamma^2 \setminus \Gamma^k$ and
$k=1,2$, where $\Gamma^1=\{z=t \omega_1
-\frac{1}{2}\omega_2:\:|t|<\frac{1}{2} \}$,
$\Gamma^2=\{z=-\frac{1}{2}\omega_1+t \omega_2:\:|t|<\frac{1}{2}
\}$ and $\xi_1$, $\xi_2$ are real-valued smooth functions defined
in a neighborhood of $\Gamma^2 \cup\{\omega_1+\Gamma^2\}$,
$\Gamma^1 \cup\{\omega_2+\Gamma^1\}$, respectively. For
energy-minimizing vortices (at given magnetic flux) the
Euler-Lagrange equations \eqref{ELequations} are still equivalent
to the self-dual ones \eqref{CSeqs}. Since \eqref{tH} just reduces
to a double periodicity for the observable quantities $F_{12}$ and
$|\phi|$ in $\Omega$, a configuration $(\mathcal{A},\phi)$ in the
form \eqref{1917} does solve \eqref{CSeqs} as soon as $u=\log
|\phi|^2$ is a doubly-periodic solution of \eqref{1} in $\Omega$,
see \cite{CY,T} for an exact derivation.

\medskip \noindent Hereafter, up to a translation, let us assume that $\phi\not=0$ on
$\partial \Omega$ (i.e. $p_1,\dots,p_N \in \Omega$) in such a way the winding number $\hbox{deg
}(\phi,\partial\Omega,0)$ is well-defined, and the vortex number $N$ is simply
given by $|\hbox{deg }(\phi,\partial\Omega,0)|$. By \eqref{tH} we
still have quantization effects as in the case of planar
topological vortices: $E=2\pi N$, $\Phi=\pm 2\pi N$ and $Q=\pm
2\pi kN$ , where the $\pm$ sign depends on whether $D_+\phi=0$ or
$D_-\phi=0$. Hereafter, up to change $\phi$ with $\bar \phi$, let
us assume that $D_+\phi=0$ and restrict our attention to
energy-minimizing condensates (at given magnetic flux), simply
referred to as condensates.

\medskip \noindent Letting $G(z,p)$ be the Green function of $-\Delta$ in
$\Omega$ with pole at $p$:
$$\left\{ \begin{array}{ll} -\Delta G(z,p)=\delta_p-\frac{1}{|\Omega|}&\hbox{in }\Omega\\
\int_\Omega G(z,p)dz=0, & \end{array}\right.$$
%By translation invariance we have that $G(z,p)=G(z-p,0)$, and $G(z,0)$ can be decomposed as %$G(z,0)=-{1\over 2\pi}\log|z|+H(z)$, where $H$ is a (not doubly-periodic) function with %$\Delta H= \frac{1}{|\Omega|}$ in $\Omega$.
one is led to consider the following equivalent regular version of
(\ref{1}):
\begin{equation}
\label{2}-\Delta v=\frac{1}{\epsilon^2} e^{u_0+v}
(1-e^{u_0+v})-\frac{4\pi N}{|\Omega|}\qquad \hbox{in }
\Omega\end{equation} in terms of $v=u-u_0$, where $u_0=-4\pi
\displaystyle \sum_{j=1}^N G(z,p_j)$ and the potential $e^{u_0}$
is a smooth non-negative function which vanishes exactly at
$p_1,\dots,p_N$. By translation invariance, notice that
$G(z,p)=G(z-p,0)$, and $G(z,0)$ can be decomposed as
$G(z,0)=-{1\over 2\pi}\log|z|+H(z)$, where $H$ is a (not
doubly-periodic) function with $\Delta H= \frac{1}{|\Omega|}$ in
$\Omega$. If $v$ is a solution of \eqref{2}, by integration over
$\Om$ notice that
\begin{equation}\label{ci0}
\int_\Om e^{u_0+v}(1-e^{u_0+v})=\int_\Om
|\phi|^2(1-|\phi|^2)=2\e^2 \int_\Omega F_{12} =4\pi N\e^2
\end{equation}
in view of \eqref{CSeqs}, yielding to the necessary condition
$$16\pi N\e^2=|\Omega| -4 \int_\Om\lf(e^{u_0+v}-{1\over2}\rg)^2<|\Omega|$$
for the solvability. According to \cite{CY}, Caffarelli and Yang
show the existence of $0<\e_c< \sqrt{\frac{|\Omega|}{16\pi N}}$ so
that \eqref{1} has a maximal doubly-periodic solution $u_\e$ for
$0<\epsilon<\epsilon_c$, while no solution exists for $\e >\e_c$.
Notice that \eqref{2} admits a variational structure with energy
functional
$$J_\e(v)={1\over2}\int_\Om|\grad
v|^2+{1\over2\e^2}\int_\Om\lf(e^{u_0+v}-1\rg)^2+{4\pi
N\over|\Om|}\int_\Om v$$
where $v \in H^1(\Om)=\{v \in H_{\text{loc}}^1(\R^2):\, v\text{
doubly periodic in }\Om\}$. Later, Tarantello \cite{T} shows that
the maximal solution $u_\e$ is a local minimum for $J_\e$ in
$H^1(\Om)$, and a second solution $u^\e$ is found as a
mountain-pass critical point for $J_\e$.

\medskip \noindent To each solution $u$ of \eqref{1} we can associate the $N-$condensate $(\mathcal{A},\phi)$ in the form \eqref{1917} (with the $+$ sign as we agreed), and let $(\mathcal{A}_\e,\phi_\e)$, $(\mathcal{A}^\e,\phi^\e)$ be the ones corresponding to $u_\e$, $u^\e$. Concerning the asymptotic behavior as $\e \to 0$, by \eqref{ci0} we can expect two classes of $N-$condensates:
\begin{itemize}
\item $|\phi| \to 1$ as $\e \to 0$ (``topological" behavior),
\item $|\phi| \to 0$ as $\e \to 0$ (``non-topological" behavior),
\end{itemize}
to be understood in suitable norms. For example,
$(\mathcal{A}_\e,\phi_\e)$ exhibits ``topological" behavior:
$$|\phi_\e| \to 1 \hbox{ in }C_{\hbox{loc}}(\bar\Om\sm\{p_1,\dots, p_N\}),$$
with
\begin{equation} \label{1820}
(F_{12})_\e \rightharpoonup 2\pi \sum_{j=1}^N \delta_{p_j} \quad
\hbox{in the sense of measures}
\end{equation}
as $\e \to 0$ according to \eqref{ci0}, see \cite{T}. The
concentration property \eqref{1820} for the magnetic field has a
definite physical interest, and supports the use of the
terminology ``vortex points" for the zeroes $p_1,\dots,p_N$ of the
Higgs field $\phi$. The $N-$condensate $(\mathcal{A}^\e,\phi^\e)$
has in general a different asymptotic behavior as $\e \to 0$:
\begin{itemize}
\item[(i)] when $N=1$, $|\phi^\e|\to 0$ in $C^m(\bar \Omega)$, for
all $m \geq 0$, and $(F_{12})^\e$ is a compact sequence in
$L^1(\Omega)$ (see \cite{T}); \item[(ii)] when $N=2$,
$|\phi^\e|\to 0$ in $C(\bar \Omega)$ and either $(F_{12})^\e$ is a
compact sequence in $L^1(\Omega)$ or $(F_{12})^\e \rightharpoonup
4\pi \delta_q$ in the sense of measures, for some $q
\not=p_1,\:p_2$ with $u_0(q)=\max_\Omega u_0$, depending on
whether
$$I(v)={1\over2}\int_\Om|\grad
v|^2-8\pi \log \left(\int_\Omega e^{u_0+v} \right) +{8\pi
\over|\Om|}\int_\Om v $$ attains its infimum or not in $H^1(\Om)$
(see \cite{NoTa3}, and also \cite{DJLW2}); \item[(iii)] when $N\geq
3$, $|\phi^\e|\to 0$ in $C(\bar \Omega)$ and $ (F_{12})^\e
\rightharpoonup 2\pi N \delta_q $ in the sense of measures, for
some $q \not=p_1,\dots,p_N$ with $u_0(q)=\max_\Omega u_0$ (see
\cite{Ch}).
\end{itemize}
In \cite{DJLPW} it is shown the existence of $N-$condensates
$(\mathcal{A},\phi)$ so that $|\phi|\to 0$ a.e. in $\Omega$ as $\e
\to 0$.
%and in \cite{ChK} the asymptotic behavior of a general sequence of $N-$condensates is %addressed.
Concerning the case $N=2$, it is a very difficult question, which
has been discussed in \cite{CLW,LiWa} for $p_1=p_2$, to know
whether or not $I$ attains the infimum in $H^1(\Omega)$. An
alternative approach of perturbative type has revelead to be
successful for $N=2$ \cite{LinYan1} (see also \cite{EsFi} among other things) by constructing a
sequence of $2-$condensates for which the second alternative in (ii) does hold, for a critical point $q$ of $u_0$.
The same approach works as well for $N\geq 3$, provided the
concentration points of the magnetic field are not vortex points.

\medskip \noindent The existence of non-topological $N-$condensates with magnetic field concentrated at vortex points as $\e \to 0$ (like in \eqref{1820}) is the main issue from a physical viewpoint and has not received an answer so far. A first partial answer has
been provided by Lin and Yan \cite{LinYan} who construct
$N-$condensates $(\mathcal{A}_\e,\phi_\e)$ so that $(F_{12})_\e
\rightharpoonup 2 \pi N \delta_{p_j}$ in the sense of measures as
$\e \to 0$, as soon as $N>4$ and $p_j$ is a simple vortex point in
$\{p_1,\dots,p_N\}$. As in \cite{CFL}, they make use of the
Chern-Simons equation $-\Delta U=e^U(1-e^U)-4\pi \delta_0$ as
limiting problem, which is not suitable to manage multiple concentration points. Moreover, such a condensate does satisfy
$\max_\Omega |\phi_\e|\geq c>0$ for $\e$ small and $|\phi_\e|\to
0$ in $C_{\hbox{loc}} (\bar \Omega \setminus \{p_j\})$, which fits the notion of ``non-topological" behavior in a weak sense. Our aim is to extend to $N-$condensates the
perturbative approach developed by Chae and Imanuvilov \cite{ChI}
for planar $N-$vortices, based on the use of the singular
Liouville equation as limiting problem. As far as non-topological behavior, let us stress that the problem on the torus is much more rigid than the planar case,
as well illustrated by the quantization property $\Phi=2\pi N$
(valid just in the doubly-periodic situation). For example, when
$F_{12}$ is concentrated like a Dirac measure at a vortex point
$p_l$, by the use of Liouville profiles it is natural, as we will
see, to have $4\pi(n_l+1)$ as concentration mass of $F_{12}$ at
$p_l$, where $n_l$ is the multiplicity of $p_l$ in the set
$\{p_1,\dots,p_N\}$, and then the relation $2\pi N=4\pi
\displaystyle \sum_{l=1}^m (n_l+1)$ does hold as soon as $F_{12}
\rightharpoonup 4\pi \displaystyle  \sum_{l=1}^m (n_l+1)
\delta_{p_l}$ in the sense of measures. In particular, the
concentration of the magnetic field can not take place at all the
vortex points $p_1,\dots,p_N$ as in the planar case \cite{ChI}. Let us stress that the $N-$condensates constructed in \cite{Nol} have exactly such a concentration property and then violate the balancing condition \eqref{hhh}.

\medskip \noindent Our aim is to provide a general answer to the long-standing question on the existence of non-topological $N-$condensates with magnetic field concentrated at some vortex points. Compared with \cite{ChI}, our main result is rather surprising and reads as follows.
\begin{thm} \label{mainbb}
Let $\{p_1,\dots,p_m\}$ be a subset of the vortex set
$\{p_1,\dots,p_N\} \subset \Omega$, $\{p_j\}_j$ be the remaining points and
$n_l$, $n_j$ be the corresponding multiplicities so that
\begin{equation} \label{hhh}
2\pi N=4\pi \sum_{l=1}^m (n_l+1).
\end{equation}
Letting $\mathcal{H}_0$ be a meromorphic function in $\Omega$
so that $|\mathcal{H}_0(z)|^2=e^{u_0+8\pi \sum_{l=1}^m (n_l+1) G(z,p_l)}$ (which exists and is unique up to rotations),  assume that $\mathcal{H}_0$ has zero residue at each $p_1,\dots,p_m$. Letting $\sigma_0(z)=-\left( \int^z \mathcal{H}_{0}(w) dw \right)^{-1}$ (a well-defined meromorphic function), assume that 
%Assume that
%\begin{equation} \label{1013}
%\frac{d^{n_l+1}}{dz^{n_l+1}} \left[ \mathcal{H}_{0}(z) \prod_{l'
%\not= l} (z-p_{l'})^{-(n_{l'}+2)}  \right](p_l)=0
%\end{equation}
%for all $l=1,\dots,m$ (so that $\sigma_0$ is a well-defined meromorphic function) 
\begin{equation} \label{ggg}
D_0=\frac{1}{\pi} \left[ \int_{\Omega \setminus \sigma_0^{-1}(B_\rho(0))} e^{u_0+8\pi
\sum_{l=1}^m (n_l+1)G(z,p_l)} - \sum_{l=1}^m (n_l+1)
\int_{\mathbb{R}^2 \setminus B_\rho(0)}
\frac{dy}{|y|^4}\right]<0
\end{equation} 
for small $\rho>0$ and the ``non-degeneracy condition"
$\hbox{det }A \not=0$, where $A$ is given by \eqref{matrixA}. Then,
for $\e>0$ small there exists $N-$condensate
$(\mathcal{A}_\e,\phi_\e)$ so that $|\phi_\e| \to 0$ in $C(\bar
\Omega)$ and
\begin{equation} \label{magconc}
(F_{12})_\e \rightharpoonup 4\pi \displaystyle \sum_{l=1}^m (n_l+1) \delta_{p_l}
\end{equation}
weakly in the sense of measures, as $\e \to 0$.
\end{thm}
\noindent Notice that we can also allow some concentration point
not to be a vortex point, by simply adding it to the vortex set
with null multiplicity. In section \ref{examples} we will see that
in the double-vortex case $N=2$ Theorem \ref{mainbb} essentially
recovers the result in \cite{EsFi,LinYan1} concerning single-point
concentration, for the assumptions just reduce to have the
concentration point $q \not=p_1,p_2$ as a non-degenerate critical
point of $u_0$ with $D_0<0$ (for similar
results concerning the Liouville equation, see \cite{BaPa,dkm,EGP} in case of bounded domains with
Dirichlet b.c. and \cite{Fi} in case of a flat two-torus). Despite of the complex statement, for a rectangle $\Omega$ with $p_1=0$,
$p_2=\frac{\omega_1}{2}$, $p_3=\frac{\omega_2}{2}$ and $p_4=\frac{\omega_1+\omega_2}{2}$, and $n_1,n_2,n_3,n_4$ even
multiplicities with $\frac{n_4}{2}$ odd, we will check in section \ref{examples} that the assumptions of Theorem
\ref{mainbb} do hold for $m=1$ and concentration point $p_1$, up to perform a small translation so to have $p_j \in \Omega$. For computational simplicity, the ``non-degeneracy condition" will be checked just for a square with $n=n_3=2$ and $(n_1,n_2)=(2,0)$ or viceversa. Even more important, examples with $m\geq 2$ will be discussed in section \ref{general}.

\medskip \noindent Following an approach developed by Tarantello \cite{T} and exploited in \cite{NoTa3}, \eqref{2} can be seen as a perturbed mean-field equation \eqref{3} with potential $e^{u_0}$ and unperturbed part 
\begin{equation} \label{10100}
-\Delta w= 4\pi N \left(\frac{e^{u_0+w} }{\int_\Omega
e^{u_0+w}}-\frac{1}{|\Omega|}\right).
\end{equation}
Since $e^{u_0}$ vanishes like $|z-p_l|^{2n_l}$ near each $p_l$, $l=1,\dots,m$, the Liouville equation $-\Delta U=|z|^{2n} e^U$ will play a central role in the construction of an approximating function in the perturbative approach. Since $U_{\delta,\sigma_0}=\log \frac{8 \delta^2}{(\delta^2 +|\sigma_0|^2)^2}$ does solve $-\Delta U= |\sigma_0'|^2 e^U$ in $\Omega \setminus \{\hbox{poles of }\sigma_0\}$, a natural choice is $\sigma_0=z^{n+1}$ when $m=1$ and $p_1=0$. Letting $P$ be a projection operator on the space of doubly-periodic functions, the approximation rate of $PU_{\delta,z^{n+1}}$ is unfortunately not sufficiently small to carry out the argument, a problem which often arises in perturbation arguments and is usually overcome by refining the ansatz via linear theory around the approximating function. However, such a procedure would require several subsequent refinements, yielding in general to a high level of complexity. Inspired by \cite{DEM4}, in section \ref{improved} we will take advantage of the Liouville formula to use the inner parameter $\sigma_0$, present in the Liouville formula, to get improved profiles. Since $PU_{\delta,\sigma_0} \sim U_{\delta,\sigma_0}-\log(8\delta^2)+\log |\sigma_0|^4+8\pi (n+1)G(z,0)$ as $\delta \to 0$, $PU_{\delta,\sigma_0}$ is a good approximate solution of \eqref{10100} if $\frac{|\sigma_0'|^2}{|\sigma_0|^4}=|(\frac{1}{\sigma_0})'|^2=e^{u_0+8\pi(n+1)G(z,0)}$. By definition of $\mathcal{H}_0$, it is enough to find a meromorphic $\sigma_0$ with $(\frac{1}{\sigma_0})'=\mathcal{H}_0$, a solvable equation if and only if $\mathcal{H}_0$ has zero residue at its unique pole $0$. As we will discuss precisely in Remark \ref{remark2bis}, the assumption on the residues of $\mathcal{H}_0$ is then necessary in our context. Moreover, since $\mathcal{H}_0$ has a pole at $0$ of multiplicity $n+2$ and zeroes $p_j$'s of multiplicities $n_j$, by the property $\mathcal{H}_0(z+\omega_j)=e^{i\theta_j}\mathcal{H}_0(z)$, $j=1,2$, near $\partial \Omega$ for some $\theta_1,\theta_2 \in \mathbb{R}$ we deduce that
$$0=\frac{1}{2\pi i} \int_{\partial \Omega} \frac{\mathcal{H}_0'}{\mathcal{H}_0}dz=n+2-\sum_j n_j=2(n+1)-N,$$
providing \eqref{hhh} as a necessary and sufficient condition for the existence of such $\mathcal{H}_0$ (the sufficient part in shown in next section). $D_0<0$ and the ``non-degeneracy condition'' will be necessary to determine $\delta$ and $a$, a sort of small translation parameter accounting for the perturbation term in  \eqref{3}, according to the asymptotic expansion for the corresponding ``reduced equations" as derived in section \ref{reduced}. Theorem \ref{mainbb} is proved in section \ref{mainresults} when $m=1$ and in section \ref{general} when $m \geq 2$.

%%%%%%%%%%%%%%%%%%%%%%%%%%%%%%%%%%%%%%%%%%%%%%%%%%%%%%%%%%%%%%%%%%%%%%%%%%%%%%%%%%%%%%%%%%%%%%%%%%%%%%%%%%%%%%%%%%%%%%%%%%%%%%
%%%%%%%%%%%%%%%%%%%%%%%%%%%%%%%%%%%%%%%%%%%%%%%%%%%%%%%%%%%%%%%%%%%%%%%%%%%%%%%%%%%%%%%%%%%%%%%%%%%%%%%%%%%%%%%%%%%%%%%%%%%%%%

\section{Improved Liouville profiles} \label{improved}
\noindent Let us decompose any solution $v$ of (\ref{2}) as
$v=w+c$, where $c=\frac{1}{|\Omega|}\int_\Omega v$. In this way,
$w$ has zero average: $\int_\Omega w dz=0$, and by (\ref{ci0}) one
has
$$e^{2c} \int_\Omega e^{2u_0+2w}-e^c \int_\Omega e^{u_0+w}+4\pi N
\epsilon^2=0.$$ This last identity then provides a relation
between $c$ and $w$ in the form $c=c_\pm (w)$, where
\begin{equation}\label{cc}
e^{c_\pm(w)}=\frac{8\pi N \epsilon^2}{\int_\Omega e^{u_0+w} \mp
\sqrt{(\int_\Omega e^{u_0+w})^2-16\pi N \epsilon^2 \int_\Omega
e^{2u_0+2w}}}, \end{equation} whenever $\big(\int_\Om
e^{u_0+w}\big)^2-16\pi N \epsilon^2 \int_\Om e^{2u_0+2w}\ge 0$.
The two possible choice of ``plus'' or ``minus'' sign in
\eqref{cc} is another indication of multiple solutions for
\eqref{2}. In \cite{T}, topological solutions are
characterized to satisfy \eqref{cc} with the ``plus'' sign. Since
we are interested to non-topological solutions, it is natural to
restrict the attention to the case $c=c_-(w)$, reducing problem
(\ref{2}) to the following equation in $\Omega$:
\begin{equation} \label{3} \left\{ \begin{array}{rl}  -\Delta w=& \displaystyle 4\pi
N\left(\frac{e^{u_0+w} }{\int_\Omega
e^{u_0+w}}-\frac{1}{|\Omega|}\right)
\\
&\displaystyle+\frac{64 \pi^2N^2 \epsilon^2 \int_\Omega
e^{2u_0+2w}}{(\int_\Omega e^{u_0+w}+\sqrt{(\int_\Omega
e^{u_0+w})^2-16\pi N\epsilon^2\int_\Omega
e^{2u_0+2w}})^2}\left(\frac{e^{u_0+w}}{\int_\Omega
e^{u_0+w}}-\frac{
e^{2u_0+2w}}{\int_\Omega e^{2u_0+2w}}\right)  \\
\displaystyle \int_\Omega w=0. \end{array}\right.\end{equation}

\medskip \noindent Here and in the next sections, we first discuss the case $m=1$ in Theorem \ref{mainbb}. Assume that $p$ is present $n-$times in $\{p_1,\dots,p_N\}$, and denote by $p_j'$s the remaining points in the set $\{p_1,\dots,p_N\}$ with corresponding multiplicities $n_j'$s. Up to a translation, we are assuming that $p_j \in \Omega$ 
%(modulo $\omega_1 \mathbb{Z}+\omega_2 \mathbb{Z}$) 
for $j=1,\dots,N$, a crucial property which will simplify the arguments below. Since  the assumptions in Theorem \ref{mainbb} for the concentration at $p$ are just local properties, for simplicity in the notations let us simply consider the case $p=0$.

\medskip \noindent Since $e^{u_0}$ behaves like $|z|^{2n}$ as $z \to 0$, the local profile of $w$ near $0$ will be given in terms of solutions of the ``singular" Liouville equation:
\begin{equation}\label{starr}
-\Delta U=  |z|^{2n}e^U.
\end{equation}
Recall that by Liouville formula the function
$$\log \frac{8|F'|^2}{(1+|F|^2)^2}$$
does solve $-\Delta U=e^U$ in the set $\{F'\not= 0 \}$, for any
holomorphic map $F$. For entire solutions of \eqref{starr} with finite-energy:
$\int_{\mathbb{R}^2} |z|^{2n}e^U<+\infty$, it is well known that
necessarily $F(z)=\frac{z^{n+1}-a}{\delta}$, and then all the
entire finite-energy solutions of \eqref{starr} are classified as
$$U_{\delta,a}(z)=\log \frac{8 (n+1)^2 \delta^2}{(\delta^2 +|z^{n+1}-a|^2)^2},\quad \delta>0, \:a \in \mathbb{C}.$$
Moreover, we have that $\int_{\mathbb{R}^2}
|z|^{2n}e^{U_{\delta,a}}=8\pi(n+1)$. Since by construction the corresponding
$v=w+c_-(w)$ will satisfy
$$ \frac{1}{\epsilon^2} e^{u_0+v}\left(1-e^{u_0+v}\right) \rightharpoonup 8\pi (n+1) \delta_0$$
in the sense of measures, the balance condition
\begin{equation} \label{balance}
2\pi N=4 \pi(n+1)
\end{equation}
is necessary in view of (\ref{ci0}).

\medskip \noindent Assume for simplicity $e^{u_0}=|z|^{2n}$. Since $ \int_\Omega
|z|^{2n} e^{U_{\delta,a}}\to 8\pi(n+1)$ as $\delta \to 0$, by
(\ref{balance}) we have the asymptotic matching of $-\Delta
U_{\delta,a}=  |z|^{2n} e^{U_{\delta,a}}$ and $4\pi N
\frac{|z|^{2n} e^{U_{\delta,a}}}{\int_\Omega
|z|^{2n}e^{U_{\delta,a}}}$ as $\delta \to 0$.  To correct
$U_{\delta,a}$ into a doubly-periodic function, we consider the
projection $PU_{\delta,a}$ of $U_{\delta,a}$ as the solution of
$$\left\{ \begin{array}{ll} -\Delta PU_{\delta,a}=-\Delta U_{\delta,a} +\frac{1}{|\Omega|} \int_\Omega \Delta U_{\delta,a}& \hbox{in }\Omega\\
\int_\Omega PU_{\delta,a}=0.& \end{array} \right.$$ In this way,
we gain the constant term
$$\frac{1}{|\Omega|} \int_\Omega \Delta U_{\delta,a}   =-\frac{1}{|\Omega|} \int_\Omega |z|^{2n} e^{U_{\delta,a}} \to -\frac{4\pi N}{|\Omega|} \qquad \hbox{as }\delta \to 0$$
in view of (\ref{balance}), and
%As we will see, the
%$\epsilon^2-$perturbation term in (\ref{3}) will be crucial to
%find the value of $\delta>0$.
we still need to check that the difference between $-\Delta
U_{\delta,a}= |z|^{2n}e^{U_{\delta,a}}$ and $4\pi N \frac{|z|^{2n}
e^{PU_{\delta,a}}}{\int_\Omega |z|^{2n} e^{PU_{\delta,a}}}$ is
asymptotically small. Thanks to an asymptotic expansion of
$PU_{\delta,a}$ in terms of $U_{\delta,a}$, we will see that the
difference is small (i.e. $PU_{\delta,a}$ is an approximating
function of (\ref{3})) but behaves at most like
$|z|^{2n}e^{U_{\delta,a}}O(|z|+\delta^2)$ which is not
sufficiently small. A first refinement of the ansatz via the
linear theory around $PU_{\delta,a}$ could improve the pointwise
error estimate into $|z|^{2n}e^{U_{\delta,a}}O(|z|^2+\delta^2)$,
which unfortunately is in general still not enough. Since there is
a strong mismatch between the dependence of $U_{\delta,a}$ on
$z^{n+1}$ and that of the error on $z$ (or even on $z^2$), we
should push such a procedure through several subsequent
refinements. Instead, we play directly with the inner parameters
present in the Liouville formula, for we have more flexibility in
the choice of $F(z)$ on bounded domains. Hereafter, let us fix an
open simply-connected domain $\tilde \Omega$ so that
$\overline{\Omega}\subset \tilde \Omega$ and $\tilde \Omega \cap
\,\left(\omega_1 \mathbb{Z}+\omega_2 \mathbb{Z}\right)=\{0\}$, and
set $\mathcal{M}(\overline{\Omega})=\{ \sigma
\Big|_{\overline{\Omega}}: \sigma\hbox{ meromorphic in }\tilde
\Omega\}$. Let $\delta \in (0,+\infty)$, $a\in \mathbb{C}$ and
$\sigma \in \mathcal{M}(\overline{\Omega})$ be a function which
vanishes only at $0$ with multiplicity $n+1$. Since $\log
|\sigma'(z)|^2$ is harmonic in $\{ \sigma' \not= 0\}$, the choice
$F(z)=\frac{\sigma(z)-a}{\delta}$ yields to solutions
$$U_{\delta,a,\sigma}(z)=\log \frac{8 \delta^2}{(\delta^2 +|\sigma(z)-a|^2)^2}$$
of $-\Delta U= |\sigma'(z)|^2 e^U$ in $\Omega \setminus
\{\hbox{poles of }\sigma\}$, for $U_{\delta,a,\sigma}$ is a smooth
function up to $\{\sigma'=0\}$.

\medskip \noindent The guess is so to find a better local approximating function $PU_{\delta,a,\sigma}$ for
a suitable choice of $\sigma$, where $PU_{\delta,a,\sigma}$ does
solve
\begin{equation} \label{lll}
\left\{ \begin{array}{ll} -\Delta PU_{\delta,a,\sigma} =
|\sigma'(z)|^2 e^{U_{\delta,a,\sigma}} -\frac{1}{|\Omega|} \int_\Omega |\sigma'(z)|^2 e^{U_{\delta,a,\sigma}}& \hbox{in }\Omega\\
\int_\Omega PU_{\delta,a,\sigma}=0.&
\end{array} \right.
\end{equation}
Notice that $PU_{\delta,a,\sigma}$ is well-defined and smooth as
long as $\sigma \in \mathcal{M}(\overline{\Omega})$, no matter
$\sigma$ has poles or not.

\medskip \noindent Recall that $G(z,0)$ can be thought as a doubly-periodic function in $\mathbb{C}$ with singularities on the lattice vertices $\omega_1 \mathbb{Z}+\omega_2\mathbb{Z}$, and $H(z)=G(z,0)+\frac{1}{2\pi} \log |z|$ is then a smooth function in $2\Omega$ with $\Delta H=\frac{1}{|\Omega|}$. Since $2 \Omega$ is simply-connected, we can find an holomorphic function $H^*$ in $2 \Omega$ having the harmonic function $H-\frac{|z|^2}{4|\Omega|}$ as real part. Since $p_j \in \Omega$, take $\tilde \Omega$ close to $\Omega$ so that $\tilde \Omega-p_j \subset 2 \Omega$ for all $j=1,\dots, N$. The function
\begin{equation} \label{definitionH}
\mathcal{H}(z)=  \prod_j  (z-p_j)^{n_j} \hbox{exp} \left(
4\pi(n+1) H^*(z) -2\pi\sum_{j=1}^N
H^*(z-p_j)-\frac{\pi}{2|\Omega|}\sum_{j=1}^N |p_j|^2
+\frac{\pi}{|\Omega|}z \overline{\sum_{j=1}^N p_j}\right)
\end{equation}
is holomorphic in $\tilde \Omega$ with
\begin{equation} \label{keyrelationH}
|\mathcal{H}(z)|^2=\frac{1}{|z|^{2n}} e^{u_0+8\pi(n+1)
H(z)}=e^{4\pi(n+2)H(z)-4 \pi \sum_j n_j G(z,p_j)} \qquad \hbox{in
}\tilde \Omega
\end{equation}
in view of \eqref{balance}. The meromorphic function $\mathcal{H}_0(z)=\frac{\mathcal{H}(z)}{z^{n+2}}$ does satisfy $|\mathcal{H}_0(z)|^2=e^{u_0+8\pi(n+1)
G(z,0)}$ in $\tilde \Omega$.
\bremark \label{1149} For simplicity in the notations, we are considering the case $p=0$. When $p \not=0$, by assuming $\tilde \Omega-p \subset 2\Omega$ the function
\begin{eqnarray*}
\mathcal{H}^p(z)&=&  \prod_j  (z-p_j)^{n_j} \hbox{exp} \left(
4\pi(n+1) H^*(z-p) +\frac{\pi(n+1)}{|\Omega|}|p|^2-\frac{2\pi(n+1)}{|\Omega|}z \bar p \right) \times\\
&&\times \hbox{exp} \left(-2\pi\sum_{j=1}^N
H^*(z-p_j)-\frac{\pi}{2|\Omega|}\sum_{j=1}^N |p_j|^2
+\frac{\pi}{|\Omega|}z \overline{\sum_{j=1}^N p_j}\right)
\end{eqnarray*}
is holomorphic in $\tilde \Omega$ with
$$|\mathcal{H}^p(z)|^2=\frac{1}{|z-p|^{2n}} e^{u_0+8\pi(n+1)
H(z-p)}=e^{4\pi(n+2)H(z-p)-4 \pi \sum_j n_j G(z,p_j)} \qquad \hbox{in
}\tilde \Omega$$
in view of \eqref{balance}. The meromorphic function $\mathcal{H}_0^p(z)=\frac{\mathcal{H}^p(z)}{(z-p)^{n+2}}$ does satisfy $|\mathcal{H}^p_0(z)|^2=e^{u_0+8\pi(n+1)
G(z,p)}$ in $\tilde \Omega$.
\eremark

\medskip \noindent Hereafter, for a meromorphic function $g$ in $\tilde \Omega$ the notation $\int^z g(w) dw$ stands for the anti-derivative of $g(z)$, which is a well-defined meromorphic function in the simply-connected domain $\tilde \Omega$ as soon as $g$ has zero residues at each of its poles. Since $\mathcal{H}(0)\not=0$ by \eqref{keyrelationH}, we define
\begin{equation} \label{sigma0}
\sigma_0(z)=-\left(\int^z \mathcal{H}_0(w) e^{-c_0
w^{n+1}} dw \right)^{-1}=-\left(\int^z \frac{\mathcal{H}(w) e^{-c_0
w^{n+1}}}{w^{n+2}} dw \right)^{-1},
\end{equation}
where
\begin{equation} \label{c0}
c_0=\frac{1}{\mathcal{H}(0) (n+1)! }\frac{d^{n+1}
\mathcal{H}}{dz^{n+1}}(0)
\end{equation}
guarantees that the residue of $\mathcal{H}_0(z) e^{-c_0
z^{n+1}}$ at $0$ vanishes. By construction $\sigma_0
\in \mathcal{M}(\overline{\Omega})$ vanishes only at zero with
multiplicity $n+1$, as needed, with
\begin{equation} \label{0942}
\lim_{z \to 0}
\frac{z^{n+1}}{\sigma_0(z)}=\frac{\mathcal{H}(0)}{n+1},
\end{equation}
and does solve
\begin{equation} \label{eq sigma0}
|\sigma_0'(z)|^2= |\sigma_0(z)|^4 e^{u_0+8\pi(n+1)G(z,0)} e^{-2
\re [c_0 z^{n+1}]}
\end{equation}
in view of \eqref{keyrelationH}.

\medskip \noindent Let $\sigma \in \mathcal{M}(\overline{\Omega})$ be a function which vanishes only at zero with multiplicity $n+1$. For $a \in \mathbb{C}$ small there exist $a_0,\dots,a_n$ so that $\{z \in \tilde \Omega: \, \sigma(z)=a \}=\{a_0,\dots,a_n\}$ (distinct points when $a \not=0$). For $a$ small the function
\begin{eqnarray}
\mathcal{H}_{a,\sigma}(z) &=&  \prod_j  (z-p_j)^{n_j} \hbox{exp}\left( 4\pi \sum_{k=0}^n H^*(z-a_k)-\frac{2\pi}{|\Omega|}z \overline{\sum_{k=0}^n a_k}-2\pi\sum_{j=1}^N H^*(z-p_j)\right. \label{Hasigma} \\
&&\left.-\frac{\pi}{2|\Omega|}\sum_{j=1}^N |p_j|^2
+\frac{\pi}{|\Omega|}z \overline{\sum_{j=1}^N p_j}\right)
\nonumber
\end{eqnarray}
is holomorphic in $\tilde \Omega$ with
\begin{equation} \label{keyrelation}
|\mathcal{H}_{a,\sigma}(z)|^2=\frac{1}{|z|^{2n}} e^{u_0+8\pi
\sum_{k=0}^n H(z-a_k)-\frac{2\pi}{|\Omega|} \sum_{k=0}^n |a_k|^2}
\qquad \hbox{in }\tilde \Omega
\end{equation}
in view of \eqref{balance}. The advantage in our construction of
$\mathcal{H}_{a,\sigma}$, which might be carried over in a
simpler and more direct way, is the holomorphic/anti-holomorphic
dependence in the $a_k$'s as well as in $z$, a crucial property as
we will see in Appendix A. When $a=0$, then $a_0=\dots=a_n=0$ and
$\mathcal{H}=\mathcal{H}_{0,\sigma}$.

\medskip \noindent Endowed with the norm $\|\sigma\|:=\| \frac{\sigma}{\sigma_0}\|_{\infty,\tilde \Omega}$, the set $\mathcal{M}'(\overline{\Omega})=\{ \sigma \in \mathcal{M}(\overline{\Omega}):\,\|\sigma\|<\infty\}$ is a Banach space, and let $\mathcal{B}_r$ be the closed ball centered at $\sigma_0$ and radius $r>0$, i.e.
\begin{equation} \label{setB}
\mathcal{B}_r=\bigg\{ \sigma \in
\mathcal{M}({\overline{\Omega}}):\:\Big\|
\frac{\sigma}{\sigma_0}-1\Big\|_{\infty,\tilde \Omega} \leq r
\bigg\}.
\end{equation}
For $a\not=0$ and $r$ small, the aim is to find a solution
$\sigma_a \in \mathcal{B}_r$ of
$$\sigma(z)= -\left[ \int^z \left(\frac{\sigma(w)-a}{\prod_{k=0}^n (w-a_k)} \frac{w^{n+1}}{\sigma(w)} \right)^2     \frac{\mathcal{H}_{a,\sigma}(w)}{w^{n+2}} e^{-c_{a,\sigma}w^{n+1}}dw\right]^{-1}$$
for a suitable coefficient $c_{a,\sigma}$. To be more precise,
letting
$$g_{a,\sigma}(z)=\frac{\sigma(z)-a}{\prod_{k=0}^{n}(z-a_k)}$$
for $|a|<\rho$ and $\sigma \in \mathcal{B}_r$, by Lemma
\ref{gomme} we have that $g_{a,\sigma} \in
\mathcal{M}(\overline{\Omega})$  never vanishes, and the problem
above gets re-written as
\begin{equation} \label{sigmaa}
\sigma(z)= -\left[ \int^z
\frac{g^2_{a,\sigma}(w)}{g^2_{0,\sigma}(w)}
\frac{\mathcal{H}_{a,\sigma}(w)}{w^{n+2}}
e^{-c_{a,\sigma}w^{n+1}}dw\right]^{-1}.
\end{equation}
The choice
\begin{equation} \label{ca}
c_{a,\sigma}=\frac{1}{(n+1)!}\frac{d^{n+1}}{dz^{n+1}}\left[
\frac{g^2_{a,\sigma}(z) g^2_{0,\sigma}(0)}{g^2_{a,\sigma}(0)
g^2_{0,\sigma}(z)}
\frac{\mathcal{H}_{a,\sigma}(z)}{\mathcal{H}_{a,\sigma}(0) }
\right](0)
\end{equation}
lets vanish the residue of the integrand function in
\eqref{sigmaa} making the R.H.S. well-defined. Since $\sigma_a \in
\mathcal{B}_r$, the function $\sigma_a$ vanishes only at zero with
multiplicity $n+1$, and satisfies
\begin{equation} \label{eq sigmaa}
|\sigma_a'(z)|^2=  |\sigma_a(z)-a|^4  \hbox{exp}\left(u_0+8\pi
\sum_{k=0}^n G(z,a_k)-\frac{2\pi}{|\Omega|} \sum_{k=0}^n
|a_k|^2-2\re [c_{a,\sigma_a}z^{n+1}]\right)
\end{equation}
in view of \eqref{keyrelation}. The resolution of problem
\eqref{sigmaa}-\eqref{ca} will be addressed in Appendix A.

\medskip\noindent We have the following expansion for $PU_{\de,a,\sigma}$ as
$\de\to0$:
\begin{lem}\label{expPU} There holds
\begin{eqnarray} \label{1138}
PU_{\de,a,\sigma}=U_{\de,a,\sigma}-\log (8 \delta^2)+4 \log |g_{a,\sigma}|+8\pi \sum_{k=0}^n H(z-a_k)+\Theta_{\delta,a,\sigma}+2 \delta^2 f_{a,\sigma}+O(\de^4)
\end{eqnarray}
in $C(\overline{\Omega})$, uniformly for $|a|< \rho$ and $\sigma \in \mathcal{B}_r$, where
$$\Theta_{\de,a,\sigma}=-\frac{1}{|\Omega|}\int_\Om \log {|\sigma(z)-a|^4\over
(\de^2+|\sigma(z)-a|^2)^2}$$ and $f_{a,\sigma}$ is defined in
\eqref{FaQ}. In particular, there holds
$$PU_{\de,a,\sigma}=8\pi \sum_{k=0}^n G(z,a_k)+\Theta_{\delta,a,\sigma}+2\delta^2 \lf(f_{a,\sigma}-{1\over |\sigma(z)-a|^2}\rg)+O(\delta^4)$$
in $C_{\text{loc}}(\overline{\Omega} \setminus\{0 \})$, uniformly for $|a|< \rho$ and $\sigma \in \mathcal{B}_r$.
\end{lem}
\begin{proof}
Define
$$r_{\de,a,\sigma}=PU_{\de,a,\sigma}-U_{\de,a,\sigma}+\log
(8 \delta^2)-4 \log |g_{a,\sigma}|-8\pi \sum_{k=0}^n H(z-a_k).$$ The function $U_{\de,a,\sigma}$ does satisfy $-\Delta U_{\de,a,\sigma}=|\sigma'(z)|^2 e^{U_{\de,a,\sigma}}$ just in $\Omega \setminus \{\hbox{poles of }\sigma\}$. At the same time, the function $-4\log |g_{a,\sigma}|$ is harmonic in $\Omega \setminus \{\hbox{poles of }\sigma\}$, and has exactly the same singular behavior of $U_{\de,a,\sigma}$ near each pole of $\sigma$. It means that
\begin{equation} \label{yth}
-\Delta \left[U_{\delta,a,\sigma}+4\log |g_{a,\sigma}|\right]=|\sigma'(z)|^2 e^{U_{\de,a,\sigma}}
\end{equation}
does hold in the whole $\Omega$. Since $\Delta H=\frac{1}{|\Omega|}$, by (\ref{lll}) and (\ref{yth}) we get that
\begin{eqnarray*}
-\Delta r_{\de,a,\sigma}= \frac{1}{|\Omega|}\left[ 8\pi(n+1)-\int_\Omega |\sigma'(z)|^2 e^{U_{\de,a,\sigma}}\right] .
\end{eqnarray*}
By the Green's representation formula we have that
\begin{eqnarray} \label{repr}
r_{\de,a,\sigma}(z)=\frac{1}{|\Omega|}\int_\Omega r_{\de,a,\sigma}+\int_{\partial \Omega}[\partial_\nu r_{\de,a,\sigma}(w) G(w,z)-r_{\de,a,\sigma}(w) \partial_\nu G(w,z)]ds(w),
\end{eqnarray}
where $\nu$ is the unit outward normal of $\partial \Omega$ and $ds(w)$ is the line integral element. Since as $\delta \to 0$ there holds
$$r_{\de,a,\sigma}(w)=PU_{\de,a,\sigma}(w)-8\pi \sum_{k=0}^n G(w,a_k) +2 \frac{\delta^2}{|\sigma(w)-a|^2}+O(\delta^4)$$
in $C^1(\partial \Omega)$ uniformly in $|a|< \rho$ and $\sigma \in \mathcal{B}_r$, by double-periodicity of $PU_{\de,a,\sigma}-8\pi \displaystyle \sum_{k=0}^n G(\cdot,a_k)$ we get that
\begin{eqnarray} \label{repr1}
\int_{\partial \Omega}[\partial_\nu r_{\de,a,\sigma}(w) G(w,z)-r_{\de,a,\sigma}(w) \partial_\nu G(w,z)]ds(w)=
2 \delta^2 f_{a,\sigma}(z) +O(\delta^4)
\end{eqnarray}
in $C(\bar \Omega)$, where
\begin{eqnarray}\label{FaQ}
f_{a,\sigma}(z)=\int_{\partial \Omega}\Big[\partial_\nu \frac{1}{|\sigma(w)-a|^2} G(w,z)-\frac{1}{|\sigma(w)-a|^2} \partial_\nu G(w,z)\Big]ds(w).\end{eqnarray}
Inserting \eqref{repr1} into \eqref{repr} we get that
\begin{eqnarray} \label{repr2}
r_{\de,a,\sigma}(z)=\Theta_{\de,a,\sigma}+2 \delta^2 f_{a,\sigma}(z) +O(\delta^4)
\end{eqnarray}
in $C(\overline{\Omega})$ uniformly in $|a|< \rho$ and $\sigma \in \mathcal{B}_r$, where
$$\Theta_{\de,a,\sigma}:=\frac{1}{|\Omega|}\int_\Omega r_{\de,a,\sigma}=-\frac{1}{|\Omega|}\int_\Om \log {|\sigma(z)-a|^4\over
(\de^2+|\sigma(z)-a|^2)^2}.$$
The estimate \eqref{repr2} yields to the desired expansion for $PU_{\delta,a,\sigma}$ as $\delta \to 0$. \qed \end{proof}
%As already observed, for all $ y \in B_\rho(0) \setminus \{0\}$ the set $\sigma^{-1}(y) \subset B_\eta(0)$ is composed by exactly $n+1$ elements. Through the change of variable $y=\sigma(z)$ in $\sigma^{-1}(B_\rho(0))$, we get that
%\begin{eqnarray*}
%\int_\Omega |\sigma'(z)|^2 e^{U_{\de,a,\sigma}}&=& (n+1) \int_{B_\rho(0)}
%\frac{8\delta^2}{(\delta^2+|y-a|^2)^2}+8\delta^2
%\int_{\Omega \setminus \sigma^{-1}(B_\rho(0))}
%\frac{|\sigma'(z)|^2}{|\sigma(z)-a|^4}+O(\delta^4)\\
%&=& 8\pi (n+1)-(n+1)\int_{\mathbb{R}^2 \setminus B_\rho(0)}
%\frac{8 \delta^2}{|y-a|^4}+8\delta^2
%\int_{\Omega \setminus \sigma^{-1}(B_\rho(0))}
%\frac{|\sigma'(z)|^2}{|\sigma(z)-a|^4}+O(\delta^4).
%\end{eqnarray*}

\medskip \noindent Letting $\sigma_a \in \mathcal{B}_r$ be the solution of \eqref{sigmaa}-\eqref{ca}, we build up the correct approximating function as $W=PU_{\de,a,\sigma_a}$. We need to control the approximation rate of $W$ for $\delta$ and $\e$ small enough, by estimating the error term
\begin{eqnarray}\label{R}
R&=&\Delta W+4\pi N\left(\frac{e^{u_0+W}}{\int_\Omega
e^{u_0+W}}-\frac{1}{|\Omega|}\right)\\
&&+ \frac{64 \pi^2N^2 \epsilon^2 \int_\Omega
e^{2u_0+2W}}{\lf(\int_\Omega e^{u_0+W}+\sqrt{(\int_\Omega e^{u_0+W})^2-16\pi
N\epsilon^2\int_\Omega
e^{2u_0+2W}}\rg)^2}\left(\frac{e^{u_0+W}}{\int_\Omega e^{u_0+W}}-\frac{
e^{2u_0+2W}}{\int_\Omega e^{2u_0+2W}}\right).\nonumber
\end{eqnarray}
In order to simplify the notations, we set $U_{\delta,a}=U_{\delta,a,\sigma_a}$, $c_a=c_{a,\sigma_a}$, $\Theta_{\delta,a}=\Theta_{\delta,a,\sigma_a}$, $f_a=f_{a,\sigma_a}$, and omit the
subscript $a$ in $\sigma_a$. We have the following crucial result.
\begin{thm}\label{estrr01550} Let $|a|<\frac{\rho}{2}$ and set
\begin{eqnarray} \label{rateeps}
\eta=\epsilon^2 \delta^{-\frac{2}{n+1}} \max\Big\{1,  \frac{|a|}{\delta}\Big\}^{\frac{2n}{n+1}}.
\end{eqnarray}
The following expansions do hold
\begin{eqnarray}
&&\Delta W+4\pi N\left( \frac{  e^{u_0+W}}{\int_\Omega
e^{u_0+W}}-\frac{1}{|\Omega|}\right) \nonumber\\
&&=|\sigma'(z)|^2 e^{U_{\delta,a}}
\left[\frac{e^{2\re[c_a z^{n+1}]}}{1+2 \re[c_a F_a(a)] + |c_a|^2  \re G_a(a)+\frac{1}{2} |c_a|^2 \Delta \re G_a(a)  \delta^2 \log \frac{1}{\delta}+\frac{\delta^2}{n+1} D_a}-1\right] \nonumber \\
&&\quad +|\sigma'(z)|^2 e^{U_{\delta,a}}O(\delta^2 |z| +\delta^2|a|^{\frac{1}{n+1}}+\delta^2|c_a|+\delta^{\frac{2n+3}{n+1}})+O(\delta^2)
\label{imp}
\end{eqnarray}
and
\begin{eqnarray}
&&\ds \frac{64 \pi^2 N^2 \epsilon^2 \int_\Omega
e^{2u_0+2W}}{(\int_\Omega e^{u_0+W}+\sqrt{(\int_\Omega e^{u_0+W})^2-16\pi
N\epsilon^2\int_\Omega e^{2u_0+2W}})^2}  \left(\frac{
e^{u_0+W}}{\int_\Omega e^{u_0+W}}-\frac{e^{2u_0+2W}}{\int_\Omega e^{2u_0+2W}}\right) \nonumber \\
&&= |\sigma'(z)|^2 e^{U_{\delta,a}}
\left[\frac{8(n+1)^2\epsilon^2}{\pi |\alpha_a|^{\frac{2}{n+1}}
\delta^{\frac{2}{n+1}}} E_{a,\delta}-\e^2 |\sigma'(z)|^2
e^{U_{\delta,a}}\right] \left[1+O(|c_a||z|^{n+1}+\eta)+o(1)
\right] \label{eps4}
\end{eqnarray}
as $\e, \delta \to 0$, where $\alpha_a$, $F_a$, $G_a$, $D_a$, $E_{a,\delta}$ are given in
\eqref{alpha0}, \eqref{FG}, \eqref{Da}, \eqref{Eadelta},
respectively.
\end{thm}
\begin{proof} Recall that \eqref{sigmaa} implies the validity of \eqref{eq sigmaa}, which, combined with Lemma \ref{expPU}, yields to the following crucial estimate:
\begin{eqnarray} \label{cinema}
W=U_{\delta,a}-\log (8 \delta^2)+\log |\sigma'(z)|^2-u_0+\frac{2\pi}{|\Omega|} \sum_{k=0}^n |a_k|^2+2\re [c_a z^{n+1}]+\Theta_{\delta,a}+2\delta^2 f_a+O(\delta^4)
\end{eqnarray}
in $C(\overline{\Omega})$ as $\delta \to 0$, uniformly for $|a|< \rho$. Since by Lemma \ref{gomme} $\sigma=q^{n+1}$ in $\sigma^{-1}(B_\rho(0))$, through the change of variables
$y=q(z)$ in $\sigma^{-1}(B_\rho(0))=q^{-1}(B_{\rho^{\frac{1}{n+1}}}(0))$, by  (\ref{cinema}) we have that
\begin{eqnarray}
&& \frac{8 \delta^2} {e^{\frac{2\pi}{|\Omega|} \sum_{k=0}^n |a_k|^2+\Theta_{\de,a}+2\de^2 f_a(0)}}\int_{\sigma^{-1}(B_{\rho}(0))} e^{u_0+W}
= \int_{q^{-1}(B_{\rho^{\frac{1}{n+1}}}(0))}
|\sigma'(z)|^2 e^{U_{\de,a}+2\re[c_a z^{n+1}] +O(\delta^2|z|+\delta^4)} \nonumber \\
&&=  \int_{B_{\rho^{\frac{1}{n+1}}}(0)} \frac{8(n+1)^2 \delta^2 |y|^{2n}}{(\delta^2+|y^{n+1}-a|^2)^2}
e^{2\re[c_a (q^{-1}(y))^{n+1}] +O(\delta^2 |y|+\delta^4)}. \label{1045}
\end{eqnarray}
Since $q^{-1}(y)\sim y $ at $y=0$, the following Taylor expansion does hold
\begin{equation} \label{keyexp}
e^{c_a(q^{-1}(y))^{n+1}}=1+c_a y^{n+1} \sum_{k=0}^{+\infty} \alpha_a^k y^k
\end{equation}
in $B_{\rho^{\frac{1}{n+1}}}(0)$, where the coefficients $\alpha_a^k$ depend on $a$ through $\sigma=\sigma_a$. In particular, we have that $\alpha_a:=\alpha_a^0$ takes the form
\begin{eqnarray} \label{alpha0}
\alpha_a=\displaystyle \lim_{z \to 0}\frac{z^{n+1}}{\sigma(z)} \not=0.
\end{eqnarray}
By \eqref{keyexp} we then deduce that
\begin{equation} \label{keyexp1}
e^{2\re[c_a (q^{-1}(y))^{n+1}]}=\big|e^{c_a (q^{-1}(y))^{n+1}}\big|^2=1+2\re\bigg[c_a y^{n+1} \sum_{k=0}^{+\infty} \alpha_a^k y^k\bigg]+|c_a|^2 |y|^{2n+2}\sum_{k,s=0}^{+\infty} \alpha_a^k \overline{\alpha}_a^s y^k \overline{y}^s.
\end{equation}
Since
$$\sum_{j=0}^n [e^{i\frac{2\pi}{n+1}j}]^k=\sum_{j=0}^n e^{i\frac{2\pi}{n+1}j}=0$$
for all integer $k\notin (n+1)\mathbb{N}$, by the change of variables $y \to  e^{i\frac{2\pi}{n+1}j} y$ we have that
\begin{eqnarray}
\int_{B_{\rho^{\frac{1}{n+1}}}(0)} \frac{|y|^m y^k}{(\delta^2+|y^{n+1}-a|^2)^2}
&=&\sum_{j=0}^n \int_{B_{\rho^{\frac{1}{n+1}}}(0)\cap C_j} \frac{|y|^m y^k}{(\delta^2+|y^{n+1}-a|^2)^2} \nonumber\\
&=& \int_{B_{\rho^{\frac{1}{n+1}}}(0)\cap C_0} \frac{|y|^m y^k}{(\delta^2+|y^{n+1}-a|^2)^2} \sum_{j=0}^n [e^{i\frac{2\pi}{n+1}j}]^k=0 \label{symmetryint}
\end{eqnarray}
for all $m\geq 0$ and integer $k \notin (n+1)\mathbb{N}$, where $C_j$ is the sector of the plane between the angles $e^{i\frac{2\pi}{n+1}j}$ and $e^{i\frac{2\pi}{n+1}(j+1)}$. Formula \eqref{symmetryint} tells us that many terms of the expansion \eqref{keyexp1} will give no contribution when inserted in an integral formula like \eqref{1045}. Using the notation $\dots$ to denote such terms, we can rewrite \eqref{keyexp1} as
\begin{eqnarray} \label{keyexp2}
e^{2\re[c_a (q^{-1}(y))^{n+1}]}&=&1+2\re\bigg[c_a \sum_{k=0}^{+\infty} \alpha_a^{k(n+1)} y^{(k+1)(n+1)}\bigg]+|c_a|^2 |y|^{2n+2} \sum_{k=0}^{+\infty} |\alpha_a^k|^2 |y|^{2k} \\
&&+2|c_a|^2 |y|^{2n+2}\re \bigg[\sum_{k=0}^{+\infty}\sum_{m=1}^{+\infty} \overline{\alpha}_a^k \alpha_a^{k+m(n+1)} |y|^{2k} y^{m(n+1)}\bigg]+\dots \nonumber
\end{eqnarray}
Setting
\begin{equation} \label{FG}
F_a(y)=\sum_{k=0}^{+\infty} \alpha_a^{k(n+1)} y^{k+1},\quad G_a(y)=|y|^2 \left[2 \sum_{k=0}^{+\infty}\sum_{m=1}^{+\infty} \overline{\alpha}_a^k \alpha_a^{k+m(n+1)} |y|^{\frac{2k}{n+1}} y^m+\sum_{k=0}^{+\infty} |\alpha_a^k|^2 |y|^{\frac{2k}{n+1}}  \right],
\end{equation}
through the change of variables $y \to y^{n+1}$ we can re-write \eqref{1045} as
\begin{eqnarray}
&&\hspace{-1.1cm} \frac{8 \delta^2}{(n+1) e^{\frac{2\pi}{|\Omega|} \sum_{k=0}^n |a_k|^2+\Theta_{\de,a}+2\de^2 f_a(0)}}\int_{\sigma^{-1}(B_\rho(0))} e^{u_0+W} \nonumber\\
&&\hspace{-1.1cm}= \int_{B_\rho(0)} \frac{8\delta^2}{(\delta^2+|y-a|^2)^2}
\left(1+\re[2 c_a F_a(y)+|c_a|^2 G_a(y)]+O(\delta^2 |y|^{\frac{1}{n+1}}+\delta^4)\right)\nonumber\\
&&\hspace{-1.1cm} = 8\pi-\int_{\mathbb{R}^2 \setminus B_\rho(0)} \frac{8\delta^2}{|y|^4}+
\int_{B_\rho(0)} \frac{8\delta^2}{(\delta^2+|y-a|^2)^2}
\re[2 c_a F_a(y)+|c_a|^2 G_a(y)]+O(\delta^2|a|^{\frac{1}{n+1}}+\delta^{\frac{2n+3}{n+1}}). \label{1046}
\end{eqnarray}
Since $|a|<\frac{\rho}{2}$ and $F$ is an holomorphic function in $B_{\frac{\rho}{2}}(a) \subset B_\rho(0)$, we can expand $F_a$ in a power series around $y=a$:
\begin{eqnarray} \label{expF}
F_a(y)=\sum_{k=0}^\infty  \frac{F_a^{(k)}(a)}{k!} (y-a)^k,
\end{eqnarray}
and then get
\begin{eqnarray}
2 \int_{B_\rho(0)} \frac{8\delta^2}{(\delta^2+|y-a|^2)^2} \re[c_a F_a(y)]&=&
2 \int_{B_{\frac{\rho}{2}}(a)} \frac{8\delta^2}{(\delta^2+|y-a|^2)^2} \re[c_a F_a(y)]+O(\delta^2|c_a|)\nonumber\\
&=& 16\pi \re[c_a F_a(a)] +O(\delta^2|c_a|)
\label{1047}
\end{eqnarray}
in view of
$$\int_{B_{\frac{\rho}{2}}(a)} \frac{(y-a)^k}{(\delta^2+|y-a|^2)^2}=0$$
for all integer $k\geq 1$. The map $\re G_a$ is just $C^{2+\frac{2}{n+1}}(B_\rho(0))$ and can be expanded up to second order in $y=a$:
\begin{eqnarray}\label{expG}
\re G_a(y)=\re G_a(a)+\langle\nabla \re G_a(a),y-a\rangle+\frac{1}{2}\langle D^2 \re G_a(a) (y-a),y-a\rangle+O(|y-a|^{\frac{2(n+2)}{n+1}})
\end{eqnarray}
for $y \in B_{\frac{\rho}{2}}(a)$, yielding to
\begin{eqnarray}
&& |c_a|^2 \int_{B_\rho(0)} \frac{8\delta^2}{(\delta^2+|y-a|^2)^2} \re G_a(y)= |c_a|^2 \int_{B_{\frac{\rho}{2}}(a)} \frac{8\delta^2}{(\delta^2+|y-a|^2)^2} \re G_a(y)+O(\delta^2|c_a|^2) \nonumber \\
&&=8\pi |c_a|^2  \re G_a(a)+\frac{|c_a|^2}{4} \Delta \re G_a(a) \int_{B_{\frac{\rho}{2}}(a)} \frac{8\delta^2}{(\delta^2+|y-a|^2)^2} |y-a|^2+O(\delta^2|c_a|^2) \nonumber\\
&&=8\pi |c_a|^2  \re G_a(a)+4\pi |c_a|^2 \Delta \re G_a(a) \delta^2 \log \frac{1}{\delta} +O(\delta^2|c_a|^2)
\label{1048}
\end{eqnarray}
in view of
\begin{eqnarray*}
&&\int_{B_{\frac{\rho}{2}}(a)} \frac{(y-a)_1}{(\delta^2+|y-a|^2)^2}=\int_{B_{\frac{\rho}{2}}(a)} \frac{(y-a)_2}{(\delta^2+|y-a|^2)^2}=
\int_{B_{\frac{\rho}{2}}(a)} \frac{(y-a)_1(y-a_2)}{(\delta^2+|y-a|^2)^2}=0\\
&&\int_{B_{\frac{\rho}{2}} (a)} \frac{(y-a)_1^2}{(\delta^2+|y-a|^2)^2}=\int_{B_{\frac{\rho}{2}} (a)} \frac{(y-a)_2^2}{(\delta^2+|y-a|^2)^2}=
\frac{1}{2}\int_{B_{\frac{\rho}{2}}(a)} \frac{|y-a|^2}{(\delta^2+|y-a|^2)^2}.
\end{eqnarray*}
By inserting \eqref{1047}, \eqref{1048} into \eqref{1046} we get that
\begin{eqnarray}
&&\frac{8\delta^2}{(n+1)  e^{\frac{2\pi}{|\Omega|} \sum_{k=0}^n |a_k|^2+\Theta_{\de,a}+2\de^2 f_a(0)}}\int_{\sigma^{-1}(B_\rho(0))} e^{u_0+W}\nonumber\\
&&= 8\pi-\int_{\mathbb{R}^2 \setminus B_\rho(0)} \frac{8\delta^2}{|y|^4}
+16\pi \re[c_a F_a(a)] +8\pi |c_a|^2  \re G_a(a)+4 \pi |c_a|^2 \Delta \re G_a(a)  \delta^2 \log \frac{1}{\delta} \nonumber \\
&&+O(\delta^2|a|^{\frac{1}{n+1}}+\delta^2|c_a|+\delta^{\frac{2n+3}{n+1}}).\label{1049}
\end{eqnarray}
By Lemma \ref{expPU}, \eqref{1049} and Lemma \ref{gomme} we get that
\begin{eqnarray}
&&\frac{ \delta^2}{\pi (n+1)  e^{\frac{2\pi}{|\Omega|} \sum_{k=0}^n |a_k|^2+\Theta_{\delta,a}+2\delta^2f_{a}(0)}}\int_\Omega e^{u_0+W}  = 1+2 \re[c_a F_a(a)] + |c_a|^2  \re G_a(a) \nonumber \\
&&+\frac{1}{2} |c_a|^2 \Delta \re G_a(a)  \delta^2 \log \frac{1}{\delta}+\frac{\delta^2}{n+1} D_a+O(\delta^2|a|^{\frac{1}{n+1}}+\delta^2|c_a|+\delta^{\frac{2n+3}{n+1}}), \label{const}
\end{eqnarray}
where
\begin{equation} \label{Da}
\pi D_a=\int_{\Omega \setminus \sigma^{-1}(B_\rho(0))} e^{u_0+8\pi \sum_{k=0}^n G(z,a_k)-\frac{2\pi}{|\Omega|} \sum_{k=0}^n |a_k|^2}   -\int_{\mathbb{R}^2 \setminus B_\rho(0)} \frac{n+1}{|y|^4}.
\end{equation}
In view of (\ref{balance}) and $\int_\Omega |\sigma'(z)|^2 e^{U_{\delta,a}}=8\pi(n+1)+O(\delta^2)$, by (\ref{cinema}) and (\ref{const}) we have that
\begin{eqnarray*}
&&\Delta W+4\pi N\left( \frac{  e^{u_0+W}}{\int_\Omega
e^{u_0+W}}-\frac{1}{|\Omega|}\right) \\
&&=|\sigma'(z)|^2 e^{U_{\delta,a}}
\left[4\pi N \frac{e^{2\re[c_a z^{n+1}]+O(\delta^2 |z|+\delta^4)}
}{8 \delta^2  e^{-\frac{2\pi}{|\Omega|} \sum_{k=0}^n |a_k|^2-\Theta_{\delta,a}-2 \delta^2 f_a
(0)} \int_\Omega e^{u_0+W}}-1 \right] +\frac{1}{|\Omega|}\left(\int_\Omega |\sigma'(z)|^2
e^{U_{\delta,a}}-4\pi N \right)\\
&&=|\sigma'(z)|^2 e^{U_{\delta,a}}
\left[\frac{e^{2\re[c_a z^{n+1}]}}{1+2 \re[c_a F_a(a)] + |c_a|^2  \re G_a(a)+\frac{1}{2} |c_a|^2 \Delta \re G_a(a)  \delta^2 \log \frac{1}{\delta}+\frac{\delta^2}{n+1} D_a}-1\right] \\
&& +|\sigma'(z)|^2 e^{U_{\delta,a}}O(\delta^2 |z| +\delta^2|a|^{\frac{1}{n+1}}+\delta^2|c_a|+\delta^{\frac{2n+3}{n+1}})+O(\delta^2)
\end{eqnarray*}
as $\de  \to 0$, yielding to the validity of \eqref{imp}.

\medskip \noindent Introducing the notation $B(w)=16\pi N (\int_\Om
e^{2u_0+2w})(\int_\Om e^{u_0+w})^{-2}$, we can write the following
expansion
\begin{equation} \label{BWW}
\frac{16 \pi N \int_\Omega e^{2u_0+2W}}{(\int_\Omega
e^{u_0+W}+\sqrt{(\int_\Omega e^{u_0+W})^2-16\pi N\epsilon^2\int_\Omega
e^{2u_0+2W}})^2}={B(W) \over 4}+O(\e^2 B^2(W)).
\end{equation}
Arguing as for (\ref{const}), the change of variables $y=\sigma(z)$ yields to
\begin{eqnarray}
&& {64 \delta^{4+\frac{2}{n+1}}\over
  e^{\frac{4\pi}{|\Omega|} \sum_{k=0}^n |a_k|^2+2\Theta_{\delta,a}}}\int_\Omega e^{2u_0+2W}=\delta^{\frac{2}{n+1}}\int_{\sigma^{-1} (B_\rho(0))} |\sigma'(z)|^4 e^{2U_{\delta,a} +O(|c_a||z|^{n+1}+\delta^2)}+O(\delta^{4+\frac{2}{n+1}})\nonumber \\
&&= 64(n+1)^3 |\alpha_a|^{-\frac{2}{n+1}} \int_{B_\rho (0)} \frac{\delta^{4+\frac{2}{n+1}} |y|^{\frac{2n}{n+1}}}{(\delta^2+|y-a|^2)^4}\left(1+O(|c_a||y|+\delta^2+|y|^{\frac{1}{n+1}})\right)+O(\delta^{4+\frac{2}{n+1}}) \nonumber \\
&&= 64(n+1)^3 |\alpha_a|^{-\frac{2}{n+1}} \int_{B_\rho (0)} \frac{\delta^{4+\frac{2}{n+1}} |y+a|^{\frac{2n}{n+1}}}{(\delta^2+|y|^2)^4}\left(1+O(\delta^2+|y|^{\frac{1}{n+1}}+|a|^{\frac{1}{n+1}})\right)+O(\delta^{4+\frac{2}{n+1}}) \label{const1}
\end{eqnarray}
in view of
\begin{eqnarray} \label{1852}
|\sigma'(z)|^2=(n+1)^2 |\alpha_a|^{-2} |z|^{2n}(1+O(|z|))=(n+1)^2 |\alpha_a|^{-\frac{2}{n+1}} |\sigma(z)|^{\frac{2n}{n+1}}(1+O(|\sigma(z)|^{\frac{1}{n+1}})),
\end{eqnarray}
where $\alpha_a$ is given by \eqref{alpha0}. We have that
$$\int_{B_\rho (0)} \frac{\delta^{4+\frac{2}{n+1}} |y+a|^{\frac{2n}{n+1}}}{(\delta^2+|y|^2)^4}=
\int_{\mathbb{R}^2} \frac{|y+\frac{a}{\delta} |^{\frac{2n}{n+1}}}{(1+|y|^2)^4} +O(\de^{4+\frac{2}{n+1}}) $$
if $|a|=O(\delta)$ and
$$\int_{B_\rho (0)} \frac{\delta^{4+\frac{2}{n+1}} |y+a|^{\frac{2n}{n+1}}}{(\delta^2+|y|^2)^4}=
\Big(\frac{|a|}{\delta}\Big)^{\frac{2n}{n+1}} \int_{\mathbb{R}^2} \frac{1}{(1+|y|^2)^4}\left[1 +O\Big(\frac{\delta}{|a|}+\delta^6\Big)\right]$$
if $|a|>>\delta$, where in the latter we have used the inequality:
$$|y+a|^{\frac{2n}{n+1}}=|a|^{\frac{2n}{n+1}}+O(|a|^{\frac{n-1}{n+1}}|y|+|y|^{\frac{2n}{n+1}}).$$
Setting
\begin{eqnarray} \label{Eadelta}
E_{a,\delta}:= \left\{\begin{array}{ll}
\ds\int_{\mathbb{R}^2} \frac{|y+\frac{a}{\delta} |^{\frac{2n}{n+1}}}{(1+|y|^2)^4} &\hbox{if }|a|=O(\delta)\\ \\
\ds\frac{\pi}{3} \Big(\frac{|a|}{\delta}\Big)^{\frac{2n}{n+1}}  &\hbox{if }|a|>>\delta, \end{array} \right.
\end{eqnarray}
by (\ref{const1}) we get that
\begin{eqnarray}
{64 \delta^{4+\frac{2}{n+1}}\over
  e^{\frac{4\pi}{|\Omega|} \sum_{k=0}^n |a_k|^2+2\Theta_{\delta,a}}}\int_\Omega e^{2u_0+2W}= 64(n+1)^3 |\alpha_a|^{-\frac{2}{n+1}} (1+o(1)) E_{a,\delta}.\label{const1517}
\end{eqnarray}
Since by a combination of (\ref{const}) and (\ref{const1517}) for $B(W)$ we have
that
\begin{equation} \label{BW}
B(W)=32 \frac{(n+1)^2}{ \pi \delta^{\frac{2}{n+1}}}|\alpha_a|^{-\frac{2}{n+1}}(1+o(1))E_{a,\delta}
\end{equation}
in view of (\ref{balance}), by \eqref{BWW} and \eqref{BW} we get that
\begin{eqnarray} \label{eps0}
\frac{16 \pi N \int_\Omega e^{2u_0+2W}}{(\int_\Omega
e^{u_0+W}+\sqrt{(\int_\Omega e^{u_0+W})^2-16\pi N\epsilon^2\int_\Omega
e^{2u_0+2W}})^2} =8 \frac{(n+1)^2}{ \pi \delta^{\frac{2}{n+1}}}|\alpha_a|^{-\frac{2}{n+1}} (1+o(1)+O(\eta)) E_{a,\delta},
\end{eqnarray}
where $\eta$ is given by \eqref{rateeps}. As we have already seen in deriving (\ref{imp}), by (\ref{cinema}) we have that
\begin{eqnarray}
4\pi N\frac{  e^{u_0+W}}{\int_\Omega
e^{u_0+W}}=|\sigma'(z)|^2 e^{U_{\delta,a}}
\left[1+O(|c_a| |z|^{n+1})+O(|c_a||a|+\delta^2 |\log \delta| ) \right],
\label{eps1}
\end{eqnarray}
and in a similar way one can show that
\begin{eqnarray}
{64(n+1)^3\over \delta^{\frac{2}{n+1}}} |\alpha_a|^{-\frac{2}{n+1}} \frac{e^{2u_0+2W} }{\int_\Omega e^{2u_0+2W} }
E_{a,\delta}= |\sigma'(z)|^4 e^{2U_{\delta,a}}\left[1+ O(|c_a||z|^{n+1})+o(1)\right]
\label{eps2}
\end{eqnarray}
in view of (\ref{const1517}). In conclusion, by (\ref{eps0})-(\ref{eps2})  we have for the
$\epsilon^2-$term in $R$ that
\begin{eqnarray*}
&&\ds \frac{64 \pi^2 N^2 \epsilon^2 \int_\Omega
e^{2u_0+2W}}{(\int_\Omega e^{u_0+W}+\sqrt{(\int_\Omega e^{u_0+W})^2-16\pi
N\epsilon^2\int_\Omega e^{2u_0+2W}})^2}  \left(\frac{
e^{u_0+W}}{\int_\Omega e^{u_0+W}}-\frac{e^{2u_0+2W}}{\int_\Omega e^{2u_0+2W}}\right)  \\
&&= |\sigma'(z)|^2 e^{U_{\delta,a}}
\left[\frac{8(n+1)^2\epsilon^2}{\pi |\alpha_a|^{\frac{2}{n+1}} \delta^{\frac{2}{n+1}}}
E_{a,\delta}-\e^2
|\sigma'(z)|^2 e^{U_{\delta,a}}\right] \left[1+O(|c_a||z|^{n+1}+\eta)+o(1) \right]
\end{eqnarray*}
in view of (\ref{balance}), yielding to the validity of \eqref{eps4}. This completes the proof. \qed
\end{proof}

\medskip \noindent Let us introduce the following weighted norm
\begin{equation}\label{wn}
\| h \|_*=\sup_{z\in \Om}
\frac{(\delta^2+|\sigma(z)-a|^2)^{1+\frac{\gamma}{2}}}{\delta^{\gamma} (|\sigma'(z)|^2+\delta^{\frac{2n}{n+1}})}\;  |h(z)|
\end{equation}
for any $h\in L^\infty(\Om)$, where $0<\gamma<1$ is a small fixed
constant. We have that
\begin{cor}\label{estrr0cor}
There exist positive constants $\delta_0$, $\e_0$ and $C_0$ such
that
\begin{equation}\label{ere}
\|R\|_*\le C_0 \left(\delta |c_a|+\delta^{2-\gamma} +\delta^{\frac{2}{n+1}-\gamma} |a|^{2+\gamma}+|c_a||a|^{\frac{n+2}{n+1}}+\eta+\eta^2\right)
\end{equation}
for any $\delta \in (0,\delta_0)$ and $\e\in(0,\e_0)$, where
$\eta$ is given by \eqref{rateeps}.
\end{cor}
\begin{proof}
Since
\begin{eqnarray*}
&&\frac{e^{2\re[c_a z^{n+1}]}}{1+2 \re[c_a F_a(a)] + |c_a|^2  \re G_a(a)+\frac{1}{2} |c_a|^2 \Delta \re G_a(a)  \delta^2 \log \frac{1}{\delta}+\frac{\delta^2}{n+1} D_a}-1\\
&&=\frac{e^{2\re[c_a z^{n+1}]}-1}{1+2 \re[c_a F_a(a)] + |c_a|^2  \re G_a(a)+\frac{1}{2} |c_a|^2 \Delta \re G_a(a)  \delta^2 \log \frac{1}{\delta}+\frac{\delta^2}{n+1} D_a}-2 \re[c_a F_a(a)]\\
&&+O(|c_a|^2|a|^2+\delta^2 |\log \delta|)=2\re[c_a (z^{n+1}-\alpha_a a)]+O(|c_a|^2 |z|^{2n+2}+ |c_a| |a|^2+\delta^2 |\log \delta|)\\
&&=2\re[\alpha_a c_a (\sigma(z)-a)]+O(|c_a| |z|^{n+2}+ |c_a| |a|^2+\delta^2 |\log \delta|),
\end{eqnarray*}
by Theorem \ref{estrr01550} we deduce that
\begin{eqnarray*}
R = |\sigma'(z)|^2 e^{U_{\delta,a}}O\left(|c_a||\sigma(z)-a|+|c_a||z|^{n+2}+ |c_a||a|^2+\delta^2|\log \delta|+\eta +\eta^2 \right) +\e^2 |\sigma'(z)|^4 e^{2U_{\delta,a}}(1+O(\eta))+ O(\delta^2)
\end{eqnarray*}
as $\delta \to 0$, where $\eta$ is given in \eqref{rateeps}. In view of the estimates $|z|=O(|\sigma(z)|^{\frac{1}{n+1}})$ and $|\sigma'(z)|^2=O(|\sigma(z)|^{\frac{2n}{n+1}})$ near $0$, by setting $y=\sigma(z)$ in $\sigma^{-1}(B_\rho(0))$ we get that
\begin{eqnarray*}
\|R\|_*&=&O\left(\sup_{y \in B_\rho(0)}  \frac{\delta^{2-\gamma}}{(\delta^2+|y-a|^2)^{1-\frac{\gamma}{2}}}\left[|c_a||y-a|+|c_a||y|^{\frac{n+2}{n+1}} +|c_a||a|^2+\delta^2 |\log \delta|+\eta +\eta^2 \right]\right)\\
&&+O\left( \sup_{y \in B_\rho(0)}  \frac{\epsilon^2 \delta^{4-\gamma}|y|^{\frac{2n}{n+1}}}{(\delta^2+|y-a|^2)^{3-\frac{\gamma}{2}}} [1+O(\eta)]\right)+O\left( \sup_{y \in B_\rho(0)} \frac{\delta^{2-\gamma} (\delta^2+|y-a|^2)^{1+\gamma/2}}{(|y|^{\frac{2n}{n+1}}+\delta^{\frac{2n}{n+1}})}  \right)  +O(\delta^{2-\gamma})\\
&=& O\left(\sup_{y \in B_{2\rho/\delta}(0)}  \frac{1}{(1+|y|^2)^{1-\frac{\gamma}{2}}}\left[\delta |c_a||y|+\delta^{\frac{n+2}{n+1}} |c_a| |y|^{\frac{n+2}{n+1}}+|c_a||a|^{\frac{n+2}{n+1}}+\delta^2 |\log \delta|+\eta +\eta^2 \right]\right)\\
&&+O\left( \sup_{y \in B_{2\rho/\delta}(0)}  \frac{\epsilon^2 \delta^{-2}(\delta^{\frac{2n}{n+1}}|y|^{\frac{2n}{n+1}}+|a|^{\frac{2n}{n+1}})}{(1+|y|^2)^{3-\frac{\gamma}{2}}}[1+O(\eta)] \right)\\
&&+O\left( \sup_{y \in B_{\rho/\delta}(0)} \frac{\delta^{\frac{2}{n+1}-\gamma} (\delta^{2+\gamma}+|a|^{2+\gamma}+\delta^{2+\gamma} |y|^{2+\gamma})}{(|y|^{\frac{2n}{n+1}}+1)}  \right)+O(\delta^{2-\gamma})\\
&=& O\left(\delta |c_a|+\delta^{2-\gamma} +\delta^{\frac{2}{n+1}-\gamma} |a|^{2+\gamma}+|c_a||a|^{\frac{n+2}{n+1}}+\eta+\eta^2 \right)
\end{eqnarray*}
as claimed. \qed \end{proof}

%%%%%%%%%%%%%%%%%%%%%%%%%%%%%%%%%%%%%%%%%%%%%%%%%%%%%%%%%%%%%%%%%%%%%%%%%%%%%%%%%%%%%%%%%%%%%%%%%%%%%%%%%%%%%%%%%%%
%%%%%%%%%%%%%%%%%%%%%%%%%%%%%%%%%%%%%%%%%%%%%%%%%%%%%%%%%%%%%%%%%%%%%%%%%%%%%%%%%%%%%%%%%%%%%%%%%%%%%%%%%%%%%%%%%%%

\section{The reduced equations}\label{reduced}
As we will discuss precisely in the next section, it will be
crucial to study the system $\int_\Omega R  \, PZ_0=0$ and $\int_\Omega R\,  PZ=0$, where
$PZ_0$ and $PZ$ are the unique solutions with zero average of
$\Delta PZ_{0} =\Delta Z_{0}-\frac{1}{|\Om|}\int_\Om \Delta Z_0$
and $\Delta PZ =\Delta Z-\frac{1}{|\Om|}\int_\Om \Delta Z$ in
$\Omega$. Here, the functions $Z_0$ and $Z$ are defined as
follows:
$$Z_0(z)=\frac{\delta^2-|\sigma(z)-a|^2}{\delta^2+|\sigma(z)-a|^2}\quad\text{and}\quad
Z(z)= \frac{\delta(\sigma(z)-a)}{\delta^2+|\sigma(z)-a|^2},$$
and are (not doubly-periodic) solutions of $-\Delta \phi =|\sigma'(z)|^2 e^{U_{\delta,a,\sigma}} \phi$ in $\Omega$. Through the changes of variable $y=\sigma(z)$ and $y \to \frac{y-a}{\delta}$ notice that
\begin{eqnarray}
\int_\Omega \Delta Z_0&=&-\int_{\sigma^{-1}(B_\rho(0))}|\sigma'(z)|^2e^{U_{\delta,a,\sigma}} Z_0+O(\delta^2) =-8(n+1) \delta^2 \int_{B_\rho(0)}\frac{\delta^2-|y-a|^2}{(\delta^2+|y-a|^2)^3} +O(\delta^2)\nonumber \\
&=&-8(n+1) \int_{B_{\rho/\delta}(0)}\frac{1-|y|^2}{(1+|y|^2)^3} +O(\delta^2)=O(\delta^2) \label{deltaZ0}
\end{eqnarray}
and
\begin{eqnarray}
\int_\Omega \Delta Z&=&-\int_{\sigma^{-1}(B_\rho(0))}|\sigma'(z)|^2 e^{U_{\delta,a,\sigma}} Z+O(\delta^3) =-8(n+1) \delta^3 \int_{B_\rho(0)}\frac{y-a}{(\delta^2+|y-a|^2)^3} +O(\delta^3) \nonumber \\
&=&-8(n+1) \int_{B_{\rho/\delta}(0)}\frac{y}{(1+|y|^2)^3} +O(\delta^3)=O(\delta^3) \label{deltaZ}
\end{eqnarray}
in view of
$$\int_{\mathbb{R}^2} \frac{1-|y|^2}{(1+|y|^2)^3}=0\,,\quad \int_{\mathbb{R}^2} \frac{y}{(1+|y|^2)^3}=0.$$
By \eqref{deltaZ0}-\eqref{deltaZ} the following expansions, useful in the sequel, are easily deduced:
\begin{equation}\label{pzij}
PZ_{0}=Z_{0} - {1\over|\Om|}\int_\Om Z_{0}+O(\de^2)\:,\qquad
PZ=Z-{1\over|\Om|}\int_\Om Z+O(\de)
\end{equation}
in $C(\overline{\Omega})$, uniformly in $|a|< \rho$ and $\sigma \in \mathcal{B}_r$.

\medskip \noindent Notice that up to now there is no relation between $a$ and $\delta$. However, as we will show in Remarks \ref{remark1} and \ref{remark2}, the range $|a|>>\delta$ is not compatible with solving simultaneously $\int_\Omega R  \, PZ_0=0$ and $\int_\Omega R\,  PZ=0$. Hence, we shall restrict our attention to the case $a=O(\delta)$ in next sections, so that, we can assume that $\eta=\epsilon^2 \delta^{-\frac{2}{n+1}}$ in \eqref{rateeps} and $E_{a,\delta}=\int_{\mathbb{R}^2} \frac{|y+\frac{a}{\delta} |^{\frac{2n}{n+1}}}{(1+|y|^2)^4}$ in \eqref{Eadelta}. We have that
\begin{prop} \label{reducedequations}
Assume $|a|\leq C_0 \delta$ for some $C_0>0$. The following
expansions do hold as $\de,\eta\to 0$
\begin{eqnarray} \label{solve1b}
\int_\Omega R \, PZ_0&=&- 16 \pi (n+1) |\alpha_a|^2 |c_a|^2 \delta^2 \log \frac{1}{\delta}-8\pi \delta^2 D_a +64(n+1)^3 |\alpha_a|^{-\frac{2}{n+1}}  \eta \int_{\mathbb{R}^2} \frac{(|y|^2-1)|y+\frac{a}{\delta} |^{\frac{2n}{n+1}}}{(1+|y|^2)^5} \nonumber \\
&&+o(\delta^2+\eta)+O(\de^2 |c_a|+|a|^{\frac{1}{n+1}}\delta^2 |\log \delta|+\eta^2),
\end{eqnarray}
and
\begin{eqnarray}  \label{solve2b}
\int_\Omega R \, PZ = 4 \pi (n+1) \delta \overline{\alpha_a c_a}
-64(n+1)^3 |\alpha_a|^{-\frac{2}{n+1}} \eta \int_{\mathbb{R}^2} {|y+\frac{a}{\delta}|^{2n\over n+1} y
\over(1+|y|^2)^5}+o(\delta|c_a|+\delta|a|+\eta+\delta^2)+O(\eta^2),
\end{eqnarray}
where $\eta=\epsilon^2 \delta^{-\frac{2}{n+1}}$ and $c_a=c_{a,\sigma_a}$, $\alpha_a$, $D_a$ are given by \eqref{ca}, \eqref{alpha0}, \eqref{Da}, respectively.
\end{prop}
\begin{proof} Through the changes of variable $y=q(z)$ in $\sigma^{-1}(B_\rho(0))$, $ y \to y^{n+1}$ and $y \to \frac{y-a}{\delta}$ we get that
\begin{eqnarray}
&&\int_\Omega \frac{\delta^{\gamma} (|\sigma'(z)|^2+\delta^{\frac{2n}{n+1}})}{(\delta^2+|\sigma(z)-a|^2)^{1+\frac{\gamma}{2}}}=\int_{\sigma^{-1}(B_\rho(0))} \frac{\delta^{\gamma} (|\sigma'(z)|^2+\delta^{\frac{2n}{n+1}})}{(\delta^2+|\sigma(z)-a|^2)^{1+\frac{\gamma}{2}}}+O(\delta^\gamma) \label{1458} \\
&&=O\left(\int_{B_{\rho^{\frac{1}{n+1}}}(0)} \frac{\delta^{\gamma} (|y|^{2n}+\delta^{\frac{2n}{n+1}})}{(\delta^2+|y^{n+1}-a|^2)^{1+\frac{\gamma}{2}}}\right)+O(\delta^\gamma)=O\left(\int_{B_\rho(0)} \frac{\delta^{\gamma} (1+\delta^{\frac{2n}{n+1}} |y|^{-\frac{2n}{n+1}})}{(\delta^2+|y-a|^2)^{1+\frac{\gamma}{2}}}\right)+O(\delta^\gamma) \nonumber \\
&&=O\left(\int_{B_{\rho/\delta}(0)} \frac{1+ |y+\frac{a}{\delta}|^{-\frac{2n}{n+1}}}{(1+|y|^2)^{1+\frac{\gamma}{2}}}\right)+O(\delta^\gamma)=O(1) \nonumber
\end{eqnarray}
in view of
$$\int_{B_{\rho/\delta}(0)} \frac{|y+\frac{a}{\delta}|^{-\frac{2n}{n+1}}}{(1+|y|^2)^{1+\frac{\gamma}{2}}}
\leq \int_{B_1(0)} |y|^{-\frac{2n}{n+1}}+\int_{\mathbb{R}^2} \frac{1}{(1+|y|^2)^{1+\frac{\gamma}{2}}}<+\infty.$$
Hence, by Corollary \ref{estrr0cor} we get that
\begin{eqnarray} \label{int|R|}
\int_\Omega |R|= O\left(\delta |c_a|+\delta^{2-\gamma} +\delta^{\frac{2}{n+1}-\gamma} |a|^{2+\gamma}+|c_a||a|^{\frac{n+2}{n+1}}+\eta+\eta^2 \right).
\end{eqnarray}
By \eqref{pzij} and \eqref{int|R|} we deduce that
\begin{eqnarray} \label{zerotermbisb}
\int_\Omega R \, PZ_0=\int_\Omega  R (Z_0+1)+o(\delta^2)+O(\eta \delta^2+\eta^2 \delta^2)
\end{eqnarray}
in view of $\int_\Om R=0$. Since by H\"older inequality
\begin{eqnarray*}
\int_\Omega |Z_0+1|&=& \int_{\sigma^{-1}(B_\rho(0))} \frac{2
\de^2}{\de^2+|\sigma(z)-a|^2}+O(\de^2)=
O\bigg(\int_{B_\rho(0)} |y|^{-{2n \over n+1}} \frac{\de^2}{\de^2+|y-a|^2}\bigg)+O(\de^2)\\
&=& O\bigg(\de^{1\over n+1} \int_{B_\rho(0)}
\frac{1}{|y|^{2n \over n+1} |y-a|^{1 \over
n+1}}\bigg)+O(\de^2)\\
&=&O\bigg( \de^{1\over n+1} \bigg[\int_{B_\rho(0)}
\frac{1}{|y|^{2n+1 \over n+1}} \bigg]^{2n \over 2n+1} \bigg[\int_{B_\rho(0)} \frac{1}{|y-a|^{2n+1 \over
n+1}} \bigg]^{1 \over 2n+1} \bigg)+O(\delta^2)=O(\de^{1\over n+1}),
\end{eqnarray*}
by (\ref{imp}) we have that
\begin{eqnarray} \label{firsttermbis}
&&\hspace{-0.5cm}  \int_\Om (Z_0+1)\lf[\lap W+4\pi N\lf({e^{u_0+W} \over \int_\Om
e^{u_0+W}}-{1\over
|\Om|}\rg)\rg] \\
&&\hspace{-0.5cm} =\int_{\sigma^{-1}(B_\rho(0))}|\sigma'(z)|^2 e^{U_{\de,a}}(Z_0+1)
\left[\frac{e^{2\re[c_a z^{n+1}]}}{1+2 \re[c_a F_a(a)] + |c_a|^2  \re G_a(a)+\frac{1}{2} |c_a|^2 \Delta \re G_a(a)  \delta^2 \log \frac{1}{\delta}+\frac{\delta^2}{n+1} D_a}-1\right] \nonumber \\
&&\hspace{-0.5cm} +O(\de^2 |c_a|)+o(\delta^2)\nonumber\\
&&\hspace{-0.5cm}  = \int_{B_{\rho^{\frac{1}{n+1}}}(0)} \frac{16(n+1)^2  \delta^4 |y|^{2n}}{(\delta^2+|y^{n+1}-a|^2)^3}
\left[\frac{e^{2\re[c_a (q^{-1}(y))^{n+1}]}}{1+2 \re[c_a F_a(a)] + |c_a|^2  \re G_a(a)+\frac{1}{2} |c_a|^2 \Delta \re G_a(a)  \delta^2 \log \frac{1}{\delta}+\frac{\delta^2}{n+1} D_a}-1 \right] \nonumber \\
&&\hspace{-0.5cm}  +O(\de^2 |c_a|)+o(\delta^2). \nonumber
\end{eqnarray}
We have that the expansion \eqref{keyexp2} still holds in this context, where the notation $\dots$ stands for terms
that give no contribution in the integral term of \eqref{firsttermbis} in view of the analogous of formula \eqref{symmetryint}:
\begin{eqnarray} \label{symmetryintbis}
\int_{B_{\rho^{\frac{1}{n+1}}}(0)} \frac{|y|^m y^k}{(\delta^2+|y^{n+1}-a|^2)^3}=0
\end{eqnarray}
for all $m\geq 0$ and integer $k \notin (n+1)\mathbb{N}$. Hence, through the changes of variables $y \to y^{n+1}$ and $y \to \frac{y-a}{\delta}$, by the symmetries we have that
\begin{eqnarray}
&&\int_{B_{\rho^{\frac{1}{n+1}}}(0)} \frac{16(n+1)^2  \delta^4 |y|^{2n}}{(\delta^2+|y^{n+1}-a|^2)^3}
e^{2\re[c_a (q^{-1}(y))^{n+1}]}= \int_{B_\rho(0)} \frac{16 (n+1) \delta^4}{(\delta^2+|y-a|^2)^3}
\re[1+2 c_a F_a(y)+|c_a|^2 G_a(y)] \nonumber \\
&&=\int_{B_{\rho}(a)} \frac{16 (n+1) \delta^4}{(\delta^2+|y-a|^2)^3}
\left[1+2 \re[c_a F_a(a)]+|c_a|^2 \re G_a(a)+\frac{1}{4} |c_a|^2 \Delta \re G_a(a) |y-a|^2 +O(|y-a|^{\frac{2(n+2)}{n+1}}) \right]\nonumber\\
&&+O(\delta^4)=8 \pi (n+1) \left[1+2 \re[c_a F_a(a)]+|c_a|^2 \re G_a(a)+\frac{1}{4}|c_a|^2 \Delta \re G_a(a) \delta^2\right] +O(\delta^{\frac{2(n+2)}{n+1}}) \label{1156}
\end{eqnarray}
in view of \eqref{expF}, \eqref{expG} and
$$\int_{\mathbb{R}^2} \frac{dy}{(1+|y|^2)^3}=\int_{\mathbb{R}^2} \frac{|y|^2}{(1+|y|^2)^3}dy=\frac{\pi}{2},$$
where $F_a$ and $G_a$ are given by \eqref{FG}. By \eqref{1156} we can re-write \eqref{firsttermbis} as
\begin{eqnarray}
&&\hspace{-0.5cm}\int_\Om (Z_0+1)\lf[\lap W+4\pi N\lf({e^{u_0+W} \over \int_\Om
e^{u_0+W}}-{1\over
|\Om|}\rg)\rg] \nonumber \\
&&\hspace{-0.5cm}= 8\pi (n+1) \left[ \frac{1+2 \re[c_a F_a(a)]+|c_a|^2 \re G_a(a)+\frac{1}{4}|c_a|^2 \Delta \re G_a(a) \delta^2}{1+2 \re[c_a F_a(a)] + |c_a|^2  \re G_a(a)+\frac{1}{2} |c_a|^2 \Delta \re G_a(a)  \delta^2 \log \frac{1}{\delta}+\frac{\delta^2}{n+1} D_a}
-1 \right]  +O(\de^2 |c_a|) \nonumber\\
&&\hspace{-0.5cm}+o(\delta^2)=  -16 \pi (n+1) |\alpha_a|^2 |c_a|^2 \delta^2 \log \frac{1}{\delta}-8\pi \delta^2 D_a +O(\de^2 |c_a|+|a|^{\frac{1}{n+1}}\delta^2 |\log \delta|)+o(\delta^2)
\label{1046pr}
\end{eqnarray}
in view of $\Delta \re G_a(a)=4|\alpha_a|^2+O(|a|^{\frac{1}{n+1}})$. By (\ref{eps4}) we also deduce that
\begin{eqnarray}
&&\int_\Om\frac{64 \pi^2 N^2 \epsilon^2 \int_\Omega
e^{2u_0+2W}}{(\int_\Omega e^{u_0+W}+\sqrt{(\int_\Omega e^{u_0+W})^2-16\pi
N\epsilon^2\int_\Omega e^{2u_0+2W}})^2} (Z_0+1)
\lf({e^{u_0+W}\over \int_\Om e^{u_0+W}}-{e^{2u_0+2W}\over
\int_\Om e^{2u_0+2W}}\rg)\nonumber \\
&&=\,\int_{\sigma^{-1}(B_\rho(0))}|\sigma'(z)|^2 e^{U_{\de,a}}(Z_0+1)
\lf[\frac{8(n+1)^2\epsilon^2}{\pi |\alpha_a|^{\frac{2}{n+1}} \delta^{\frac{2}{n+1}}}
E_{a,\delta}-\e^2 |\sigma'(z)|^2 e^{U_{\delta,a}} \rg]\left[1+O(|c_a||z|^{n+1}+\eta)+o(1) \right]\nonumber \\
&&+O(\delta^4 \eta)=\frac{128 (n+1)^3\epsilon^2}{\pi |\alpha_a|^{\frac{2}{n+1}} \delta^{\frac{2}{n+1}}}
E_{a,\delta} \int_{B_\rho(0)}{\de^4\over(\de^2+|y-a|^2)^3} \left[1+O(|c_a||y|+\eta)+o(1) \right] \nonumber \\
&&-128(n+1)^3\e^2 |\alpha_a|^{-\frac{2}{n+1}} \int_{B_\rho(0)}\frac{\de^6|y|^{2n\over
n+1}}{(\de^2+|y-a|^2)^5}\left[1+O(|y|^{1\over n+1}+\eta)+o(1) \right] +O(\delta^4 \eta)\nonumber \\
&& =64(n+1)^3 |\alpha_a|^{-\frac{2}{n+1}}  \epsilon^2 \delta^{-\frac{2}{n+1}} E_{a,\delta}-128(n+1)^3\e^2 |\alpha_a|^{-\frac{2}{n+1}} \int_{B_\rho(0)}\frac{\de^6|y+a|^{2n\over n+1}}{(\de^2+|y|^2)^5}\left[1+O(|y|^{1\over n+1}+\eta)+o(1) \right]\nonumber \\
&& +o(\eta+\delta^2)+O(\eta^2) \nonumber
\end{eqnarray}
in view of \eqref{1852}. Since
$$\delta^{\frac{2}{n+1}} \int_{B_\rho(0)}\frac{\de^6|y+a|^{2n\over n+1}}{(\de^2+|y|^2)^5}\left[1+O(|y|^{1\over n+1}+\eta)+o(1) \right] =\int_{\mathbb{R}^2}\frac{|y+\frac{a}{\delta}|^{2n\over n+1}}{(1+|y|^2)^5}+o(1)+O(\eta)$$
when $|a|=O(\delta)$, we then have that
\begin{eqnarray}
&&\int_\Om\frac{64 \pi^2 N^2 \epsilon^2 \int_\Omega
e^{2u_0+2W}}{(\int_\Omega e^{u_0+W}+\sqrt{(\int_\Omega e^{u_0+W})^2-16\pi
N\epsilon^2\int_\Omega e^{2u_0+2W}})^2} (Z_0+1)
\lf({e^{u_0+W}\over \int_\Om e^{u_0+W}}-{e^{2u_0+2W}\over
\int_\Om e^{2u_0+2W}}\rg)\nonumber \\
&& =64(n+1)^3 |\alpha_a|^{-\frac{2}{n+1}}  \eta \int_{\mathbb{R}^2} \frac{(|y|^2-1)|y+\frac{a}{\delta} |^{\frac{2n}{n+1}}}{(1+|y|^2)^5}+o(\eta+\delta^2)+O(\eta^2) \label{2024}
\end{eqnarray}
in view of \eqref{Eadelta}. Inserting \eqref{1046pr} and \eqref{2024} into \eqref{zerotermbisb}, we get the validity of \eqref{solve1b}.
\bremark \label{remark1}  
Notice that in the range $|a|>>\delta$ we find that
$$\delta^{\frac{2}{n+1}} \int_{B_\rho(0)}\frac{\de^6|y+a|^{2n\over n+1}}{(\de^2+|y|^2)^5}\left[1+O\bigg(|y|^{1\over n+1}+\eta\Big(\frac{|a|}{\delta}\Big)^{\frac{2n}{n+1}}\bigg)+o(1) \right] =\frac{\pi}{4} \Big(\frac{|a|}{\delta}\Big)^{\frac{2n}{n+1}}\bigg[1+o(1)+O\Big(\eta\Big(\frac{|a|}{\delta}\Big)^{\frac{2n}{n+1}}\Big)\bigg]$$
in view of the inequality $|y+a|^{\frac{2n}{n+1}}=|a|^{\frac{2n}{n+1}}+O(|a|^{\frac{n-1}{n+1}}|y|+|y|^{\frac{2n}{n+1}})$, so that the main order of $\int_\Omega R PZ_0$ in this range is essentially given by
\begin{eqnarray*}
- 16 \pi (n+1) |\alpha_a|^2 |c_a|^2 \delta^2 \log \frac{1}{\delta}-8\pi \delta^2 D_a - \frac{32 \pi}{3} (n+1)^3 |\alpha_a|^{-\frac{2}{n+1}}  \eta \Big(\frac{|a|}{\delta}\Big)^{\frac{2n}{n+1}}.\end{eqnarray*}
\eremark

\medskip \noindent By \eqref{pzij} and \eqref{int|R|} we deduce that
\begin{eqnarray} \label{zerotermbisc}
\int_\Omega R \, PZ =\int_\Omega  R Z+o(\delta |c_a|+\delta |a|+\eta+\delta^2)+O(\eta^2 \delta)
\end{eqnarray}
in view of $\int_\Om R=0$. Since as before
\begin{eqnarray*}
\int_\Omega |Z|&=& \int_{\sigma^{-1}(B_\rho(0))} \frac{
\de|\sigma(z)-a|}{\de^2+|\sigma(z)-a|^2}+O(\de)=
O\bigg(\int_{B_\rho(0)} |y|^{-{2n \over n+1}} \frac{\de |y-a|}{\de^2+|y-a|^2}\bigg)+O(\de)\\
&=& O\bigg(\de^{1\over n+1} \int_{B_\rho(0)}
\frac{1}{|y|^{2n \over n+1} |y-a|^{1 \over
n+1}}\bigg)+O(\de)=O(\de^{1\over n+1}),
\end{eqnarray*}
by (\ref{imp}) we have that
\begin{eqnarray}
&&\hspace{-0.5cm} \int_\Om Z\,\lf[\lap W+4\pi N\lf({ e^{u_0+W} \over \int_\Om
e^{u_0+W}}-{1\over |\Om|}\rg)\rg]  \label{1011} \\
&&\hspace{-0.5cm} =\int_{\sigma^{-1}(B_\rho(0))}|\sigma'(z)|^2 e^{U_{\de,a}}Z
\left[\frac{e^{2\re[c_a z^{n+1}]}}{1+2 \re[c_a F_a(a)] + |c_a|^2  \re G_a(a)+\frac{1}{2} |c_a|^2 \Delta \re G_a(a)  \delta^2 \log \frac{1}{\delta}+\frac{\delta^2}{n+1} D_a}-1\right]\nonumber\\
&&\hspace{-0.5cm} +O(\de^2 |c_a|)+o(\delta^2) \nonumber\\
&&\hspace{-0.5cm}=\int_{B_{\rho^{\frac{1}{n+1}} }(0)} \frac{8(n+1)^2 \delta^3 |y|^{2n}(y^{n+1}-a)}{(\delta^2+|y^{n+1}-a|^2)^3} \frac{e^{2\re[c_a (q^{-1}(y))^{n+1}]}}{1+2 \re[c_a F_a(a)] + |c_a|^2  \re G_a(a)+\frac{1}{2} |c_a|^2 \Delta \re G_a(a)  \delta^2 \log \frac{1}{\delta}+\frac{\delta^2}{n+1} D_a}\nonumber \\
&&\hspace{-0.5cm} -\int_{B_\rho(0)} \frac{8(n+1) \delta^3 (y-a)}{(\delta^2+|y-a|^2)^3}+O(\de^2 |c_a|)+o(\delta^2)\nonumber \\
&&\hspace{-0.5cm} =\frac{\int_{B_{\rho^{\frac{1}{n+1}} }(0)} \frac{8(n+1)^2 \delta^3 |y|^{2n}(y^{n+1}-a)}{(\delta^2+|y^{n+1}-a|^2)^3}  e^{2\re[c_a (q^{-1}(y))^{n+1}]}}{1+2 \re[c_a F_a(a)] + |c_a|^2  \re G_a(a)+\frac{1}{2} |c_a|^2 \Delta \re G_a(a)  \delta^2 \log \frac{1}{\delta}+\frac{\delta^2}{n+1} D_a}+O(\de^2 |c_a|)+o(\delta^2)\nonumber
\end{eqnarray}
in view of
$$\int_{B_\rho(a)} \frac{8(n+1) \delta^3 (y-a)}{(\delta^2+|y-a|^2)^3}=0.$$
Since expansion \eqref{keyexp2} is still valid in view of \eqref{symmetryintbis}, through the changes of variables $y \to y^{n+1}$ and $y \to \frac{y-a}{\delta}$, by the symmetries we have that
\begin{eqnarray} \label{1158}
&&\int_{B_{\rho^{\frac{1}{n+1}} }(0)} \frac{8(n+1)^2 \delta^3 |y|^{2n}(y^{n+1}-a)}{(\delta^2+|y^{n+1}-a|^2)^3} e^{2\re[c_a (q^{-1}(y))^{n+1}]}\nonumber \\
&&=\int_{B_\rho(0)} \frac{8(n+1) \delta^3(y-a)}{(\delta^2+|y-a|^2)^3} \re[1+2 c_a F_a(y)+|c_a|^2 G_a(y)] \nonumber \\
&&=\int_{B_{\rho}(a)} \frac{8 (n+1) \delta^3}{(\delta^2+|y-a|^2)^3}
\left[\overline{c_a F_a'(a)} |y-a|^2+\frac{1}{2} |c_a|^2 (\partial_1+i\partial_2) \re G_a(a) |y-a|^2 +O(|c_a|^2 |y-a|^3) \right]+O(\delta^3)\nonumber\\
&&=4 \pi (n+1) \delta \left[\overline{c_a F_a'(a)} +\frac{1}{2} |c_a|^2 (\partial_1+i\partial_2) \re G_a (a) \right] +O(\de^2 |c_a|^2+\delta^3)
\end{eqnarray}
in view of \eqref{expF}, \eqref{expG} and $\int_{\mathbb{R}^2} \frac{|y|^2}{(1+|y|^2)^3}dy=\frac{\pi}{2}$, where $F_a$ and $G_a$ are given by \eqref{FG}. By \eqref{1158} we can re-write \eqref{1011} as
\begin{eqnarray}
&&\int_\Om Z \lf[\lap W+4\pi N\lf({e^{u_0+W} \over \int_\Om
e^{u_0+W}}-{1\over
|\Om|}\rg)\rg] =4 \pi (n+1) \delta \left[\overline{c_a F_a'(a)} +\frac{1}{2} |c_a|^2 (\partial_1+i\partial_2) \re G_a(a) \right] \nonumber\\
&&+o(\de |c_a|+\delta^2)=4 \pi (n+1) \delta \overline{\alpha_a c_a} +o(\de |c_a|+\delta^2) \label{firstterm}
\end{eqnarray}
in view of $F_a'(a)=\alpha_a+O(|a|)$ and $\frac{1}{2} (\partial_1+i\partial_2) \re G_a(a)=O(|a|)$. As far as the second term of $R$, by (\ref{eps4}) we have that
\begin{eqnarray}
&&\int_\Om\frac{64 \pi^2 N^2 \epsilon^2 \int_\Omega
e^{2u_0+2W}}{(\int_\Omega e^{u_0+W}+\sqrt{(\int_\Omega e^{u_0+W})^2-16\pi
N\epsilon^2\int_\Omega e^{2u_0+2W}})^2}  Z \lf({e^{u_0+W}\over
\int_\Om e^{u_0+W}}-{e^{2u_0+2W}\over
\int_\Om e^{2u_0+2W}}\rg)\nonumber \\
&&=\,\int_{\sigma^{-1}(B_\rho(0))}|\sigma'(z)|^2 e^{U_{\de,a}}Z
\lf[\frac{8(n+1)^2\epsilon^2}{\pi |\alpha_a|^{\frac{2}{n+1}} \delta^{\frac{2}{n+1}}}
E_{a,\delta}-\e^2 |\sigma'(z)|^2 e^{U_{\delta,a}} \rg] \left[1+O(|c_a||z|^{n+1}+\eta)+o(1) \right] \nonumber\\
&&+O(\de^3\eta) = \frac{64(n+1)^3 \epsilon^2}{\pi  |\alpha_a|^{\frac{2}{n+1}} \delta^{\frac{2}{n+1}}}
E_{a,\delta} \int_{B_\rho(0)}\frac{\de^3 (y-a)}{(\de^2+|y-a|^2)^3}dy \left[1+O(|c_a||y|+\eta)+o(1) \right] \nonumber \\
&&-64 (n+1)^3 \e^2 |\alpha_a|^{-\frac{2}{n+1}} \int_{B_\rho(0)} \frac{\delta^5 |y|^{2n \over n+1}(y-a)}{(\de^2+|y-a|^2)^5}\left[1+O(|y|^{\frac{1}{n+1}}+\eta)+o(1) \right]+O(\de^3\eta) \nonumber \\
&&=\,-64(n+1)^3 |\alpha_a|^{-\frac{2}{n+1}} \eta \int_{\mathbb{R}^2} {|y+\frac{a}{\delta}|^{2n\over n+1} y
\over(1+|y|^2)^5}+o(\eta)+O(\eta^2)
\label{secondterm}
\end{eqnarray}
in view of \eqref{1852} and
$$\int_{B_\rho(0)}\frac{\de^3 (y-a)}{(\de^2+|y-a|^2)^3}dy=\int_{B_\rho(a)}\frac{\de^3 (y-a)}{(\de^2+|y-a|^2)^3}dy+O(\de^3)=O(\de^3).$$
Inserting (\ref{firstterm}) and (\ref{secondterm}) into (\ref{zerotermbisc}), we get the validity of (\ref{solve2b}).\qed
\end{proof}
\bremark\label{remark2}  Since for $|a|>>\delta$ and $n>1$
\begin{eqnarray*}
\delta^{\frac{2}{n+1}} \int_{B_\rho(0)}\frac{\de^5 |y|^{2n\over n+1}(y-a)}{(\de^2+|y-a|^2)^5}=\delta^{\frac{2}{n+1}} \int_{B_\rho(0)}\frac{\de^5 |y+a|^{2n\over n+1}y}{(\de^2+|y|^2)^5}+o(1)= \frac{\pi n}{12(n+1)} \Big(\frac{|a|}{\delta}\Big)^{-\frac{2}{n+1}} \frac{a}{\delta} [1+o(1)]
\end{eqnarray*}
in view of
$$\int_{\mathbb{R}^2} \frac{|y|^2}{(1+|y|^2)^5}=\int_{\mathbb{R}^2} \frac{1}{(1+|y|^2)^4}-\int_{\mathbb{R}^2} \frac{1}{(1+|y|^2)^5}= \frac{\pi}{12}$$
and the inequality
$$|y+a|^{\frac{2n}{n+1}}=|a|^{\frac{2n}{n+1}}+\frac{n}{n+1} |a|^{-\frac{2}{n+1}}(a\overline{y}+\overline{a}y)+O(|a|^{-\frac{2}{n+1}}|y|^2+|y|^{\frac{2n}{n+1}}),$$
notice that the main order of $\int_\Omega R PZ$, in this range, is essentially given by
$$4 \pi (n+1) \delta \overline{\alpha_a c_a}-\frac{16}{3} \pi n  (n+1)^2 \e^2 \delta^{-\frac{2}{n+1}} |\alpha_a|^{-\frac{2}{n+1}}  \Big(\frac{|a|}{\delta}\Big)^{-\frac{2}{n+1}} \frac{a}{\delta}.$$
Since $\alpha_a$ is uniformly away from zero, the vanishing of $\int_\Omega R PZ$, which is equivalent to have $\epsilon^2 \delta^{-\frac{2}{n+1}} (\frac{|a|}{\delta})^{\frac{2n}{n+1}} \sim \overline{\alpha_a c_a a}$, is generally not compatible in the range $|a|>>\delta$ with the vanishing of $\int_\Omega RPZ_0$ in view of Remark \ref{remark1}, which can take place only if $c_0=0$ (in which case $c_a \sim a$). Indeed, the vanishing of $\int_\Omega RPZ$ and $\int_\Omega RPZ_0$ in the range $|a|>>\delta$ implies the contradiction $|a|^2 \sim \delta^2$. This explains why we don't consider the case $|a|>>\delta$.
\eremark

%%%%%%%%%%%%%%%%%%%%%%%%%%%%%%%%%%%%%%%%%%%%%%%%%%%%%%%%%%%%%%%%%%%%%%%%%%%%%%%%%%%%%%%%%%%%%%%%%%%%%%%%%%%
%%%%%%%%%%%%%%%%%%%%%%%%%%%%%%%%%%%%%%%%%%%%%%%%%%%%%%%%%%%%%%%%%%%%%%%%%%%%%%%%%%%%%%%%%%%%%%%%%%%%%%%%%%%

\section{Proof of the main results}\label{mainresults}
In the previous section, we have built up an approximating function $W=PU_{\de,a,\sigma_a}$. We will now look for solutions
$w$ of the form $w=W+\phi$, where $\phi$ is a small correcting
term. In terms of $\phi$, problem \eqref{3} is equivalent to find
a doubly-periodic solution $\phi$ of
\begin{equation}\label{ephi}
L(\phi)=-[R+N(\phi)]\qquad\text{ in $\Om$}
\end{equation}
with $\int_\Omega \phi=0$. Recalling the notation $B(w)=16 \pi N
(\int_\Om e^{2u_0+2w})(\int_\Om e^{u_0+w})^{-2}$, the linear operator $L$
is given by
$$L(\phi) = \Delta \phi + \ml{K} \phi+\tilde \gamma(\phi),$$
where
$$\ml{K}=4\pi N {e^{u_0+W}\over\int_\Om e^{u_0+W}} +\frac{4 \pi N \epsilon^2  B(W)}{\lf(1+\sqrt{1-\epsilon^2 B(W)}\rg)^2} \left({e^{u_0+W}\over\int_\Om e^{u_0+W}}- 2 \frac{e^{2u_0+2W}}{\int_\Omega e^{2u_0+2W}}\right) $$ and
\begin{equation*}
\begin{split}
\tilde \gamma(\phi)&=-4\pi N {e^{u_0+W} \int_\Om
e^{u_0+W} \phi \over(\int_\Om e^{u_0+W})^2 }-\frac{4 \pi N \epsilon^2  B(W)}{\lf(1+\sqrt{1-\epsilon^2 B(W)}\rg)^2} {e^{u_0+W} \over (\int_\Om e^{u_0+W})^2} \int_\Om e^{u_0+W} \phi \\
&+\frac{8 \pi N \epsilon^2  B(W)}{\lf(1+\sqrt{1-\epsilon^2
B(W)}\rg)^2}  \frac{
e^{2u_0+2W}}{(\int_\Omega e^{2u_0+2W})^2} \int_\Omega e^{2u_0+2W} \phi \\
&+4 \pi N \epsilon^2 \frac{DB(W)[\phi]}{(1+\sqrt{1-\epsilon^2
B(W)})^2\sqrt{1-\epsilon^2 B(W)}} \left(\frac{e^{u_0+W}}{\int_\Omega
e^{u_0+W}}-\frac{e^{2u_0+2W}}{\int_\Omega e^{2u_0+2W}}\right)
\end{split}
\end{equation*}
with
$$DB(W)[\phi]= 2 B(W) \left(
{\int_\Om e^{2u_0+2W} \phi \over \int_\Om e^{2u_0+2W}}- {\int_\Om
e^{u_0+W} \phi \over \int_\Om e^{u_0+W}}\right).$$ The nonlinear term
$N(\phi)$, which is quadratic in $\phi$, is given by
\begin{eqnarray} \label{nlt}
&&N(\phi)=4\pi N\left[\frac{e^{u_0+W+\phi}}{\int_\Omega
e^{u_0+W+\phi}}-\frac{e^{u_0+W}}{\int_\Omega e^{u_0+W}}-{e^{u_0+W}\over\int_\Om
e^{u_0+W}}\lf(\phi-\frac{\int_\Om
e^{u_0+W}\phi}{\int_\Om e^{u_0+W}}\rg)\right]\nonumber\\
&&+\left[\frac{4 \pi N \epsilon^2 B(W+\phi)}{(1+\sqrt{1-\epsilon^2 B(W+\phi)})^2}-\frac{4 \pi N \epsilon^2  B(W)}{(1+\sqrt{1-\epsilon^2 B(W)})^2}-\frac{4 \pi N \epsilon^2 DB(W)[\phi]}{(1+\sqrt{1-\epsilon^2 B(W)})^2\sqrt{1-\epsilon^2 B(W)}} \right]\times \nonumber\\
&&\times \left(\frac{e^{u_0+W+\phi}}{\int_\Omega
e^{u_0+W+\phi}}-\frac{
e^{2(u_0+W+\phi)}}{\int_\Omega e^{2(u_0+W+\phi)}}\right)\nonumber \\
&&+\frac{4 \pi N \epsilon^2 B(W)}{\lf(1+\sqrt{1-\epsilon^2
B(W)}\rg)^2}\left[\frac{e^{u_0+W+\phi}}{\int_\Omega e^{u_0+W+\phi}}-\frac{e^{u_0+W}}{\int_\Omega e^{u_0+W}}-{e^{u_0+W} \over\int_\Om
e^{u_0+W}}\lf(\phi-\frac{\int_\Om
e^{u_0+W}\phi}{\int_\Om e^{u_0+W}}\rg) \right] \\
&&-\frac{4 \pi N \epsilon^2  B(W)}{\lf(1+\sqrt{1-\epsilon^2
B(W)}\rg)^2}\left[\frac{e^{2(u_0+W+\phi)}}{\int_\Omega
e^{2(u_0+W+\phi)}}-\frac{e^{2(u_0+W)}}{\int_\Omega e^{2(u_0+W)}}-2
\frac{
e^{2(u_0+W)}}{\int_\Omega e^{2(u_0+W)}}\left( \phi-\frac{\int_\Omega e^{2(u_0+W)} \phi}{\int_\Omega e^{2(u_0+W)}} \right)\right] \nonumber\\
&&+\frac{4 \pi N \epsilon^2  DB(W)[\phi]}{(1+\sqrt{1-\epsilon^2
B(W)})^2\sqrt{1-\epsilon^2 B(W)}}
\left(\frac{e^{u_0+W+\phi}}{\int_\Omega
e^{u_0+W+\phi}}-\frac{e^{u_0+W}}{\int_\Omega e^{u_0+W}}-\frac{
e^{2(u_0+W+\phi)}}{\int_\Omega e^{2(u_0+W+\phi)}}+\frac{
e^{2(u_0+W)}}{\int_\Omega e^{2(u_0+W)}}\right). \nonumber
\end{eqnarray}
Notice that we can re-write $\tilde \gamma (\phi)$ as
\begin{equation*}
\begin{split}
\tilde \gamma(\phi)&=- \ml{K} {\int_\Om e^{u_0+W}\phi \over \int_\Om
e^{u_0+W}} +\frac{8\pi N \epsilon^2  B(W)}{(1+\sqrt{1-\epsilon^2
B(W)})^2 \sqrt{1-\epsilon^2 B(W)}}
 \left({\int_\Om e^{2(u_0+W)} \phi \over \int_\Om e^{2(u_0+W)}}- {\int_\Om e^{u_0+W} \phi \over \int_\Om e^{u_0+W} }\right)
\left[{ e^{u_0+W} \over \int_\Om e^{u_0+W}}\right.\\
&\left.+(\sqrt{1-\epsilon^2 B(W)}-1){ e^{2(u_0+W)} \over \int_\Om e^{2(u_0+W)}}\right]\\
&=\ml{K}\left[- {\int_\Om e^{u_0+W}\phi \over \int_\Om e^{u_0+W}}
+\frac{\epsilon^2  B(W)}{(1+\sqrt{1-\epsilon^2
B(W)})\sqrt{1-\epsilon^2 B(W)}}
 \left({\int_\Om e^{2(u_0+W)} \phi \over \int_\Om e^{2(u_0+W)}}- {\int_\Om e^{u_0+W} \phi \over \int_\Om e^{u_0+W}}\right)
\right],
\end{split}
\end{equation*}
and  $L$ as
\begin{equation}\label{ol}
L(\phi) = \Delta \phi + \ml{K} \left[ \phi+ \gamma(\phi)\right],
\end{equation}
where
$$\gamma(\phi)=- {\int_\Om e^{u_0+W}\phi \over \int_\Om e^{u_0+W}}
+\frac{\epsilon^2  B(W)}{(1+\sqrt{1-\epsilon^2
B(W)})\sqrt{1-\epsilon^2 B(W)}}
 \left({\int_\Om e^{2(u_0+W)} \phi \over \int_\Om  e^{2(u_0+W)}}- {\int_\Om e^{u_0+W} \phi \over \int_\Om e^{u_0+W}}\right).$$
Let us observe that
$$\int_\Om R=\int_\Om L(\phi)=\int_\Om N(\phi)=0.$$

\medskip \noindent Since the operator $L$ is not invertible, equation $L(\phi)=-R-N(\phi)$ is not generally solvable.
The linear theory we will develop in Appendix B states that
$L$ has a kernel which is almost generated by $PZ_0$, $PZ$ and
$\overline{PZ}$, yielding to
\begin{prop} \label{prop4.1}
Let $M_0>0$. There exists $\eta_0>0$ small such that for any $0<\de\leq \eta_0$,
$|\log \delta| \e^2 \leq \eta_0 \de^{2\over n+1}$, $|a|\leq M_0 \de$ and $h\in
L^\infty(\Om)$ with $\int_\Om h=0$ there is a unique solution
$\phi$, $d_0\in\R$ and $d \in\C$ to
\begin{equation}\label{plcobis}
\left\{\begin{array}{ll}
L(\phi) =h + d_0 \Delta PZ_{0}+\re[d \lap PZ] &\text{in }\Om\\
\int_{\Omega } \phi=\int_{\Omega } \phi \Delta PZ_0 = \int_\Om \phi \Delta PZ=0.&
\end{array} \right.
\end{equation}
Moreover, there is a constant $C>0$ \st
$$\|\phi \|_\infty \le C\lf(\log \frac 1\de \rg)\|h\|_*,\qquad
|d_{0}|+|d| \le C\|h\|_*.$$
\end{prop}
\noindent As a consequence, in Appendix C we will show
\begin{prop}\label{nlp}
Let $M_0>0$. There exists $\eta_0>0$ small such that for any
$0<\de\leq\eta_0$, $|\log \delta|^2 \e^2\leq \eta_0 \de^{2\over n+1}$
and $|a|\leq M_0 \de$ there is a unique solution
$\phi=\phi(\de,a)$, $d_0=d_0(\de,a)\in\R$ and $d=d(\de,a)\in\C$ to
\begin{equation}\label{linear}
\left\{\begin{array}{ll}
L(\phi) =-[R+N(\phi)]  + d_0 \Delta PZ_{0}+\re[d \lap PZ] &\text{in }\Om\\
\int_{\Omega } \phi=\int_{\Omega } \phi  \Delta PZ_0= \int_\Om \phi \Delta PZ=0.&
\end{array} \right.
\end{equation}
Moreover, the map $(\de,a)\mapsto \phi(\de,a)$ is $C^1$ with
\begin{equation}\label{estphi}
\|\phi\|_\infty\le C |\log \delta| \|R\|_*.
%\:,\qquad \|\partial \phi\|_\infty\le C \delta^{1-\sigma}|\log \delta|^2
\end{equation}
\end{prop}
\noindent The function $W+\phi$ will be a true solution of equation (\ref{3}) once we adjust $\delta$ and $a$ to have
$d_0(\de,a)=d(\de,a)=0$. \noindent The crucial point is the following:
\begin{lem} \label{1039}
Let $\phi=\phi(\de,a)$, $d_0=d_0(\de,a)\in\R$ and
$d=d(\de,a)\in\C$ be the solution of \eqref{linear} given by Proposition \ref{nlp}. There exists $\eta_0>0$ \st if
$0<\de \leq \eta_0$, $|a| \leq \eta_0$ and
\begin{equation}  \label{solve}
\int_\Omega (L(\phi)+N(\phi)+R) PZ_0=0,\qquad \int_\Omega
(L(\phi)+N(\phi)+R) PZ=0
\end{equation}
do hold, then $W+\phi$ is a solution
of \eqref{3}, i.e. $d_0(\de,a)=d(\de,a)=0$.
\end{lem}
\begin{proof} Since by (\ref{pzij}) and $\|Z_0\|_\infty+\|Z\|_\infty\leq 2$
there hold
\begin{eqnarray*}
\int_\Om \lap PZ_0PZ_0&=& \int_\Om\lap Z_0  PZ_0=-\int_{\sigma^{-1}(B_\rho(0))}|\sigma'(z)|^2 e^{U_{\de,a}}Z_0(Z_0+1)+O(\de^2)\\
&=&- 16 (n+1) \delta^4 \int_{B_\rho(0)}
\frac{\delta^2-|y-a|^2}{(\delta^2+|y-a|^2)^4} +O(\de^2) =- \frac{8 \pi}{3} (n+1) +O(\de^2)
\end{eqnarray*}
and
\begin{eqnarray*}
\int_\Om \lap PZPZ_0&=&\int_\Om\lap Z PZ_0=-\int_{\sigma^{-1}(B_\rho(0))}|\sigma'(z)|^2 e^{U_{\de,a}}Z(Z_0+1)+O(\de^2)\\
&=& -\int_{B_\rho(0)}{16(n+1)\de^5 (y-a)
\over(\de^2+|y-a|^2)^4} +O(\de^2)=-
\int_{B_\rho(0)}{16(n+1)\de^5 y \over(\de^2+|y|^2)^4}
+O(\de^2)=O(\de^2)
\end{eqnarray*}
in view of \eqref{deltaZ0}-\eqref{deltaZ} and
$$\int_{\mathbb{R}^2} \frac{1-|y|^2}{(1+|y|^2)^4}dy=2 \int_{\mathbb{R}^2} \frac{dy}{(1+|y|^2)^4}-\int_{\mathbb{R}^2} \frac{dy}{(1+|y|^2)^3}=\frac{\pi}{6},$$
by (\ref{linear}) we rewrite the first of (\ref{solve}) as
\begin{eqnarray*}
0=d_0 \int_\Omega \lap PZ_0PZ_0+\int_\Om\re[d \lap PZPZ_0]=-\frac{8}{3}\pi
(n+1) d_0+O(\delta^2 |d_0|+\delta^2 |d|).
\end{eqnarray*}
Similarly, the second of (\ref{solve}) gives that
\begin{eqnarray*}
0&=& d_0 \int_\Omega \lap PZ_0 PZ+\int_\Om {1\over 2}\lf[d \lap
PZ+\bar d \lap\overline{PZ}\rg]PZ=
-\int_{\sigma^{-1}(B_\rho(0))} {1\over 2} |\sigma'(z)|^2 e^{U_{\de,a}} \lf[d Z +\bar d\ \overline{Z} \rg] Z\\
&&+ O(\delta^2 |d_0|+\delta |d|)= -4 (n+1) \bar
d\int_{\mathbb{R}^2}  \frac{|y|^2 }{(1+|y|^2)^4}+ O(\delta^2
|d_0|+\delta |d|)
\end{eqnarray*}
in view of $\int_\Omega \lap PZ_0 PZ=\int_\Omega \lap PZ PZ_0=O(\delta^2)$, \eqref{deltaZ} and \eqref{pzij}. Hence, (\ref{solve}) can be simply re-written as $d_0+O(\delta^2
|d_0|+\delta^2 |d|)=0$, $d+O(\delta^2 |d_0|+\delta |d|)=0$.
Summing up the two relations, we then obtain $|d_0|+|d|=\delta
O(|d_0|+|d|)$ which implies $d_0=d=0$.\qed
\end{proof}

\bremark \label{remark2bis} Since $\phi$ is sufficiently small, the system \eqref{solve} will be a perturbation of the reduced equations $\int_\Omega R \, PZ_0=0$, $\int_\Omega R \, PZ=0$. The integral coefficient in \eqref{solve1b} is negative for all $\frac{a}{\delta}$, as we will see in Appendix D. Since $\alpha_a \to \alpha_0 =\frac{\mathcal{H}(0)}{n+1}\not=0$ and $c_a \to c_0$ as $a \to 0$, we can always exclude the case $c_0 \not=0$. Indeed, in such a case the equation $\int_\Omega R \, PZ_0=0$ yields to $\e^2\de^{-{2\over n+1}} \sim \delta^2 |\log \delta|$ as $\delta \to 0$ by means of \eqref{solve1b} (we are implicitly assuming $\e^2\de^{-{2\over n+1}} \to 0$, which is a natural range for solving the reduced equations through \eqref{solve1b}-\eqref{solve2b}). This is not compatible with $\int_\Omega R \, PZ=0$, which allows at most $\delta=O(\e^2\de^{-{2\over n+1}})$ by means of \eqref{solve2b}.
\eremark

\medskip \noindent The last ingredient is an expansion of the system (\ref{solve}) with the aid of Proposition \ref{reducedequations}:
\begin{prop} \label{1219}
Assume $c_0=0$ and $|a|\leq M_0\delta$ for some $M_0>0$. The
following expansions do hold as $\de\to 0$ and $\e\to0$
\begin{eqnarray}
\int_\Omega (L(\phi)+N(\phi)+R) PZ_0&=& -8\pi\delta^2  D_0 +64(n+1)^{\frac{3n+5}{n+1}} |\mathcal{H}(0)|^{-\frac{2}{n+1}} \e^2 \de^{-{2 \over n+1}}
\int_{\mathbb{R}^2} \frac{(|y|^2-1)|y+\frac{a}{\delta} |^{\frac{2n}{n+1}}}{(1+|y|^2)^5}  \nonumber \\
&&+ o(\de^2+\e^2\de^{-{1\over n+1}})+O(\e^4 \de^{-{2\over n+1}}|\log \delta|^2+\e^8 \de^{-{4\over n+1}}|\log \delta|^2) \label{solve1}
\end{eqnarray}
and
\begin{eqnarray}
\int_\Omega (R+L(\phi)+N(\phi)) PZ &=& 4 \pi \delta (\bar \Upsilon a+ \bar \Gamma \bar a)  -64(n+1)^{\frac{3n+5}{n+1}} |\mathcal{H}(0)|^{-\frac{2}{n+1}} \e^2 \de^{-{2\over n+1}}
\int_{\mathbb{R}^2} {|y+\frac{a}{\delta}|^{2n\over n+1} y
\over(1+|y|^2)^5} \nonumber \\
&&+o(\de^2+\epsilon^2 \delta^{-\frac{2}{n+1}})+O(\e^4 \de^{-{2\over n+1}}|\log \delta|^2+\e^8 \de^{-{4\over n+1}}|\log \delta|^2), \label{solve2}
\end{eqnarray}
where $D_0$ and $\Gamma$, $\Upsilon$ are defined in \eqref{ggg} and
Lemma \ref{derivca}, respectively.
\end{prop}
\begin{proof}
First, note that from the assumptions and \eqref{ere}, we find
that $\|R\|_*=O(\de^{2-\gamma}+\eta +\eta^2)$, where $\eta=\epsilon^2 \delta^{-\frac{2}{n+1}}$. Hence,
since $|\gamma(\phi)|=O((1+\eta)\|\phi\|_\infty)$ in view of \eqref{BW}, by (\ref{estphi}), (\ref{diesis}), (\ref{diesisdiesis})
and (\ref{star}) we have that
\begin{eqnarray} \label{zerotermbis}
\int_\Omega (R+L(\phi)+N(\phi)) PZ_0&=& \int_\Omega R PZ_0+O\lf( (1+\eta) \Big\|\ti
L\Big(PZ_0+\frac{1}{|\Omega|}\int_\Omega
Z_0\Big)\Big\|_*\|\phi\|_\infty+\|\phi\|_\infty^2\rg)\\
&=& \int_\Omega R P Z_0+
o(\de^2+\eta)+O(\eta^2+\eta^4)  \nonumber
\end{eqnarray}
and
\begin{eqnarray} \label{zeroterm}
\int_\Omega (R+L(\phi)+N(\phi)) PZ&=&\int_\Omega RPZ+O\lf( (1+\eta) \Big\|\ti
L\Big(PZ+\frac{1}{|\Omega|}\int_\Omega
Z\Big)\Big\|_*\|\phi\|_\infty+\|\phi\|_\infty^2\rg)\\
&=& \int_\Omega R PZ+ o(\de^2+\eta)+O(\eta^2+ \eta^4)\nonumber
\end{eqnarray}
in view of $PZ_0=O(1)$ and $PZ=O(1)$, where $\ti L(\phi)=\lap
\phi+\ml{K}\phi$. Since by Lemma \ref{derivca} $\mathcal{H}(0)
c_a=\Gamma a+\Upsilon \bar a+o(|a|)$ as $a \to 0$ in view of
$c_0=0$, the desired expansions \eqref{solve1}-\eqref{solve2}
follow by a combination of \eqref{solve1b}-\eqref{solve2b} and
\eqref{zerotermbis}-\eqref{zeroterm}. We have used that $\alpha_a
\to \alpha_0=\frac{\mathcal{H}(0)}{n+1}$ as $a \to 0$ in view of
\eqref{0942}, where $\alpha_a$ is given by (\ref{alpha0}), and $D_a
\to D_0$ as $a \to 0$, where $D_a$ is given by \eqref{Da}.\qed
\end{proof}

\medskip \noindent Thanks to \eqref{solve1}-\eqref{solve2}, the aim is to find $(\de(\e),a(\e))$ so that \eqref{solve} does hold. To simplify the notations, we denote
$$\varphi_0(\de,a,\epsilon)=\int_\Omega (L(\phi)+N(\phi)+R)
PZ_0 \qquad \varphi(\de,a,\epsilon)=\overline{ \int_\Omega
(L(\phi)+N(\phi)+R) PZ},$$ and \eqref{solve} reduces to find a solution of
\begin{equation}\label{solve3}
\varphi_0(\de(\e),a(\e),\epsilon)=\varphi(\de(\e),a(\e),\epsilon)=0
\end{equation}
for $\e$ small. We are now ready to prove our first main result, which clearly implies the validity of Theorem \ref{mainbb} with $m=1$.
\begin{thm} \label{main} Let $\mathcal{H}_0=\frac{\mathcal{H}}{z^{n+2}}$, where $\mathcal{H}$ is given in \eqref{definitionH}, be a meromorphic function in $\Omega$ with $|\mathcal{H}_0(z)|^2=e^{u_0+8\pi(n+1)G(z,0)}$ (which exists in view of \eqref{balance} and is unique up to rotations), and $\sigma_0(z)=-(\int^z \mathcal{H}_0 (w) dw)^{-1}$. Assume that
\begin{equation}\label{pc}
\frac{d^{n+1} \mathcal{H}}{dz^{n+1}}(0)=0
\end{equation}
and for some small $\rho>0$
\begin{equation} \label{D0}
D_0:=\frac{1}{\pi}\left[\int_{\Omega \setminus \sigma_0^{-1} (B_\rho(0))}  e^{u_0+8\pi(n+1)G(z,0)} -
\int_{\mathbb{R}^2 \setminus B_\rho(0)}
\frac{n+1}{|y|^4}\right]<0.
\end{equation}
If the ``non-degeneracy condition"
\begin{equation} \label{nondegenracy}
|\Gamma| \not= \Big|\Upsilon+{ n(2n+3) \over n+1} D_0\Big|
\end{equation}
does hold, where $\Gamma$ and $\Upsilon$ are given in Lemma \ref{derivca}, for $\epsilon>0$ small there exist $a (\epsilon)$,
$\delta(\epsilon)>0$ small so that $w_\epsilon=PU_{\delta(\epsilon),a(\epsilon), \sigma_{a(\epsilon)}}+\phi(\delta(\epsilon),a(\epsilon))$ does solve \eqref{3} with
\begin{eqnarray*}  &&4\pi N \frac{e^{u_0+w_\e}}{\int_\Omega e^{u_0+w_\e}}+
\frac{64 \pi^2N^2 \epsilon^2 \int_\Omega
e^{2u_0+2w_\e}}{(\int_\Omega e^{u_0+w_\e}+\sqrt{(\int_\Omega
e^{u_0+w_\e})^2-16\pi N\epsilon^2\int_\Omega
e^{2u_0+2w_\e}})^2}\left(\frac{e^{u_0+w_\e}}{\int_\Omega
 e^{u_0+w_\e}} -\frac{e^{2u_0+2w_\e}}{\int_\Omega
e^{2u_0+2w_\e}}\right)\\
&&\hspace{2cm} \rightharpoonup 8\pi(n+1) \delta_0
\end{eqnarray*}
in the sense of measures as $\epsilon \to 0$.
\end{thm}
\bremark \label{0923} For simplicity, we are considering the case $p=0$ in Theorem \ref{main}, which however is still true for $p \not=0$ by simply replacing in the statement $\mathcal{H}$, $\mathcal{H}_0$ and corresponding quantities with $\mathcal{H}^p$, $\mathcal{H}_0^p$ and corresponding quantities at $p$, where the latter have been defined in Remark \ref{1149}.
\eremark
\begin{proof} Since the equation $\varphi_0(\de,a,\epsilon)=0$ naturally requires $\delta^2 \sim \e^2 \delta^{-\frac{2}{n+1}}$ in view of \eqref{solve1}, we make the following change of variables: $\delta=[\frac{(n+1)\e^{n+1} }{|\mathcal{H}(0)|}]^{\frac{1}{n+2}} \mu$ and $\zeta=\frac{a}{\delta}$. The system \eqref{solve3} is equivalent to find zeroes of
$$\Gamma_\epsilon(\mu,\zeta):=\left[ \frac{(n+1) \e^{n+1}}{|\mathcal{H}(0)|} \right] ^{-\frac{2}{n+2}}  \left(- \frac{1}{8} \varphi_0, \frac{1}{4\pi \mu^2 } \varphi\right)
\left(\bigg[\frac{ (n+1) \e^{n+1}}{|\mathcal{H}(0)|}\bigg]^{\frac{1}{n+2}} \mu , \bigg[\frac{ (n+1) \e^{n+1}}{|\mathcal{H}(0)|}\bigg]^{\frac{1}{n+2}} \mu \zeta,\epsilon\right),$$
which has the expansion $\Gamma_\epsilon(\mu,\zeta)=\Gamma_0(\mu, \zeta)+o(1)$ as $\e \to 0^+$, uniformly for $\mu$ in compact subsets of
$(0,+\infty)$, in view of (\ref{solve1})-(\ref{solve2}), where the map $\Gamma_0: \mathbb{R} \times \mathbb{C} \to
\mathbb{R} \times \mathbb{C}$ is defined as
$$\Gamma_0(\mu,\zeta)= \left(\pi D_0 \mu^2-\frac{8 (n+1)^3 }{\mu^{{2\over n+1}}} \int_{\mathbb{R}^2} \frac{(|y|^2-1)|y+\zeta|^{\frac{2n}{n+1}}}{(1+|y|^2)^5},  \Gamma \zeta+\Upsilon \bar \zeta-{16 (n+1)^3  \over \pi \mu^{{2(n+2)\over n+1}}}  \int_{\mathbb{R}^2} {|y+\zeta|^{2n\over n+1} \bar y \over(1+|y|^2)^5}    \right).$$
We need to exhibit ``stable" zeroes of $\Gamma_0$ in $(0,+\infty)\times \mathbb{C}$, which will persist under $L^\infty-$small perturbations yielding to zeroes of $\Gamma_\epsilon$ as required. The easiest case is given by the point $(\mu_0,0)$, that solves $\Gamma_0=0$ for $\mu_0=({8 (n+1)^3 I_0 \over \pi  D_0})^{n+1\over 2(n+2)}>0$ in view of the assumption \eqref{D0} and  (see \eqref{1228})
$$I_0:=\int_{\mathbb{R}^2} \frac{(|y|^2-1)|y|^{\frac{2n}{n+1}}}{(1+|y|^2)^5}<0.$$
Regarding $\Gamma_0$ as a map from $\mathbb{R}^3$
into $\mathbb{R}^3$ and setting $\Gamma=\Gamma_1+i\Gamma_2$, $\Upsilon=\Upsilon_1+i\Upsilon_2$, we have that
$$D \Gamma_0(\mu_0,0)=\left(\begin{array}{ccc} \frac{2(n+2)}{n+1}\pi D_0 \mu_0 & 0 & 0 \\ 0 & \Gamma_1+\Upsilon_1 +{ n(2n+3)\over n+1} D_0 & \Upsilon_2-\Gamma_2 \\ 0 & \Gamma_2+\Upsilon_2 & \Gamma_1-\Upsilon_1 -{ n(2n+3)\over n+1} D_0 \end{array}\right)$$
in view of \eqref{1228} and
$$ \int_{\mathbb{R}^2}\frac{|y|^{\frac{2n}{n+1}}}{(1+|y|^2)^5}dy=\pi \int_0^\infty \frac{\rho^{\frac{n}{n+1}}}{(1+\rho)^5}d\rho=\pi I_5^{\frac{n}{n+1}}.$$
Since
$$\hbox{det }D\Gamma_0(\mu_0,0)=\frac{2(n+2)}{n+1}\pi D_0 \mu_0 \lf(|\Gamma |^2-\lf|\Upsilon+{ n(2n+3) \over n+1} D_0\rg|^2\rg) \not= 0$$
in view of assumption (\ref{nondegenracy}), the point $(\mu_0,0)$ is
an isolated zero of $\Gamma_0$ with non-trivial local index. Since
$D\Gamma_0(\mu_0,0)$ is an invertible matrix, there exists $\nu>0$ small so that
$|D\Gamma_0(\mu_0,0)(\mu-\mu_0,\zeta)|\geq \nu
|(\mu-\mu_0,\zeta)|$. By a Taylor expansion of $\Gamma_0$ we can find $r_0>0$ small so that
$$|\Gamma_\epsilon(\mu,\zeta)|=|\Gamma_0(\mu,\zeta)|+o(1) \geq \nu |(\mu-\mu_0,\zeta)|+O\left((\mu-\mu_0)^2+|\zeta|^2\right)+o(1) \geq  {\nu \over 2}|(\mu-\mu_0,\zeta)|$$
for all $(\mu,\zeta) \in \partial B_{r}(\mu_0,0)$ and all $r
\leq r_0$, for $\epsilon$ sufficiently small depending on
$r$. Then, the map $\Gamma_\epsilon$ has in $
B_{r_0}(\mu_0,0)$ well-defined degree for all $\epsilon$ small,
and it then coincides with the local index of $\Gamma_0$ at
$(\mu_0,0)$. In this way, the map $\Gamma_\epsilon$ has a zero of
the form $(\mu_\epsilon,\zeta_\epsilon)$ with $\mu_\epsilon \to \mu_0$
and $|\zeta_\epsilon|\to 0$ as $\epsilon \to 0$. Therefore, we have
solved \eqref{solve3} for $\de(\e)=[\frac{(n+1)\e^{n+1} }{|\mathcal{H}(0)|}]^{\frac{1}{n+2}} \mu_\epsilon$
and $a(\e)=\delta(\epsilon)\zeta_\epsilon$, and the corresponding $w_\e$ does solve (\ref{3}) and satisfy the required concentration property as
stated in Theorem \ref{main}.\qed
\end{proof}
\bremark
With some extra work, it is rather standard to see that \eqref{solve1} does hold in a $C^1-$sense. For $\zeta$ in a bounded set, by IFT we can find $\epsilon>0$ small so that the first equation in $\Gamma_\epsilon(\mu,\zeta)=0$ can be solved by $\mu(\epsilon,\zeta)$, depending continuously in $\zeta$, so that
$$\mu(\epsilon,\zeta) \to \mu(\zeta):= \left(\frac{8 (n+1)^3}{\pi D_0} \int_{\mathbb{R}^2} \frac{(|y|^2-1)|y+\zeta|^{\frac{2n}{n+1}}}{(1+|y|^2)^5}\right)^{\frac{n+1}{2(n+2)}}$$
as $\epsilon \to 0$. In Appendix D it is proved that $ \int_{\mathbb{R}^2} \frac{(|y|^2-1)|y+\zeta|^{\frac{2n}{n+1}}}{(1+|y|^2)^5}<0$ for all $\zeta \in \mathbb{C}$, yielding to $\mu(\zeta)>0$ when $D_0<0$. Plugging $\mu(\epsilon,\zeta)$ into the second equation in $\Gamma_\epsilon(\mu,\zeta)=0$ we are reduced to find a ``stable" zero of
$$\int_{\mathbb{R}^2} \frac{(|y|^2-1)|y+\zeta|^{\frac{2n}{n+1}}}{(1+|y|^2)^5}  \left(\bar \Upsilon \zeta +\bar \Gamma \bar \zeta\right)-2  D_0    \int_{\mathbb{R}^2} {|y+\zeta|^{2n\over n+1} y \over(1+|y|^2)^5}=0.$$
Notice that $\bar \Upsilon \zeta+\bar \Gamma  \bar \zeta$ acts in real notation as the multiplication for the matrix
$$A=\left(\begin{array}{cc} \re (\Gamma+ \Upsilon)& \im (\Upsilon-\Gamma) \\ -\im (\Gamma+ \Upsilon) & \re (\Upsilon-\Gamma) \end{array}\right).$$
Since by Appendix D we have that
$$\int_{\mathbb{R}^2} \frac{(|y|^2-1)|y+\zeta|^{\frac{2n}{n+1}}}{(1+|y|^2)^5}=f(|\zeta|),\qquad  \int_{\mathbb{R}^2} {|y+\zeta|^{2n\over n+1} y \over(1+|y|^2)^5}=g(|\zeta|)\zeta,$$
we can re-write the above equation as $A\zeta=\frac{2D_0 g(|\zeta|)}{f(|\zeta|)}\zeta$. Letting $(\lambda_1,e_1)$ be an eigen-pair of $A$ with $|e_1|=1$, we can find a solution $\zeta_0=|\zeta_0|e_1$ as soon as $|\zeta_0|\not= 0$ does solve $\frac{2D_0 g(|\zeta_0|)}{f(|\zeta_0|)}=\lambda_1$. Since by Appendix D we know that $f<0<g$, we can find solutions $(\mu_\epsilon,\zeta_\epsilon)$ of $\Gamma_\epsilon(\mu,\zeta)=0$ with $\zeta_\epsilon$ bifurcating from $\zeta_0 \not= 0$ as soon as one of the eigenvalues of $A$ positive and belongs to $\frac{2D_0 g}{f}(0,+\infty)$. In particular, by \eqref{1228}-\eqref{1902} and \eqref{1903}-\eqref{1904} we have that
$$\frac{g(0)}{f(0)}=-\frac{(2n+3)(3n+1)}{4(n+1)},\qquad \frac{g(|\zeta|)}{f(|\zeta|)} \to -\frac{51}{356} \hbox{ as }|\zeta|\to \infty,$$
and the condition above is fullfilled if one of the eigenvalues of $A$ lies in $(\frac{51}{178}|D_0|, \frac{(2n+3)(3n+1)}{2(n+1)}|D_0|)$.
\eremark

%%%%%%%%%%%%%%%%%%%%%%%%%%%%%%%%%%%%%%%%%%%%%%%%%%%%%%%%%%%%%%%%%%%%%%%%%%%%%%%%%%%%%%%%%%%%%%%%%%%%%%%%%%%%%%%%%%%%%%%%%%%%%%%%%%%%%%%
%%%%%%%%%%%%%%%%%%%%%%%%%%%%%%%%%%%%%%%%%%%%%%%%%%%%%%%%%%%%%%%%%%%%%%%%%%%%%%%%%%%%%%%%%%%%%%%%%%%%%%%%%%%%%%%%%%%%%%%%%%%%%%%%%%%%%%%
%\newpage

\section{Examples and comments}\label{examples}
In this section, we will discuss the validity of \eqref{pc}-\eqref{nondegenracy} by providing some  examples. Recall that in Theorem \ref{main} we were implicitly assuming that $\{p_1,\dots,p_N\} \subset \Omega$ and denoting for simplicity the concentration point $p$ as $0$. The assumption $\{p_1,\dots,p_N\} \subset \Omega$ simplifies the global construction in $\tilde \Omega$ of $\mathcal{H}$  but \eqref{pc}-\eqref{nondegenracy} just require the local existence for such $\mathcal{H}$ at $0$ as well as for $\sigma_0$ and $H^*$. In this respect, the only relevant assumption is that the concentration point lies in $\Omega$, and so we will provide examples with $0 \in \{\tilde p_1,\dots,\tilde p_N\} \subset \bar \Omega$. To be more precise, let us explain the general strategy we will adopt below. Since we are in a doubly-periodic setting, the configuration of the vortex points has to be periodic in $\bar \Omega$: for all $j=1,\dots,N$ the points $(\tilde p_j +\omega_1 \mathbb{Z}+\omega_2 \mathbb{Z})\cap \bar \Omega$ belong to $\{\tilde p_1,\dots,\tilde p_N\}$ and have all the same multiplicity. Then, we can find $J \subset \{1,\dots,N\}$ so that the points $\{\tilde p_j:\, j \in J\}$ are all non-zero, distinct modulo $\omega_1 \mathbb{Z}+\omega_2 \mathbb{Z}$ and $\left(\{\tilde p_j:\, j \in J\}+\omega_1 \mathbb{Z}+\omega_2 \mathbb{Z}\right) \cap \bar \Omega=\{\tilde p_1,\dots, \tilde p_N\}\setminus \{0\}$. Take now a translation vector $\tau \in \Omega$ so that $\{\tilde p_1+\tau,\dots,\tilde p_N+\tau\}\cap \partial \Omega=\emptyset$, or equivalently $\left(\{\tilde p_1,\dots,\tilde p_N\}+\tau+ \omega_1 \mathbb{Z}+\omega_2 \mathbb{Z}\right)\cap \partial \Omega=\emptyset$. Then, it follows that  $\left(\tilde p_j+\tau+\omega_1 \mathbb{Z}+\omega_2 \mathbb{Z}\right)\cap \Omega$ is composed by a single point $p_j$, for all $j=1,\dots,N$. The idea is to apply Theorem \ref{main}, as formulated in Remark \ref{0923}, to the translated vortex configuration $\{ \tau\} \cup \{p_j:j \in J\}\subset \Omega$ with $\tau$ as concentration point. The validity of \eqref{pc}-\eqref{nondegenracy} in the translated situation will follow by appropriate assumptions on $\{\tilde p_1,\dots,\tilde p_N\}$.

\medskip \noindent Before stating our first result, let us introduce the notion of even vortex configuration: $-\tilde p_j\in \{\tilde p_1,\dots,\tilde p_N\}+\omega_1 \mathbb{Z}+\omega_2 \mathbb{Z}$ with the same multiplicity of $\tilde p_j$, for all $j=1,\dots,N$. In the periodic case, notice that $\{\tilde p_j: j \in J\}$ is still an even configuration. The validity of \eqref{pc} is discussed in the following:
\begin{prop} \label{propp1}Assume $n$ is even and the periodic vortex configuration is even with $0 \in \{\tilde p_1,\dots,\tilde p_N\}$. Let $\mathcal{H}^\tau$ be the function corresponding to $p=\tau$ and remaining vortex points $\{p_j:\, j \in J\}\subset \Omega$, as given in Remark \ref{1149}.  Then, there holds
$$\frac{d^k \mathcal{H}^\tau}{dz^k}(\tau)=0$$
for all odd number $k$.
\end{prop}
\begin{proof} Since $-\Omega=\Omega$ and the periodic vortex configuration $\{\tilde p_1,\dots,\tilde p_N\}$ is even, we have that $G(z)$, $H(z)$ and $e^{-4\pi \sum_{j \in J} n_j G(z,\tilde p_j)}$ are even functions in view of $G(z,p)=G(z-p,0)$. So, it follows that  $e^{4\pi(n+2)H(z-\tau)-4\pi \sum_{j\in J} n_j G(z,\tilde p_j+\tau)}=e^{4\pi(n+2)H(z-\tau)-4\pi \sum_{j\in J} n_j G(z,p_j)}$ takes the same value at $\pm z+\tau$ for all $z \in \Omega$. The function $\mathcal{H}^\tau$ satisfies $|\mathcal{H}^\tau|(z+\tau)=|\mathcal{H}^\tau|(-z+\tau)$ for all $z \in \Omega$, and then $\mathcal{H}^\tau(z+\tau)=\mathcal{H}^\tau(-z+\tau)$ for all $z$ since $\mathcal{H}^\tau$ is an holomorphic function. By differentiating $k-$times at $\tau$, it yields to $\frac{d^k \mathcal{H}^\tau}{dz^k}(\tau)=0$ when $k$ is odd.\qed
\end{proof}

\medskip \noindent The discussion of \eqref{D0} is more interesting and will make use of the Weierstrass elliptic function $\wp$ to represent $D_0$ in case of an even periodic vortex configuration. Furthermore, when $\Omega$ is a rectangle, the points $p_j$'s are half-periods and all the multiplicities are even numbers, by some ideas in \cite{CLW} we will show that assumption \eqref{D0} holds if and only if $\frac{n_3}{2}$ is an odd number, where $n_3$ is the multiplicity of the half-period $\frac{\omega_1+\omega_2}{2}$. Due to the presence of
high order derivatives ($2(n+1)$th order) in \eqref{nondegenracy}, we will verify the validity of the ``non-degeneracy" condition in the simplest case $n=n_3=2$ and $\Omega$ a square torus. As we will see, the validity of \eqref{nondegenracy} is just a computational matter which could be carried out in very generality for each case of interest.

\medskip \noindent We have the following representation formula:
\begin{prop} \label{1014}
Assume that the periodic vortex configuration is even with $0 \in \{\tilde p_1,\dots,\tilde p_N\}$, and $n_j$ is even when $\tilde p_j \in \{{\omega_1\over 2}, {\omega_2\over 2},{\omega_1+\omega_2\over 2}\}$. Let $D_0^\tau$ be the coefficient corresponding to $p=\tau$ and remaining vortex points $\{p_j:\, j \in J\}\subset \Omega$, as given in Theorem \ref{main}.  Then, for $\tau $small we have that $D_0^\tau$ is given by \eqref{1846}, and does not depend on $\tau$.
\end{prop}
\begin{proof} The Weierstrass elliptic function
$$\wp(z)=\frac{1}{z^2}+\sum_{(n,m)\not=(0,0)} \left( \frac{1}{(z+n\omega_1+m\omega_2)^2}-\frac{1}{(n\omega_1+m\omega_2)^2}\right)$$
is a doubly-periodic meromorphic function with a single pole in $\Omega$ at $0$ of multiplicity $2$. Moreover, the only branching points of $\wp$ are simple and given by the three half-periods ${\omega_1\over 2}$, ${\omega_2\over 2}$ and $\frac{\omega_3}{2}={\omega_1+\omega_2\over 2}$, i.e. $\wp'(\frac{\omega_j}{2})=0$ and $\wp''(\frac{\omega_j}{2})\not=0$ for $j=1,2,3$.
For $p\in \bar \Omega \setminus \{ 0\}$, note that $2\pi[2G(z,0)-G(z,p)-G(z,-p)]$ is a doubly-periodic harmonic
function in $\Omega$ with a singular behavior $-2\log|z|$ at $z=0$. Moreover, it behaves like $\log|z-p|$ at $z=p$ and $\log|z+p|$ at $z=-p$ when $p \not=\frac{\omega_1}{2},\frac{\omega_2}{2},\frac{\omega_3}{2}$, and like $2\log|z-p|$ if $p\in \{\frac{\omega_1}{2},\frac{\omega_2}{2},\frac{\omega_3}{2}\}$. Thus, we have that
$$2\pi[2G(z,0)-G(z,p)-G(z,-p)]=\log|\wp(z)-\wp(p)|+\text{const.}$$
no matter $p$ is an half-period or not, in view of $\wp(p)=\wp(-p)$, $\wp'(p)=-\wp'(-p)\not=0$ if $p \not=\frac{\omega_1}{2},\frac{\omega_2}{2},\frac{\omega_3}{2}$ and $\wp'(p)=0$, $\wp''(p)\not=0$ if $p\in \{\frac{\omega_1}{2},\frac{\omega_2}{2},\frac{\omega_3}{2}\}$. Since the periodic vortex configuration is even, take $I$ as the minimal subset of $J$ so that $\left(\{\tilde p_k,-\tilde p_k: \, k \in I\}+\omega_1 \mathbb{Z}+\omega_2 \mathbb{Z}\right) \cap \{\tilde p_j:\,j \in J\}=\{\tilde p_j:\,j \in J\}$ and
$$\hat n_k=\left\{ \begin{array}{ll} \frac{n_k}{2} &\hbox{if }\tilde p_k \hbox{ is an half-period}\\
n_k& \hbox{otherwise}. \end{array}\right.$$
Letting $N=n+\sum_{j \in J}  n_j$ and $u_0(z)=-4\pi nG(z,0)-4\pi \sum_{j \in J}n_j G(z,\tilde p_j)$, assumption \eqref{balance} implies that
$$u_0+8\pi(n+1)G(z,0)=4\pi\sum_{k \in I} \hat n_k [2G(z,0)- G(z,\tilde p_k)-G(z,-\tilde p_k)],$$
yielding to
\begin{equation*}
e^{u_0+8\pi(n+1)G(z,0)}=  \hbox{const.}\:  \big| \prod_{k \in I} (\wp(z)-\wp(\tilde p_k))^{\hat n_k} \big|^2.
\end{equation*}
The additional assumption that $n_j$ is even when $\tilde p_j$ is an half-period is crucial to have $(\wp(z)-\wp(\tilde p_j))^{\hat n_j}$ as a single-valued function. The function \begin{equation} \label{explicitH}
\mathcal{H}_0(z)= \lambda_0 \prod_{k \in I} (\wp(z)-\wp(\tilde p_k))^{\hat n_k},\quad \lambda_0=e^{2\pi(n+2)H(0)-2\pi \sum_{j \in J} n_j G(0,\tilde p_j)}
\end{equation}
is an elliptic function with a single pole at $0$ of zero residue, which satisfies 
\begin{equation} \label{1328}
|\mathcal{H}_0|^2=e^{u_0+8\pi(n+1)G(z,0)}.
\end{equation}
Then 
\begin{equation} \label{1010}
\sigma_0(z)=-\left(\int^z   \mathcal{H}_0(w) dw \right)^{-1}
=-\lambda_0^{-1}\left(\int^z  \prod_{k \in I} (\wp(w)-\wp(\tilde p_k))^{\hat n_k} dw \right)^{-1}
\end{equation}
is a well-defined meromorphic function in $2\Omega$ which satisfies
\begin{equation} \label{1345}
\Big|   \Big( \frac{1}{\sigma_0} \Big)'(z)  \Big|^2=|\mathcal{H}_0|^2(z)
= e^{u_0+8\pi(n+1)G(z,0)}.
\end{equation}
Switching now to the translated vortex configuration $\{\tau\}\cup \{p_j:\,j \in J\}$, let us first notice that the total multiplicity is still $N$, and introduce $u_0^\tau=u_0(z-\tau)=-4\pi nG(z,\tau)-4\pi \sum_{j \in J}n_j G(z,p_j)$. We have that $\mathcal{H}_0^\tau(z)=\mathcal{H}_0(z-\tau)$ is a meromorphic function in $\Omega$ with
$$|\mathcal{H}_0^\tau|^2=e^{u_0^\tau+8\pi(n+1)G(z,\tau)}$$
in view of (\ref{1328}). Since such a function $\mathcal{H}_0^\tau$ is unique up to rotations, we can assume that $\mathcal{H}_0^\tau$ coincides with the function $\mathcal{H}_0$ corresponding to $p=\tau$ and remaining vortex points $\{p_j:\, j \in J\}\subset \Omega$, as given in Theorem \ref{main}. Setting $\mathcal{H}(z)=z^{n+2}\mathcal{H}_0(z)$, we also have that
\begin{equation} \label{1344}
\mathcal{H}^\tau(z)=\mathcal{H} (z-\tau)
\end{equation}
for all $z \in \Omega$. Letting 
$$\sigma_0^\tau(z)=-\left(\int^z   \mathcal{H}_0^\tau(w) dw \right)^{-1}$$
with the correct choice of the constant in the integration$\int^z$, we easily deduce that
\begin{equation} \label{0935}
\sigma_0^\tau(z)=\sigma_0(z-\tau)
\end{equation}
for all $z \in \Omega$ in view of $(\frac{1}{\sigma_0^\tau})'(z)=(\frac{1}{\sigma_0})'(z-\tau)$. Since $(\sigma_0^\tau)^{-1}(B_\rho(0))-\tau=(\sigma_0)^{-1}(B_\rho(0))$ in view of \eqref{0935}, according to \eqref{D0} let us re-write $D_0^\tau$ as
\begin{eqnarray*}
\pi D_0^\tau&=&\int_{\Omega \setminus (\sigma_0^\tau)^{-1} (B_\rho(0))}  e^{u_0^\tau+8\pi(n+1)G(z,\tau)} 
-\int_{\mathbb{R}^2 \setminus B_\rho(0)}\frac{n+1}{|y|^4}\\
&=&
\int_{(\Omega-\tau) \setminus (\sigma_0)^{-1}(B_\rho(0))}   e^{u_0+8\pi(n+1)G(z,0)}
-\int_{\mathbb{R}^2 \setminus B_\rho(0)}\frac{n+1}{|y|^4}\\
&=&\int_{\Omega \setminus (\sigma_0)^{-1}(B_\rho(0))}   e^{u_0+8\pi(n+1)G(z,0)}
-\int_{\mathbb{R}^2 \setminus B_\rho(0)}\frac{n+1}{|y|^4}
\end{eqnarray*}
by the double-periodicity of $e^{u_0+8\pi(n+1)G(z,0)}$, once we assume for $\tau$ small that $(\sigma_0)^{-1}(B_\rho(0)) \subset \Omega \cap (\Omega-\tau)$. By \eqref{1345} and the change of variable $z \to \frac{1}{\sigma_0}(z)$ we get that
\begin{eqnarray}
\pi D_0^\tau&=&\pi D_0= \int_{\Omega \setminus (\sigma_0)^{-1}(B_\rho(0)) }  \Big|\left(\frac{1}{\sigma_0}\right)'\Big|^2
-\int_{\mathbb{R}^2 \setminus B_\rho(0)}\frac{n+1}{|y|^4} \nonumber \\
&=&\hbox{Area } \left[ \frac{1}{\sigma_0}\left(\Omega \setminus \sigma_0^{-1} (B_\rho(0)) \right) \right]  - (n+1) \hbox{Area} \left( B_{\frac{1}{\rho}}(0) \right). \label{1846}
\end{eqnarray}
By the Cauchy argument principle the number of pre-images in $\Omega \setminus \sigma_0^{-1}(B_\rho(0))$ through the map $\frac{1}{\sigma_0}$ is constant for all values in each connected component of $\mathbb{C} \setminus \left(\frac{1}{\sigma_0}(\partial \Omega) \cup \partial B_{\frac{1}{\rho}}(0) \right)$, and the area of each of these components has to be counted in \eqref{1846} according to the multiplicity of pre-images. \qed
\end{proof}

\medskip \noindent Thanks to \eqref{1846}, we can now discuss the validity of \eqref{D0}.
\begin{prop} \label{propp2}
Let $\Om$ be a rectangle, and assume that the vortex configuration is the periodic one generated by $\{0,{\omega_1\over
2},{\omega_2 \over 2},{\omega_1+\omega_2 \over 2}\}$ with even multiplicities $n,n_1,n_2,n_3 \geq 0$. Suppose that
\begin{equation}\label{balanceex}
\frac{n_1}{2}+\frac{n_2}{2}+\frac{n_3}{2}=\frac{n}{2}+1.
\end{equation}
Given $D_0^\tau$ as in Propostion \ref{1014}, then $D_0^\tau<0$ $(>0)$when $\frac{n_3}{2}$ is odd $(\hbox{even})$.
\end{prop}
\begin{proof}
The balance condition \eqref{balance} is satisfied in view of \eqref{balanceex}. Let $\tilde p_1={\omega_1\over 2}$, $\tilde p_2={\omega_2\over 2}$ and $\tilde p_3={\omega_1+\omega_2\over 2}$ be the three half-periods. When $\Omega$ is a rectangle, the function $\wp$ takes real values on $\partial \Omega$ and $\wp''(\tilde p_j)>0$ for $j=1,2$, $\wp''(\tilde p_3)<0$. As a consequence, we have that
\begin{equation} \label{2026}
\wp(\tilde p_1)-\wp(z),\:  \wp(z)-\wp(\tilde p_2),\: \wp(\pm \tilde p_1+it)-\wp(\tilde p_3) ,\:  \wp(\tilde p_3)-\wp(\pm \tilde p_2+t) \ge 0 \end{equation}
for all $z\in\ptl\Om$ and $t\in\R$. Write $\sigma_0(z)$ in \eqref{1010} as
$$ \sigma_0(z)=(-1)^{\frac{n+n_2}{2}}\lambda_0^{-1} \left(\int^z  (\wp(\tilde p_1)-\wp(w))^{\frac{n_1}{2}}(\wp(w)-\wp(\tilde p_2))^{\frac{n_2}{2}}(\wp(\tilde p_3)-\wp(w))^{\frac{n_3}{2}} dw \right)^{-1}$$
in view of \eqref{balanceex}. Since
$$\frac{d}{dt}\left[ \frac{(-1)^{\frac{n+n_2}{2}}}{\sigma_0 (\pm \tilde p_2+t)} \right]= \lambda_0 (\wp(\tilde p_1)-\wp(\pm \tilde p_2+t))^{\frac{n_1}{2}}(\wp(\pm \tilde p_2+t)-\wp(\tilde p_2))^{\frac{n_2}{2}}(\wp(\tilde p_3)-\wp(\pm \tilde p_2+t))^{\frac{n_3}{2}} \geq 0$$
in view of \eqref{2026}, the function $\frac{(-1)^{\frac{n+n_2}{2}}}{\sigma_0}$ maps the horizontal sides of $\partial \Omega$ into horizontal segments with same orientation. In the same way, the vertical sides of $\partial \Omega$ are mapped into vertical segments with same/opposite orientation depending on whether $\frac{n_3}{2}$ is an even/odd number. So, $T:=\frac{(-1)^{\frac{n+n_2}{2}}}{\sigma_0} (\partial \Omega)$ is still a rectangle with same/opposite orientation and $\frac{(-1)^{\frac{n+n_2}{2}}}{\sigma_0 (\tilde p_3)}$ is the right upper/lower corner of $T$ depending on whether $\frac{n_3}{2}$ is an even/odd number. For $\rho$ small, we then have that $\mathbb{C} \setminus \left(\frac{1}{\sigma_0}(\partial \Omega) \cup \partial B_{\frac{1}{\rho}}(0) \right)$ has three connected components: the interior $\Omega'$ of $(-1)^{\frac{n+n_2}{2}} T$, $B_{\frac{1}{\rho}}(0) \setminus \overline{\Omega'}$ and $\mathbb{C} \setminus \overline{B_{\frac{1}{\rho}}(0)}$. By Lemma \ref{gomme} we have that values in $B_{\frac{1}{\rho}}(0) \setminus \overline{\Omega'}$, $\mathbb{C} \setminus \overline{B_{\frac{1}{\rho}}(0)}$ have exactly $n+1$, $0$ pre-images in $\Omega \setminus \sigma_0^{-1}(B_\rho(0))$ through the map $\frac{1}{\sigma_0}$, respectively. By \eqref{1846} we have that $\pi D_0^\tau=[k - (n+1)] \hbox{Area} (\Omega')$, where $k$ is the
number of pre-images corresponding to values in $\Omega'$.

\medskip \noindent Since $\wp(z)-\wp(\tilde p_3)={\wp''(\tilde p_3)\over 2}(z-\tilde p_3)^2+O(|z-\tilde p_3|^3)$ as $z \to \tilde p_3$, we obtain that
$$\lf[\frac{(-1)^{\frac{n+n_2}{2}}}{\sigma_0}\rg]'(z)= \mu (z-\tilde p_3)^{n_3}+O(|z-\tilde p_3|^{n_3+1})$$
and
$$\frac{(-1)^{\frac{n+n_2}{2}}}{\sigma_0(z)}-\frac{(-1)^{\frac{n+n_2}{2}}}{\sigma_0(\tilde p_3)}=\mu {(z-\tilde p_3)^{n_3+1}\over
n_3+1}+O(|z-\tilde p_3|^{n_3+2})$$
as $z \to \tilde p_3$, where $\mu:=\lambda_0 \lf(-{\wp''(\tilde p_3)\over
2}\rg)^{\frac{n_3}{2}} [\wp(\tilde p_1)-\wp(\tilde p_3)]^{\frac{n_1}{2}}[\wp(\tilde p_3)-\wp(\tilde p_2)]^{\frac{n_2}{2}}>0$. When $\frac{n_3}{2}$ is an odd number, $\frac{(-1)^{\frac{n+n_2}{2}}}{\sigma_0 (\tilde p_3)}$ is the right lower corner of $T$ and the function $\frac{(-1)^{\frac{n+n_2}{2}}}{\sigma_0}$ maps $\{z=\tilde p_3+\rho e^{i\theta}\mid
\pi\le\theta\le {3\pi\over 2}, 0\le\rho<\rho_0\}$ onto a
region whose part inside/outside $T$ is covered ${n_3-2\over 4}$/${n_3-2\over 4}+1$ times, respectively, in view of
$$(n_3+1)\pi\le(n_3+1)\theta\le(n_3+1){3\pi\over
2}=(n_3+1)\pi+2\pi {n_3-2\over 4}+\pi+{\pi\over 2}.$$
Hence, near $\tilde p_3$ the map $\frac{1}{\sigma_0}$ covers ${n_3-2\over 4}$/${n_3-2\over 4}+1$ times the interior/exterior part of $\Omega'$ near $\frac{1}{\sigma_0(\tilde p_3)}$. Since $\frac{1}{\sigma_0}$ covers $n+1$ times every values in
$B_{\frac{1}{\rho}}(0)\sm \overline{\Omega'}$, there should be
$n-{n_3-2\over 4}$ distinct points $x\in \Om\sm \sigma_0^{-1}(B_\rho(0))$, away from $\tilde p_1,\tilde p_2, \tilde p_3$, so that $\sigma_0(x)=\sigma_0(\tilde p_3)$. Since
$\sigma_0'(x) \not= 0$ if $x\not= \tilde p_1,\tilde p_2,\tilde p_3$, it follows that around any such $x$
$\frac{1}{\sigma_0}$ is a local homeomorphism, and then $\frac{1}{\sigma_0}$ covers exactly $n$/$n+1$ times the interior/exterior part of $\Omega'$ near $\frac{1}{\sigma_0(\tilde p_3)}$. Hence, it
follows that $k=n$ and $\pi D_0^\tau=-\hbox{Area} (\Omega')<0$. When $\frac{n_3}{2}$ is even, in a similar way we get that $k=n+2$ and $\pi D_0^\tau=\hbox{Area} (\Omega')>0$.\qed
\end{proof}

\medskip \noindent Now, to discuss \eqref{nondegenracy} we further restrict the attention to the case $n=n_3=2$ to get
\begin{prop} \label{propp3}
Let $\Om$ be a square of side $a$, $a>0$, and assume that the vortex configuration is the periodic one generated by $\{0,{a\over
2},{ia \over 2},{a+ia \over 2}\}$ with multiplicities $2,n_1,n_2,2$ and $(n_1,n_2)=(2,0)$ (or viceversa).
Then, for $\tau \in \Omega$ assumption \eqref{nondegenracy} does hold for the vortex configuration $\{\tau\}\cup \{p_j:\, j \in J\}\subset \Omega$.
\end{prop}
\begin{proof}
We are restricting the attention to the cases $(n_1,n_2)=(2,0),\, (0,2)$ for they are the only possibilities to have even multiplicities satisfying \eqref{balanceex} for $2,n_1,n_2,2$. Letting $\tilde p_1={a \over 2}$, $\tilde p_2={ia \over 2}$ and $\tilde p_3={a+ia \over 2}$ be the three half-periods, the ``non-degeneracy condition" reads as
\begin{equation}\label{ceae2}
\bigg| 3 (\ml{H}^\tau)''(\tau )f_3'(\tau)+\ml{H}^\tau(\tau) f_3'''(\tau) \bigg|\ne
\lf|{6\pi\over a^2} \overline{b_{3}} (\mathcal{H}^\tau)''(\tau)-{28\over 3} D_0^\tau \rg|
\end{equation}
in view of $(\ml{H}^\tau)'(\tau)=(\ml{H}^\tau)'''(\tau)=0$ by Proposition \ref{propp1},
%in view of $\mathcal{H}(0)=\lambda_0$ (see \eqref{explicitH}) and $\ml{H}'(0)=\ml{H}'''(0)=0$, 
where
$$f_l(z)=\frac{1}{l!}\frac{d^l}{dw^l} \left[2\log \frac{w-q_0^\tau(z)}{(q_0^\tau)^{-1}(w)-z}+4\pi
H^*(z-(q_0^\tau)^{-1}(w))\right](0)\:,\quad b_l=\frac{1}{l!}\frac{d^l (q_0^\tau)^{-1}}{dw^l}(0).$$
Since  $\sigma_0^\tau(z)=\sigma_0(z-\tau)$ by \eqref{0935}, we deduce that $q_0^\tau(z)=q_0(z-\tau)$ and $(q_0^\tau)^{-1}=\tau +q_0^{-1}$, where $q_0=z [\frac{\sigma_0(z)}
{z^{n+1}}]^{\frac{1}{n+1}}$ is defined out of $\sigma_0$ as in Appendix A. Since $\ml{H}^\tau(z)=\ml{H}(z-\tau)$ in view of (\ref{1344}), by \eqref{1846} the ``non-degeneracy condition"  \eqref{ceae2} gets re-written in the original variables as:
\begin{equation}\label{ceae21357}
\bigg| 3 \ml{H}''(0)f_3'(0)+\lambda_0 f_3'''(0) \bigg|\ne
\lf|{6\pi\over a^2} \overline{b_{3}} \mathcal{H}''(0)-{28\over 3} D_0 \rg|
\end{equation}
in view of $\mathcal{H}(0)=\lambda_0$ (see \eqref{explicitH}),  where
$$f_l(z)=\frac{1}{l!}\frac{d^l}{dw^l} \left[2\log \frac{w-q_0(z)}{q_0^{-1}(w)-z}+4\pi
H^*(z-q_0^{-1}(w))\right](0)\:,\quad b_l=\frac{1}{l!}\frac{d^l q_0^{-1}}{dw^l}(0).$$
Since ${d^k \ml{H} \over dz^k}(0)=0$ for all odd $k\in\N$, we have that
\begin{eqnarray*}
\frac{z^3}{\sigma_0(z)}=\frac{\lambda_0}{3}+\frac{\ml{H}''(0)}{2}  z^2-\frac{\ml{H}^{(4)}(0)}{24}z^4 -\frac{\ml{H}^{(6)}(0)}{2160 }z^6+
O(z^8),
\end{eqnarray*}
and then
$$\sigma_0(z)=\frac{3}{\lambda_0}z^3 -\frac{9\ml{H}''(0)}{2\lambda_0^2 }z^5+O(z^7),\quad
q_0(z)=\frac{3^{\frac{1}{3}}}{\lambda_0^{\frac{1}{3}}}z-\frac{3^{\frac{1}{3}}\ml{H}''(0)}{2 \lambda_0^{\frac{4}{3}}}z^3+O(z^5),\quad
q_0^{-1}(w)=\frac{\lambda_0^{\frac{1}{3}}}{3^{\frac{1}{3}}}w+\frac{\ml{H}''(0)}{6}w^3+O(w^5)$$
as $z,w \to 0$. Direct computation shows that $b_3=\frac{\ml{H}''(0)}{6}$
and
\begin{eqnarray*}
f_3(z)&=&-\frac{2}{3\sigma_0(z)}+\frac{2\lambda_0 }{9 z^3}+\frac{2b_3}{z}-\frac{2\pi \lambda_0}{9}(H^*)'''(z)-4\pi b_3 (H^*)'(z)\\
&=&\frac{\ml{H}^{(4)}(0)}{36}z +\frac{\ml{H}^{(6)}(0)}{3240}z^3-\frac{2\pi \lambda_0}{9}(H^*)'''(z)-\frac{2\pi}{3} \ml{H}''(0) (H^*)'(z)+O(z^5)
\end{eqnarray*}
as $z \to 0$. Since then
$$f_3'(0)=\frac{\ml{H}^{(4)}(0)}{36}-\frac{2\pi \lambda_0}{9}(H^*)^{(4)}(0)-\frac{2\pi}{3} \ml{H}''(0) (H^*)''(0),\quad f_3'''(0)=\frac{\ml{H}^{(6)}(0)}{540} -\frac{2\pi \lambda_0}{9}(H^*)^{(6)}(0)-\frac{2\pi}{3} \ml{H}''(0) (H^*)^{(4)}(0),$$
condition \eqref{ceae21357} is equivalent to
\begin{eqnarray*}
&&\bigg| \frac{\ml{H}''(0)\ml{H}^{(4)}(0)}{12}+\frac{\lambda_0 \ml{H}^{(6)}(0)}{540}
-2\pi (\ml{H}''(0))^2 (H^*)''(0)-\frac{4 \pi \lambda_0}{3}  \ml{H}''(0) (H^*)^{(4)}(0)-\frac{2\pi \lambda_0^2}{9}(H^*)^{(6)}(0)\bigg|\\
&&\ne
\lf|{\pi\over a^2}  |\mathcal{H}''(0)|^2-{28\over 3} D_0\rg|.
\end{eqnarray*}
By the explicit expression \eqref{explicitH} of $\ml{H}_0$ we have that
$$\mathcal{H}(z)= \lambda_0 z^4 (\wp(z)-\wp(\tilde p_1)) (\wp(z)-\wp(\tilde p_3)).$$
Replacing $\mathcal{H}$ with $\frac{\mathcal{H}}{\lambda_0}$, we can assume $\lambda_0=1$ and simply study the stronger condition
\begin{eqnarray}\label{yyy}
\hspace{-0.2cm} \bigg| \frac{\ml{H}''(0)\ml{H}^{(4)}(0)}{4}+\frac{\ml{H}^{(6)}(0)}{180}
-6\pi (\ml{H}''(0))^2 (H^*)''(0)- 4 \pi \ml{H}''(0) (H^*)^{(4)}(0)-\frac{2\pi}{3}(H^*)^{(6)}(0)\bigg|< {3 \pi\over a^2}  |\mathcal{H}''(0)|^2\end{eqnarray}
in view of Proposition \ref{1014} and \eqref{1846}. Letting $G_l=\displaystyle \sum_{(n,m) \not= (0,0)}{1\over (n\omega_1+m\omega_2)^l}$, $l\geq 3$, be the Eisenstein series, the Laurent expansion of $\wp$ near $0$ simply re-writes as
$$\wp(z)={1\over
z^2}+\sum_{l=1}^\infty(2l+1)G_{2l+2}z^{2l},$$
and then
\begin{eqnarray*}
\mathcal{H}(z)=1-(\wp(\tilde p_1)+\wp(\tilde p_3))z^2+\left(\wp(\tilde p_1)\wp(\tilde p_3)+6 G_4 \right) z^4
+\left(10 G_6 -3G_4 \wp(\tilde p_1)-3G_4 \wp(\tilde p_3)\right) z^6+O(z^8)
\end{eqnarray*}
as $z \to 0$. Letting $e_j=\wp(\tilde p_j)$ for $j=1,2,3$, recall that
\begin{equation} \label{propej}
e_2<e_3\le0<e_1,\quad e_1+e_2+e_3=0,\quad 15 G_4=-(e_1e_2+e_1e_3+e_2e_3),\quad 35 G_6=e_1 e_2 e_3,
\end{equation}
with $e_3=0$ if and only if $\Omega$ is a square (see \cite{AbSte}). By the expansion of $\ml{H}$ and \eqref{propej}, we deduce that
$$\ml{H}''(0)=2 e_2,\:\ml{H}^{(4)}(0)=24(e_1 e_3 +6 G_4),\:\ml{H}^{(6)}(0)= 720(10 G_6 +3G_4 e_2),$$
and condition \eqref{yyy} gets re-written as
\begin{eqnarray} \label{ceae3}
\bigg| 460 G_6 +84  G_4e_2-24\pi e_2^2 (H^*)''(0)-8\pi e_2 (H^*)^{(4)}(0) -\frac{2\pi}{3} (H^*)^{(6)}(0)\bigg|<
{12\pi\over a^2}   e_2^2
\end{eqnarray}
in view of \eqref{propej}.

\medskip \noindent From an explicit formula for the Green's function (see \cite{ChO}) we have that
\begin{equation*}
\begin{split}
H(z)-{|z|^2\over 4|\Omega|}
=\re\lf(-{z^2\over 4 a^2}+{iz\over 2a}+{1\over 12}\rg)-\frac{1}{2\pi}\log\lf|{1-e\lf(\frac{z}{a}\rg)\over z}
\times\prod_{k=1}^\infty\lf(1-e\lf(\frac{kai+z}{a}\rg)\rg)\lf(1-e\lf(\frac{kai-z}{a}\rg)\rg)\rg|,
\end{split}
\end{equation*}
where $e(z)=e^{2\pi iz}$, yielding to
\begin{equation*}
H^*(z)= -{z^2\over 4 a^2}+{iz\over 2a}+{1\over 12}-\frac{1}{2\pi}
\log \left[\lf({1-e\lf(\frac{z}{a}\rg)\over z}\rg)
\times\prod_{k=1}^\infty\lf(1-e\lf(\frac{kai+z}{a}\rg)\rg)\lf(1-e\lf(\frac{kai-z}{a}\rg)\rg) \right].\end{equation*}
Direct, but tedious, computations show that
\begin{eqnarray*}
&&(H^*)''(0)=-{1 \over 2 a^2}+{\pi\over 6a^2}-{4\pi \over a^2}\sum_{k=1}^\infty \lambda_k(\lambda_k+1),\quad (H^*)^{(4)}(0)={\pi^3\over 15a^4}+{16\pi^3\over a^4}\sum_{k=1}^\infty \lambda_k(\lambda_k+1)(6\la_k^2+6\la_k+1)\\
&& (H^*)^{(6)}(0)={8\pi^5\over 63a^6}-{64\pi^5\over a^6}\sum_{k=1}^\infty \lambda_k(\lambda_k+1)(120\la_k^4+240\la_k^3+150\la_k^2+30\la_k+1),
\end{eqnarray*}
where $\la_k:={1\over e^{2\pi k}-1}$. On a square torus the Green function $G(z,0)$ has an additional symmetry, the invariance under $\frac{\pi}{2}-$rotations. Therefore, $H^*(iz)=H^*(z)$ for all $z\in\Om$, and then $(H^*)''(0)=(H^*)^{(6)}(0)=0$. Since $e_3=G_6=0$, condition \eqref{ceae3} becomes
\begin{eqnarray} \label{ceae4}
\bigg| \frac{28}{5}e_1^2 -8\pi  (H^*)^{(4)}(0) \bigg|< {12\pi\over a^2}   e_1
\end{eqnarray}
in view of \eqref{propej} and $e_1=-e_2>0$. From the study of the Weierstrass function $\wp$ it is known that (see \cite{Ap})
\begin{equation*}
\sum_{(n,m)\ne
(0,0)} {1\over (n+m\tau)^4}={\pi^4\over
45}+{16\pi^4\over 3}\sum_{m,k=1}^\infty k^3e^{2\pi
i km\tau} \end{equation*}
for $\tau\in\C$ with $\im \tau >0$. The choice $\tau=i$ yields to
$$15 a^4 G_4=a^4 e_1^2={\pi^4\over
3}+80 \pi^4\sum_{m,k=1}^\infty k^3e^{-2\pi
km}$$
in view of \eqref{propej}, which turns \eqref{ceae4} into
\begin{eqnarray} \label{ceae5}
\hspace{-0.3cm}\bigg| {\pi^4 \over
3}+112 \pi^4\sum_{m,k=1}^\infty k^3e^{-2\pi
km} -32\pi^4 \sum_{k=1}^\infty \lambda_k(\lambda_k+1)(6\la_k^2+6\la_k+1) \bigg| <
3\pi   \sqrt{{\pi^4\over
3}+80 \pi^4\sum_{m,k=1}^\infty k^3e^{-2\pi
km}}.
\end{eqnarray}
Since numerically we can approximately compute
$$32 \pi^4\sum_{k=1}^\infty
\lambda_k(\lambda_k+1)(6\la_k^2+6\la_k+1)\approx 5,9194
\qquad 80 \pi^4 \sum_{m,k=1}^\infty k^3e^{-2\pi
km} \approx 14,7985,$$
we get the validity of \eqref{ceae5}, or equivalently \eqref{nondegenracy} for the vortex configuration $\{\tau\}\cup \{p_j:\, j \in J\}\subset \Omega$. \qed
\end{proof}

\medskip \noindent As a combination of Propositions \ref{propp1}, \ref{propp2} and \ref{propp3} we finally get that
\begin{thm} \label{thmexample}
Let $\Om$ be a square of side $a$, $a>0$, and assume that the vortex configuration is the periodic one generated by $\{0,{a\over
2},{ia \over 2},{a+ia \over 2}\}$ with multiplicities $2,n_1,n_2,2$ and $(n_1,n_2)=(2,0)$ (or viceversa).
Then, for $\tau$ small the assumption of Theorem \ref{main} do hold for the slightly translated vortex configuration $\{-\tau(1+i), -\tau(1+i)+\frac{a}{2}, -\tau(1+i)+\frac{ia}{2}, -\tau(1+i)+\frac{a+ia}{2} \}$. In particular, for $\e>0$ small we can find $N-$condensate $(\mathcal{A}_\e,\phi_\e)$ so that $|\phi_\e| \to 0$ in $C(\bar
\Omega)$ and
\begin{equation} \label{magconc}
(F_{12})_\e \rightharpoonup 12\pi \delta_0
\end{equation}
weakly in the sense of measures, as $\e \to 0$, where $\{0,{a\over 2},{ia \over 2},{a+ia \over 2}\}$ are the zeroes of $\phi_\e$ with multiplicities $2,n_1,n_2,2$ and $(n_1,n_2)=(2,0)$ (or viceversa).
\end{thm}

\medskip \noindent As a final remark, observe that for $n=0$ Theorem \ref{main} essentially recovers the result in
\cite{LinYan1} concerning single-point concentration in any torus
$\Om$ (see also \cite{EsFi}). Notice that $n=0$ corresponds to
have that the concentration point $0$ is not really a singular
point and a more simple approach is possible as in the
above-mentioned papers. By \eqref{balance} the total multiplicity $N$
is $2$ produced by two vortex-points $p_1,p_2\in \Omega \setminus \{0 \}$. Assumption
\eqref{pc} is equivalent to have $(\log \mathcal{H})'(0)=0$. By
the Cauchy-Riemann equations, the last condition can be just
re-written as
$$ \nabla [2 \re \log \mathcal{H}](0)=\nabla \log |\mathcal{H}|^2(0)=\nabla [8\pi H+u_0](0)=0.$$
Since $\nabla H(0)=0$ in view of $H(z)=H(-z)$, we have that \eqref{pc} simply reads as: $0$ is a critical point of $u_0$. As far as \eqref{D0}, notice that $D_0$ does not depend on $\rho>0$ small for
\begin{eqnarray*}
\int_{\sigma_0^{-1} (B_\rho(0)) \setminus \sigma_0^{-1} (B_r(0))}
e^{u_0+8\pi G(z,0)} -
\int_{B_\rho(0) \setminus B_r(0)}
\frac{dy}{|y|^4}=
\hbox{Area}\left( B_{\frac{1}{r}}(0) \setminus B_{\frac{1}{\rho}}(0)   \right)-\pi \Big(\frac{1}{r^2}-\frac{1}{\rho^2}\Big)=0
\end{eqnarray*}
for all $0<r\leq \rho$, in view of \eqref{eq sigma0} with $c_0=0$. Therefore, $D_0$ can be re-written as
\begin{eqnarray*}
D_0 = \frac{1}{\pi}\left[\int_{\Omega \setminus \sigma_0^{-1} (B_\rho(0))}
e^{u_0+8\pi G(z,0)} -
\int_{\mathbb{R}^2 \setminus B_\rho(0)}
\frac{dy}{|y|^4}\right]=\frac{1}{\pi}\lim_{r\to0} \bigg[\int_{\Omega \setminus \sigma_0^{-1}(B_r(0))}
{e^{8\pi H(z,0)+u_0} \over |z|^4} -
\int_{\mathbb{R}^2 \setminus B_{r}(0)} \frac{1}{|y|^4}\bigg].
\end{eqnarray*}
Since $\sigma_0(z)=\frac{z}{\lambda_0}+\frac{\mathcal{H}''(0)}{2\lambda_0^2}z^3+O(|z|^5)$ and $\sigma_0^{-1}(z)=\lambda_0z+O(|z|^3)$ with $\lambda_0=e^{4\pi H(0)-\frac{u_0(0)}{2}}$, notice that $B_{\lambda_0r-Cr^3}(0)\subset \sigma_0^{-1}(B_r(0)) \subset B_{\lambda_0 r+Cr^3}(0)$ for all $r>0$ small, for some constant $C>0$. Thus,  there holds
\begin{eqnarray*}
&&\bigg|\int_{\Omega \setminus \sigma_0^{-1}(B_r(0))}
{1\over |z|^4}e^{8\pi[H(z,0)-H(0,0)]+[u_0(z)-u_0(0)]}-\int_{\Omega \setminus B_{\lambda_0 r}(0)}
{1\over |z|^4}e^{8\pi[H(z,0)-H(0,0)]+[u_0(z)-u_0(0)]}\bigg|\\
&&=O\left( \int_{ B_{\lambda_0 r+Cr^3}(0) \setminus B_{\lambda_0  r-Cr^3}(0)} \frac{1}{|z|^2}\right) =o(1)
\end{eqnarray*}
as $r \to 0$ in view of $\nabla [8\pi H+u_0](0)=0$, yielding to the same expression for $D_0$ as in \cite{EsFi,LinYan1}:
$$D_0=\frac{\lambda_0^2}{\pi}\lim_{r\to0} \bigg[\int_{\Omega \setminus B_r(0)}
{1\over |z|^4}e^{8\pi[H(z,0)-H(0,0)]+[u_0(z)-u_0(0)]} -
\int_{\mathbb{R}^2 \setminus B_{r}(0)} \frac{1}{|y|^4}\bigg].$$
The ``non-degeneracy condition" \eqref{nondegenracy} reads as
$$\lf| \frac{\ml{H}''(0)}{\ml{H}(0)}-4\pi (H^*)''(0)\rg|=\lf| (\log \ml{H})''(0)-4\pi (H^*)''(0)\rg|\ne {2\pi\over |\Omega|},$$
in view of $\sigma_0=q_0$, $b_1=\lambda_0$, $f_1(z)=-4\pi \lambda_0 (H^*)'(z)+\frac{2 \lambda_0}{z}-\frac{2}{\sigma_0(z)}$ and $\ml{H}'(0)=0$. Setting
$\mathcal{H}_1(z)=e^{-4\pi H^*(z)} \mathcal{H}(z)$, we have that $|\mathcal{H}_1(z)|^2=e^{u_0+\frac{2\pi}{|\Omega|}|z|^2}$
and
\begin{eqnarray*}(\log \ml{H})''(0)-4\pi (H^*)''(0)&=&(\log \ml{H}_1)''(0)=2 (\re \log \ml{H}_1)''(0)= (\log |\ml{H}_1|^2 )''(0)=\Big(u_0+\frac{2 \pi}{|\Omega|} |z|^2\Big)''(0)\\
&=& \frac{1}{4}[(u_0)_{xx}(0)-(u_0)_{yy}(0)-2i (u_0)_{xy}(0)]
\end{eqnarray*}
in view of \eqref{definitionH}-\eqref{keyrelationH}, and the above condition turns into
\begin{eqnarray*} 0 &\not=&\frac{1}{16}\lf| (u_0)_{xx}(0)-(u_0)_{yy}(0)-2i (u_0)_{xy}(0) \rg|^2 -{4\pi^2\over |\Omega|^2}=
\frac{1}{16} \left((u_0)_{xx}(0)-(u_0)_{yy}(0)\right)^2+\frac{1}{4}(u_0)^2_{xy}(0) -{4\pi^2\over|\Omega|^2}\\
&=&\frac{1}{16}(\Delta u_0)^2 (0)-\frac{1}{4}\hbox{det}\,D^2 u_0(0)-{4\pi^2\over |\Omega|^2}=-\frac{1}{4}\hbox{det}\,D^2 u_0(0).
\end{eqnarray*}
In conclusion, when $n=0$ the assumptions in Theorem \ref{main} are equivalent to have $0$ as a non-degenerate critical point of $u_0(z)=-4\pi G(z,p_1)-4\pi G(z,p_2)$ with $D_0<0$.

%%%%%%%%%%%%%%%%%%%%%%%%%%%%%%%%%%%%%%%%%%%%%%%%%%%%%%%%%%%%%%%%%%%%%%%%%%%%%%%%%%%%%%%%%%%%%%%%%%%%%%%%%%%%%%%%%%%%%%%%%%%%%%%%%%%%%%%
%%%%%%%%%%%%%%%%%%%%%%%%%%%%%%%%%%%%%%%%%%%%%%%%%%%%%%%%%%%%%%%%%%%%%%%%%%%%%%%%%%%%%%%%%%%%%%%%%%%%%%%%%%%%%%%%%%%%%%%%%%%%%%%%%%%%%%%

%\newpage
\section{A more general result}\label{general}
In this section we deal with the case $m\geq 2$ in Theorem \ref{mainbb}. For more clearness, let us denote the concentration points as $\xi_l$, $l=1,\dots,m$, the remaining points in the vortex set as $p_j$, and by $n_l,n_j$ the corresponding multiplicities.

\medskip \noindent From section $2$ recall that $H(z)= G(z,0)+\frac{1}{2\pi} \log|z|$ is a smooth function in $2\Omega$ with $\Delta H=\frac{1}{|\Omega|}$, and $H^*$ is an holomorphic function  in $2 \Omega$ with $\re H^*=H-\frac{|z|^2}{4|\Omega|}$. Up to a translation, we are assuming that $p_j \in \Omega$ for all $j=1,\dots,N$, and taking $\tilde \Omega$ close to $\Omega$ so that $\tilde \Omega -p_j \subset 2 \Omega$ for all $j=1,\dots,N$. Arguing as for \eqref{definitionH}, the function
\begin{eqnarray*}
&& \mathcal{H}(z)= \prod_j (z-p_j)^{n_j} \hbox{exp}\left( 4\pi \sum_{l=1}^m (n_l+1) H^*(z-\xi_l)  -2\pi\sum_{j=1}^N H^*(z-p_j)\right.\\
&&\left. +\frac{\pi}{|\Omega|} \sum_{l=1}^m (n_l+1)(\xi_l-2z) \overline{\xi_l} -\frac{\pi}{2|\Omega|}\sum_{j=1}^N |p_j|^2+\frac{\pi}{|\Omega|}z \overline{\sum_{j=1}^N p_j} \right)
\end{eqnarray*}
is holomorphic in $\tilde \Omega$ and satisfies
$$|\mathcal{H}(z)|^2=\left(\prod_{l=1}^m |z-\xi_l|^{-2n_l}\right) \hbox{exp}\left(u_0+8\pi \sum_{l=1}^m (n_l+1) H(z-\xi_l)\right)$$
in view of \eqref{hhh}. For $l=1,\dots,m$ the function
$$\mathcal{H}^l(z)= \mathcal{H}(z) \prod_{l' \not= l} (z-\xi_{l'})^{-(n_{l'}+2)}$$
is holomorphic near $\xi_l$ and satisfies
\begin{eqnarray}
|\mathcal{H}^l(z)|^2=\hbox{exp}\left(4\pi (n_l+2)H(z-\xi_l)+4\pi \sum_{l'\not=l} (n_{l'}+2) G(z, \xi_{l'})-4\pi \sum_j n_j G(z,p_j) \right).  \label{1849}
\end{eqnarray}

\medskip \noindent To be more clear, let us spend few words to compare the case $m=1$ and $m\geq 2$. When $m=1$ notice that $\mathcal{H}$ satisfies $|\mathcal{H}|^2=e^{u_0+8\pi(n+1)H(z)-2n \log |z|}$ in view of \eqref{keyrelationH}. The function $e^{u_0+8\pi(n+1)H(z)-2n \log |z|}$ is a sort of effective potential for \eqref{3} at $0$, where $e^{u_0-2n \log |z|}$ is the non-vanishing part of $e^{u_0}$ and $e^{8\pi(n+1) H(z)}$ is the self-interaction of the concentration point $0$ driven by $PU_{\delta,0,\sigma_0}$ through \eqref{1138}. When $m\geq 2$, \eqref{1849} can be re-written as
$$|\mathcal{H}^l(z)|^2=\hbox{exp} \left(u_0+8\pi(n_l+1)H(z-\xi_l)+8\pi \sum_{l'\not=l} (n_{l'}+1) G(z, \xi_{l'})-2n_l \log |z-\xi_l|\right)$$
for $l=1,\dots,m$, yielding to an effective potential for \eqref{3} at $\xi_l$ exhibiting an additional interaction term $e^{8\pi \sum_{l'\not=l} (n_{l'}+1) G(z, \xi_{l'})}$ generated by the effect of the concentration points $\xi_{l'}$, $l'\not=l$, through \eqref{1139}.

\medskip \noindent Setting $\mathcal{H}_{0}=\frac{\mathcal{H}}{(z-\xi_1)^{n_1+2}\dots (z-\xi_m)^{n_m+2}}$, we now define $\sigma_0$ as
\begin{equation} \label{1143}
\sigma_0(z)=-\left( \int^z \mathcal{H}_{0}(w) \hbox{exp}\left[-\sum_{l=1}^m c_0^l (w-\xi_l)^{n_l+1} \prod_{l' \not=l}(w-\xi_{l'})^{n_{l'}+2}\right] dw \right)^{-1},
\end{equation}
where
$$c_0^l=\frac{1}{\mathcal{H}_0(\xi_l) (n_l+1)!}  \frac{d^{n_l+1} \mathcal{H}^l }{dz^{n_l+1}}(\xi_l),\quad l=1,\dots,m,$$
guarantee that all the residues of the integrand function in the definition of $\sigma_0$ vanish. The presence of the term $ \prod_{l' \not=l}(w-\xi_{l'})^{n_{l'}+2}$ is crucial to compute explicitly the $c_0^l$'s for
$$c_0^l (w-\xi_l)^{n_l+1} \prod_{l' \not=l}(w-\xi_{l'})^{n_{l'}+2}=O((w-\xi_{l'})^{n_{l'}+2})$$
has an high-order effect near any other $\xi_{l'}$, $l' \not= l$. By construction $\sigma_0 \in \mathcal{M}(\overline{\Omega})$ vanishes only at the $\xi_l$'s with multiplicity $n_l+1$ and
$$\lim_{z \to \xi_l} \frac{(z-\xi_l)^{n_l+1}}{\sigma_0(z)}=\frac{\mathcal{H}^l(\xi_l)}{n_l+1},$$
and satisfies
$$|\sigma_0'(z)|^2= |\sigma_0(z)|^4 \hbox{exp}\left(u_0+8\pi\sum_{l=1}^m (n_l+1)G(z,\xi_l) -2 \sum_{l=1}^m \re \bigg[c_0^l (z-\xi_l)^{n_l+1} \prod_{l'\not=l}(z-\xi_{l'})^{n_{l'}+2}\bigg] \right).$$
Under the assumptions of Theorem \ref{mainbb}, notice that $c_0^l=0$ for all $l=1,\dots,m$ and 
 $$\left|\left(1\over\sigma_0\right)'(z)\right|^2=|\mathcal{H}_0(z)|^2=e^{u_0+8\pi \sum_{l=1}^m (n_l+1) G(z,\xi_l)}.$$

\medskip \noindent Since each $\xi_l$ gives a contribution to the dimension of the kernel for the linearized operator \eqref{ol}, the parameters $\delta$ and $a$ are no longer enough to recover all the degeneracies induced by the ansatz $PU_{\delta,a,\sigma}$, for $\sigma \in \mathcal{M}(\overline{\Omega})$ a function which vanishes only at the points $\xi_l$, $l=1,\dots,m$, with multiplicity $n_l+1$. In our construction, the correct number of parameters to use is $2m+1$, given by $m$ small complex numbers $a_1,\dots,a_m$ and $\delta>0$ small, where the latter gives rise to the concentration parameter $\delta_l$ at $\xi_l$, $l=1,\dots,m$, by means of \eqref{repla2}. The request that all the $\delta_l$'s tend to zero with the same rate
is necessary as we will discuss later.

\medskip \noindent We need to construct an ansatz that looks as $PU_{\delta_l,a_l,\sigma_{a,l}}$ near each $\xi_l$, for a suitable $\sigma_{a,l}$ which makes the approximation near $\xi_l$ good enough. In order to localize our previous construction, let us define $PU_{\delta_l,a_l,\sigma}$ as the solution of
$$\left\{ \begin{array}{ll} -\Delta PU_{\delta_l,a_l,\sigma} =
\chi(|z-\xi_l|) |\sigma'(z)|^2 e^{U_{\delta_l,a_l,\sigma}} -\frac{1}{|\Omega|} \int_\Omega \chi(|z-\xi_l|) |\sigma'(z)|^2 e^{U_{\delta_l,a_l,\sigma}}& \hbox{in }\Omega\\
\int_\Omega PU_{\delta_l,a_l,\sigma}=0,&
\end{array} \right.$$
where $\chi$ is a smooth radial cut-off function so that $\chi=1$ in $[-\eta,\eta]$, $\chi=0$ in $(-\infty,-2\eta]\cup [2\eta,+\infty)$, $0<\eta<\frac{1}{2} \min\{|\xi_l-\xi_{l'}|, \hbox{dist }(\xi_l,\partial \Omega): l,l'=1,\dots,m,\, l\not= l' \}$. The approximating function is then built as $W=\displaystyle \sum_{l=1}^m PU_l$, where $U_{\delta_l,a_l,\sigma_{a,l}}$ and $PU_{\delta_l,a_l,\sigma_{a,l}}$ will be simply denoted by $U_l$ and $PU_l$.

\medskip \noindent Let us now explain how to find the functions $\sigma_{a,l}$, $l=1,\dots,m$. Setting
$$\mathcal{B}_r^l=\bigg\{ \sigma \hbox{ holomorphic in }B_{2\eta}(\xi_l):\:\Big\| \frac{\sigma}{\sigma_0}-1\Big\|_{\infty,B_{2\eta}(\xi_l)} \leq r  \bigg\}$$
for $l=1,\dots,m$, Lemma \ref{gomme} still holds in this context for all $\sigma \in  \mathcal{B}_r^l$, by simply replacing $0$, $n$ with $\xi_l$, $n_l$ and $\tilde \Omega$ with $B_{2\eta}(\xi_l)$.
Then, for all $\sigma=(\sigma_1,\dots,\sigma_m) \in \mathcal{B}_r:=\mathcal{B}_r^1 \times \dots \times \mathcal{B}_r^m$ and $a=(a_1,\dots,a_m) \in \mathbb{C}^m$ with $\|a\|_\infty <\rho$ there exist points $a_i^l$, $l=1,\dots,m$ and $i=0,\dots,n_l$, so that $\{z \in B_{2\eta}(\xi_l): \, \sigma_l(z)=a_l \}=\{\xi_l+a_0^l,\dots,\xi_l+a_{n_l}^l\}$ for all $l=1,\dots,m$. Arguing as for \eqref{Hasigma}, for $l=1,\dots,m$ the function
\begin{eqnarray*}
&& \hspace{-0.3cm} \mathcal{H}_{a,\sigma}^l(z)= \prod_j (z-p_j)^{n_j} \prod_{l' \not= l} (z-\xi_{l'})^{n_{l'}} \prod_{l' \not= l} \prod_{i=0}^{n_{l'}} (z-\xi_{l'}-a_i^{l'})^{-2} \hbox{exp}\left( 4\pi \sum_{l'=1}^m \sum_{i=0}^{n_{l'}} H^*(z-\xi_{l'}-a_i^{l'})  \right.\\
&&\hspace{-0.3cm} \left. -2\pi\sum_{j=1}^N H^*(z-p_j)+\frac{\pi}{|\Omega|} \sum_{l'=1}^m (n_{l'}+1) (\xi_{l'}-2z) \overline{ \xi_{l'}}
-\frac{\pi}{2|\Omega|}\sum_{j=1}^N |p_j|^2-\frac{2\pi}{|\Omega|}\sum_{l'=1}^m (z-\xi_{l'}) \overline{\sum_{i=0}^{n_{l'}} a_i^{l'}}
+\frac{\pi}{|\Omega|}z \overline{\sum_{j=1}^N p_j} \right)
\end{eqnarray*}
is holomorphic near $\xi_l$ and satisfies
\begin{equation} \label{keyrelationgen}
%\hspace{-0.3cm}
|\mathcal{H}_{a,\sigma}^l(z)|^2= |z-\xi_l|^{-2n_l}
\exp[u_0+8\pi \sum_{i=0}^{n_l} H(z-\xi_l-a_i^l)+8\pi
\sum_{l'\not= l} \sum_{i=0}^{n_{l'}}
G(z,\xi_{l'}+a_i^{l'})-\frac{2\pi}{|\Omega|} \sum_{l'=1}^m
\sum_{i=0}^{n_{l'}} |a_i^{l'}|^2]
\end{equation}
in view of \eqref{hhh}. Setting
$$g_{a_l,\sigma_l}^l(z)=\frac{\sigma_l(z)-a_l}{\prod_{i=0}^{n_l}(z-\xi_l-a_i^l)},\quad z \in B_{2\eta}(\xi_l),$$
and
\begin{equation} \label{cagen}
c_{a,\sigma}^l=\frac{\prod_{l'
\not=l}(\xi_l-\xi_{l'})^{-(n_{l'}+2)}}{(n_l+1)!}\frac{d^{n_l+1}}{dz^{n_l+1}}\left[
\Big(\frac{g^l_{a_l,\sigma_l}(z)
g^l_{0,\sigma_l}(\xi_l)}{g^l_{a_l,\sigma_l}(\xi_l)
g^l_{0,\sigma_l}(z)}\Big)^2
\frac{\mathcal{H}_{a,\sigma}^l(z)}{\mathcal{H}_{a,\sigma}^l(\xi_l)
} \right](\xi_l),
\end{equation}
the aim is to find a solution $\sigma_a =(\sigma_{a,1},\dots, \sigma_{a,m})\in \mathcal{B}_r$ of the system $(l=1,\dots,m)$:
\begin{equation} \label{sigmaagen}
\sigma_l(z)= -\left( \int^z
\Big(\frac{g^l_{a_l,\sigma_l}(w)}{g^l_{0,\sigma_l}(w)}\Big)^2
\frac{\mathcal{H}^l_{a,\sigma}(w)}{(w-\xi_l)^{n_l+2}}
\hbox{exp}\left[-\sum_{l'=1}^m c_{a,\sigma}^{l'}
(w-\xi_{l'})^{n_{l'}+1} \prod_{l''
\not=l'}(w-\xi_{l''})^{n_{l''}+2}\right] dw\right)^{-1},
\end{equation}
where the definition of $c_{a,\sigma}^l$ makes null the residue at $\xi_l$ of the integrand function in \eqref{sigmaagen}. The function $\sigma_{a,l}$ will vanish only at $\xi_l$ with multiplicity $n_l+1$ and satisfy
\begin{eqnarray} \label{eq sigmaagen}
|\sigma_{a,l}'(z)|^2&=&  |\sigma_{a,l}(z)-a_l|^4  \hbox{exp}\left(u_0+8\pi \sum_{l'=1}^m \sum_{i=0}^{n_{l'}} G(z,\xi_{l'}+a_i^{l'})-\frac{2\pi}{|\Omega|} \sum_{l'=1}^m \sum_{i=0}^{n_{l'}} |a_i^{l'}|^2\right.\\
&&\left. -2\sum_{l'=1}^m \re \Big[c^{l'}_{a,\sigma_{a}}
(z-\xi_{l'})^{n_{l'}+1}
\prod_{l''\not=l'}(z-\xi_{l''})^{n_{l''}+2}\Big] \right) \nonumber
\end{eqnarray}
in view of \eqref{keyrelationgen}.

\medskip \noindent Since $\mathcal{H}_{0,\sigma}^l=\mathcal{H}^l$ and $c^l_{0,\sigma}=c_0^l$ for all $l=1,\dots,m$, when $a=0$ the system \eqref{sigmaagen} reduces to $m$-copies of \eqref{1143} in each $B_{2\eta}(\xi_l)$, $l=1,\dots,m$, and it is natural to find $\sigma_a$ branching off $(\sigma_0,\dots,\sigma_0)$ for $a$ small by IFT. Let us emphasize that each $\sigma_{a,l}$, $l=1,\dots,m$, is close to $\sigma_0\Big|_{B_{2\eta}(\xi_l)}$, a crucial property to have $D_0$ defined in terms of a unique $\sigma_0$ (see \eqref{ggg}). Letting $q_{0,l}$ be the function so that $\sigma_0=q_{0,l}^{n_l+1}$ near $\xi_l$, arguing as in Lemma \ref{derivca} we have that
\begin{lem}\label{derivcagen}
Up to take $\rho$ smaller, there exists a $C^1-$map $a \in
B_\rho(0) \to \sigma_a \in \mathcal{B}_r$ so that $\sigma_a$ solves the system \eqref{cagen}-\eqref{sigmaagen}.
Moreover, the map $a \in B_\rho(0) \to c_a^l:=c^l_{a,\sigma_a}$ is $C^1$
with
\begin{eqnarray}
&& \Gamma^{ll}:=\mathcal{H}(\xi_l) \partial_{a_l} c_a^l
\Big|_{a=0}=\frac{1}{n_l !}\frac{d^{n_l+1}}{dz^{n_l+1}}\bigg[
\mathcal{H}^l(z)f_{n_l+1}^l(z)\bigg]  (\xi_l) \label{primo}\\
&&\Upsilon^{ll}:=\mathcal{H}(\xi_l) \partial_{\bar a_l} c_a^l \Big|_{a=0}=-{2\pi(n_l+1)\over |\Om|
n_l!}\overline{b_{n_l+1}^l}\,\frac{d^{n_l} \mathcal{H}^l }{dz^{n_l}}(\xi_l) \label{secondo}
\end{eqnarray}
and for $j \not= l$
\begin{eqnarray}
&&\Gamma^{lj}:=\mathcal{H}(\xi_l) \partial_{a_j} c_a^l
\Big|_{a=0}=\frac{n_j+1}{(n_l+1)!}\frac{d^{n_l+1}}{dz^{n_l+1}}\bigg[
\mathcal{H}^l(z) \ti f_{n_j+1}^j(z) \bigg]  (\xi_l) \label{terzo}\\
&&\Upsilon^{lj}:=\mathcal{H}(\xi_l) \partial_{\bar a_j} c_a^l \Big|_{a=0}=-{2\pi(n_j+1)\over |\Om| n_l!}\overline{b_{n_j+1}^j}\,\frac{d^{n_l} \mathcal{H}^l}{dz^{n_l}}(\xi_l) \label{quarto},
\end{eqnarray}
where
$$f_{n+1}^l(z)=\frac{1}{(n+1)!} \frac{d^{n+1}}{dw^{n+1}}\left[2\log \frac{w-q_{0,l}(z)}{q_{0,l}^{-1}(w)-z}+4\pi
H^*(z-q_{0,l}^{-1}(w))\right] (0)\:,\qquad b_{n+1}^l=\frac{1}{(n+1)!}\frac{d^{n+1} q_{0,l}^{-1}}{dw^{n+1}}(0)$$
and for $j \not=l$
$$\tilde f_{n+1}^j(z)=\frac{1}{(n+1)!} \frac{d^{n+1}}{dw^{n+1}} \bigg[-2\log(z-q_{0,j}^{-1}(w))+  4\pi H^*(z-q_{0,j}^{-1}(w))\bigg](0).$$
\end{lem}

\noindent Letting $n=\min\{n_l :\: l=1,\dots,m\}$, up to re-ordering, assume that $n=n_1=\dots=n_{m'}<n_l$ for all $l=m'+1,\dots,m$, where $1\leq m'\leq m$. The matrix $A$ in Theorem \ref{mainbb} is the $2m \times 2m-$matrix in the form
\begin{equation} \label{matrixA}
A=\left( \begin{array}{ccc} A_{1,2}^{1,2}& \dots & A_{1,2}^{2m-1,2m}\\
\vdots& \vdots &\vdots\\
A_{2m-1,2m}^{1,2}& \dots& A_{2m-1,2m}^{2m-1,2m} \end{array} \right),
\end{equation}
where the $2\times 2$-blocks are given by
$$A_{2l-1,2l}^{2l'-1,2l'}=\left(\begin{array}{cc} \re [\Gamma^{ll'}+\Upsilon^{ll'}+\frac{n(2n+3)}{n+1} D_0 \frac{|\mathcal{H}^l(\xi_l)|^{-\frac{2}{n+1}}}{\sum_{j=1}^{m'}|\mathcal{H}^j(\xi_j)|^{-\frac{2}{n+1}}}  \delta_{ll'}]& \im [\Upsilon^{ll'}-\Gamma^{ll'}]\\
\im [\Gamma^{ll'}+\Upsilon^{ll'}]& \im [\Gamma^{ll'}-\Upsilon^{ll'}-\frac{n(2n+3)}{n+1} D_0 \frac{|\mathcal{H}^l(\xi_l)|^{-\frac{2}{n+1}}}{ \sum_{j=1}^{m'}|\mathcal{H}^j(\xi_j)|^{-\frac{2}{n+1}}} \delta_{ll'}] \end{array}\right)$$
when $l=1,\dots,m'$ and by
$$A_{2l-1,2l}^{2l'-1,2l'}=\left(\begin{array}{cc} \re [\Gamma^{ll'}+\Upsilon^{ll'}]& \im [\Upsilon^{ll'}-\Gamma^{ll'}]\\
\im [\Gamma^{ll'}+\Upsilon^{ll'}]& \im [\Gamma^{ll'}-\Upsilon^{ll'}] \end{array}\right)$$
when $l=m'+1,\dots,m$, with $\Gamma^{ll'}$ and $\Upsilon^{ll'}$ given by \eqref{primo}, \eqref{terzo} and \eqref{secondo}, \eqref{quarto}, respectively, and $\delta_{ll'}$ the Kronecker's symbol.

\medskip \noindent Arguing as in Lemma \ref{expPU}, for $l=1,\dots,m$ we have that
\begin{eqnarray*}
PU_{\delta_l,a_l,\sigma_l}&=&\chi(|z-\xi_l|) \left[U_{\delta_l,a_l,\sigma_l}-\log (8 \delta^2_l)+4 \log |g_{a_l,\sigma_l}^l| \right]\\
&&+8\pi \sum_{i=0}^{n_l} \left[ \frac{1}{2\pi} (\chi(|z-\xi_l|)-1) \log |z-\xi_l-a_i^l|+H(z-\xi_l-a_i^l)\right]+\Theta_{\delta_l,a_l,\sigma_l}+2\delta^2_l f_{a_l,\sigma_l}+O(\de^4_l)
\end{eqnarray*}
and
\begin{eqnarray} \label{1139}
PU_{\delta_l,a_l,\sigma_l}=8\pi \sum_{i=0}^{n_l} G(z,\xi_l+a_i^l)+\Theta_{\delta_l,a_l,\sigma_l}+2\delta^2_l \lf( f_{a_l,\sigma_l}-{\chi(|z-\xi_l|)\over|\sigma_l(z)-a_l|^2}\rg)+O(\delta^4_l)
\end{eqnarray}
do hold in $C(\overline{\Omega})$ and $C_{\text{loc}}(\overline{\Omega} \setminus\{\xi_l\})$, respectively, uniformly for $|a|< \rho$ and $\sigma_l \in \mathcal{B}_r^l$, where
$$\Theta_{\de_l,a_l,\sigma_l}=-\frac{1}{|\Omega|}\int_\Om \chi(|z-\xi_l|) \log {|\sigma_l(z)-a_l|^4\over
(\de_l^2+|\sigma_l(z)-a_l|^2)^2}$$ and $f_{a_l,\sigma_l}$ is a smooth function in $z$ (with a uniform control in $a_l$ and $\sigma_l$ of it and its derivatives in $z$).
Choosing $\sigma_l=\sigma_{a,l}$ and summing up over $l=1,\dots,m$, by \eqref{eq sigmaagen} for our approximating function there hold
\begin{eqnarray}\label{ieagr}
W&=& U_{\delta_l,a_l,\sigma_l}-\log (8 \delta^2_l)+\log |\sigma_l'|^2
-u_0+\frac{2\pi}{|\Omega|} \sum_{l'=1}^m \sum_{i=0}^{n_{l'}} |a_i^{l'}|^2+\Theta^l(a,\delta) \\
&&+2 \re \Big[c^l_{a,\sigma_l} (z-\xi_l)^{n_l+1}
\prod_{l'\not=l}(z-\xi_{l'})^{n_{l'}+2}\Big]
+O(|z-\xi_l|^{n_l+2}\sum_{l'\ne l}|c^{l'}_{a,\sigma_{l'}}|)+\sum_{l'=1}^m O(\delta^2_{l'}
|z-\xi_l|+\de^4_{l'})\nonumber
\end{eqnarray}
and
$$W= 8\pi \sum_{l=1}^m \sum_{i=0}^{n_l} G(z,\xi_l+a_i^l)+O\bigg(
\sum_{l'=1}^m \delta^2_{l'} \log|\delta_{l'}|\bigg)$$
uniformly in $B_\eta(\xi_l)$ and in $\Omega \setminus \cup_{l=1}^m B_\eta(\xi_l)$, respectively, where
$$\Theta^l(a,\delta):=\sum_{l'=1}^m [\Theta_{\delta_{l'},a_{l'},\sigma_{l'}}+\delta^2_{l'} f_{a_{l'},\sigma_{l'}}(\xi_l)].$$
As a consequence, we have that
$$\int_\Omega e^{u_0+W}= \sum_{l'=1}^m \left[\int_{B_\rho(0)} \frac{n_{l'}+1}{(\delta_{l'}^2+|y-a_{l'}|^2)^2} +o\Big(\frac{1}{\delta_{l'}^2}\Big)\right]= \pi \sum_{l'=1}^m \frac{n_{l'}+1}{\delta_{l'}^2} [1+o(1)],$$
and then near $\xi_l$ there holds
$$4\pi N \frac{e^{u_0+W}}{\int_\Omega e^{u_0+W}}=4\pi N \frac{|\sigma_l'|^2 e^{U_{\delta_l,a_l,\sigma_l}+O(|z-\xi_l|^{n_l+1})+o(1)}}{8\pi \sum_{l'=1}^m (n_{l'}+1) \delta_l^2 \delta_{l'}^{-2}(1+o(1))}.$$
In order to construct a $N-$condensate $(\mathcal{A}_\e,\phi_\e)$ which satisfies
\eqref{magconc} as $\e \to 0$, we look for a solution $w_\e$ of \eqref{3} in the form $w_\e=\displaystyle \sum_{l=1}^m PU_{\delta_l,a_l,\sigma_l}+\phi$, where $\phi$ is a small remainder term and $\delta_l=\delta_l(\e)$, $a_l=a_l(\e)$ are suitable small parameters, so that
\begin{eqnarray*}   &&4\pi N \frac{e^{u_0+w_\e}}{\int_\Omega e^{u_0+w_\e}}+
\frac{64 \pi^2N^2 \epsilon^2 \int_\Omega
e^{2u_0+2w_\e}}{(\int_\Omega e^{u_0+w_\e}+\sqrt{(\int_\Omega
e^{u_0+w_\e})^2-16\pi N\epsilon^2\int_\Omega
e^{2u_0+2w_\e}})^2}\left(\frac{e^{u_0+w_\e}}{\int_\Omega
 e^{u_0+w_\e}} -\frac{e^{2u_0+2w_\e}}{\int_\Omega
e^{2u_0+2w_\e}}\right) \\
&&\hspace{3cm} \rightharpoonup 8\pi \sum_{l=1}^m (n_l+1) \delta_{\xi_l}
\end{eqnarray*}
in the sense of measures as $\epsilon \to 0$. Since $|\sigma_l'|^2 e^{U_{\delta_l,a_l,\sigma_l}} \rightharpoonup 8\pi (n_l+1) \delta_{\xi_l}$ as $\delta_l,a_l \to 0$, to have the correct concentration property we need that
$$8\pi \sum_{l'=1}^m (n_{l'}+1) \delta_l^2 \delta_{l'}^{-2} \to 4\pi N$$
for all $l=1,\dots,m$, and then $\frac{\delta_l}{\delta_{l'}} \to 1$ for all $l,l'=1,\dots,m$ in view of \eqref{hhh}. It is then natural to introduce just one parameter $\delta$ and to chose the $\delta_l$'s as
\begin{equation}\label{repla2}
\de_l=\de \qquad l=1,\dots,m.
\end{equation}

\bigskip \noindent We restrict our attention to the case $c_0^l=0$ for all $l=1,\dots,m$, which is necessary in our context and is simply a re-formulation of the assumption that $\mathcal{H}_0$ has zero residues at $p_1,\dots,p_m$. As in Theorem \ref{main}, we will work in the parameter's range:
$$a_l=o(\delta),\qquad \delta \sim \e^{n+1\over n+2}$$
as $\e\to 0^+$. Since then
$$K^{-1} \le \frac{\de^2+|z-\xi_l|^{2n_l+2}}{\de^2+\big|\sigma_l(z)
-a_l|^2} \le K, \qquad  K^{-1}|z-\xi_l|^{2n_l}\leq |\sigma_l'(z)|^2 \leq  K |z-\xi_l|^{2n_l}$$
in $B_{2\eta}(\xi_l)$ for all $\sigma_l \in \mathcal{B}_r^l$ and $l=1,\dots,m$, where $K>1$, the norm \eqref{wn} can be now simply defined as
$$\| h \|_*=\sup_{z\in \Om}\lf[ \sum_{l=1}^m\frac{\delta^{\gamma}
\lf(|z-\xi_l|^{2n_l}+\delta^{\frac{2n_l}{n_l+1}}\rg)}{(\delta^2+|z-\xi_l|^{2n_l+2})^{1+\frac{\gamma}{2}}}\rg]^{-1}\;
|h(z)|$$
for any $h\in L^\infty(\Om)$, where $0<\gamma<1$ is a small
fixed constant. In order to simplify notations, we set $U_l=U_{\de_l,a_l,\sigma_l}$, $c_a^l=c_{a,\sigma_l}^l$, $\Theta_l=\Theta_{\de_l,a_l,\sigma_l}$ and $f_l=f_{a_l,\sigma_l}$. We have that
\begin{lem}\label{estrrm}
There exists a constant $C>0$ independent of $\de$ \st
\begin{equation}\label{erem}
\|R\|_*\le C\de^{2-\gamma}.
\end{equation}
\end{lem}
\begin{proof}
We shall sketch the proof of \eqref{erem}, by following ideas used in the proof of Theorem \ref{estrr01550}. Through the change of variable $y=\sigma_l(z)$ in $\sigma_l^{-1}(B_\rho(0))$, by Lemma \ref{derivcagen}, \eqref{ieagr}, \eqref{repla2} and $c_0^l=0$ for all $l=1,\dots,m$ we find that
\begin{eqnarray*}
&&{8\de^2\over e^{{2\pi\over|\Om|}\sum_{l'=1}^m\sum_{i=0}^{n_{l'}}|a_i^{l'}|^2+\Theta^l(a,\de)}}\int_{\sigma^{-1}_l(B_{\rho}(0))} e^{u_0+W}
=\int_{\sigma^{-1}_l(B_{\rho}(0))} |\sigma_l'|^{2}
e^{U_{l}+O(|z-\xi_l|^{n_l+1} \sum_{l'=1}^m |c_{a}^{l'}|+\delta^2|z-\xi_l|+\delta^4)}\\
&&=8\pi(n_l+1)
- \int_{\mathbb{R}^2 \setminus B_{\rho}(0)} \frac{8(n_l+1) \delta^2}{|y|^4}+ O\Big(\|a\|^2+\de\|a\|+\de^{2n_l+3\over n_l+1}\Big),
\end{eqnarray*}
where $\|a\|^2=\displaystyle \sum_{l=1}^m|a_l|^2$. Setting $\Om_\rho=\cup_{l=1}^m \sigma_l^{-1}(B_\rho(0))$ we get that
\begin{eqnarray*}
&&
{8\de^2\over e^{{2\pi\over|\Om|}\sum_{l'=1}^m\sum_{i=0}^{n_{l'}}|a_i^{l'}|^2+\sum_{l'=1}^m\Theta_{l'}}}\,\int_\Omega e^{u_0+W} = \sum_{l=1}^m e^{\de^2\sum_{l'=1}^mf_{l'}(\xi_l)}\bigg[8\pi(n_l+1)-\int_{\mathbb{R}^2 \setminus B_{\rho}(0)}
\frac{8(n_l+1) \delta^2}{|y|^4}\\
&&+O(\|a\|^2+\de\|a\|+\de^{2n_l+3\over n_l+1}\Big)\bigg]+8\delta^2 \int_{\Omega \setminus\Om_\rho} e^{u_0+8\pi \sum_{l=1}^m\sum_{i=0}^{n_l} G(z,\xi_l+a_{i}^l)}+O(\de^4|\log \delta| +\delta^2\|a\|^{\frac{2}{\max_l n_l+1}})\\
&&=\sum_{l=1}^m\bigg[8\pi(n_l+1)+8\pi(n_l+1)\de^2\sum_{l'=1}^m f_{l'}(\xi_l)-8(n_l+1)\de^2\int_{\mathbb{R}^2 \setminus B_{\rho}(0)} \frac{1}{|y|^4}\bigg]\\
&&+8\delta^2 \int_{\Omega \setminus \Om_\rho} e^{u_0+8\pi \sum_{l=1}^m\sum_{i=0}^{n_l} G(z,\xi_l+a_i^l)}+o(\de^2)=\;4\pi N\lf[1+{2\over N}\de^2D_a+{2\over
N}\de^2\sum_{l,l'=1}^m(n_l+1)f_{l'}(\xi_l)+o(\de^2)\rg]
\end{eqnarray*}
in view of \eqref{hhh}, where $D_a$ is given by
$$\pi D_a=\int_{\Omega \setminus \Omega_\rho} e^{u_0+8\pi \sum_{l=1}^m\sum_{i=1}^{n_l}G(z,\xi_l+a_i^l)} -
\sum_{l=1}^m (n_l+1)\int_{\mathbb{R}^2 \setminus B_{\rho}(0)} \frac{1}{|y|^4}.$$
Hence, for $|z-\xi_l| \leq \eta$ we have that
\begin{eqnarray}\label{impm}
&&\Delta W+4\pi N\left( \frac{e^{u_0+W}}{\int_\Omega e^{u_0+W}}-\frac{1}{|\Omega|}\right)=|\sigma_l'|^2 e^{U_l} \bigg[2\hbox{Re }\Big[c_a^l(z-\xi_l)^{n_l+1}\prod_{l' \ne l}(z-\xi_{l'})^{n_{l'}+2}\Big]\\
&&+\de^2\sum_{l'=1}^m f_{l'}(\xi_l)-{2D_a\over N}\delta^2-{2\de^2\over N}\sum_{j,l'=1}^m(n_j+1)
f_{l'}(\xi_j)+O(\|a\| |z-\xi_l|^{n_l+2}+\de^2|z-\xi_l|)+o(\delta^2)\bigg]
+O(\delta^2)\nonumber 
\end{eqnarray}
as $\de  \to 0$, in view of \eqref{hhh} and $\int_\Omega \chi_l
|\sigma_l'|^{2} e^{U_{l}}=8\pi(n_l+1)+O(\delta^2)$ for all $l=1,\dots,m$. For $z \in
\Omega \setminus \cup_{l=1}^mB_\eta(\xi_l)$, we have that
\begin{eqnarray}
\Delta W+4\pi N\left( \frac{e^{u_0+W} }{\int_\Omega e^{u_0+W}
}-\frac{1}{|\Omega|}\right)=O(\delta^2). \label{impextm}
\end{eqnarray}
On the other hand, arguing as in \eqref{const1}, we have that
\begin{eqnarray*}
{64\delta^{4}\over
e^{{4\pi\over|\Om|}\sum_{l'=1}^m\sum_{i=1}^{n_{l'}}|a_i^{l'}|^2+2\sum_{l'=1}^m\Theta_{l'}}}\int_\Omega
e^{2u_0+2W}
=64\sum_{l=1}^{m'}{(n+1)^3\over |\alpha_{a,l}|^{{2\over
n+1}}\de^{{2\over n+1}} }\int_{\mathbb{R}^2}
\frac{|y+a_l \delta^{-1} |^{\frac{2n}{n+1}}}{(1+|y|^2)^4}
+O(\de^{-\frac{1}{n+1}}),
\end{eqnarray*}
where $\ds\alpha_{a,l}=\lim_{z\to \xi_l}{(z-\xi_l)^{n_l+1}\over \sigma_l(z)}$.
Recall that $n=\min\{n_l: l=1,\dots,m\}=n_1=\dots=n_{m'}<n_l$ for all $l=m'+1,\dots,m$. Setting
$$\ds\ti
D_{a,\de}=\sum_{l=1}^{m'}{(n+1)^3\over|\alpha_{a,l}|^{{2\over
n+1}}\de^{{2\over n+1}}}\int_{\mathbb{R}^2}
\frac{|y+a_l\delta^{-1} |^{\frac{2n}{n+1}}}{(1+|y|^2)^4}\,dy,$$
we have that
$$\frac{4\pi N\e^2B(W)}{(1+\sqrt{1-\e^2B(W)})^2}=64\e^2\ti D_{a,\de}
+o(\e^2\de^{-\frac{2}{n+1}}),$$
and there hold
\begin{equation}\label{eps4m}
\frac{4\pi
N\e^2B(W)}{(1+\sqrt{1-\e^2B(W)})^2}\left(\frac{e^{u_0+W}}{\int_\Omega
e^{u_0+W}}-\frac{e^{2u_0+2W}}{\int_\Omega
e^{2u_0+2W}}\right)=|\sigma_l'|^{2}e^{U_l}\bigg[{16\e^2\over
\pi N}\ti
D_{a,\de}-\e^2|\sigma_l'|^{2}e^{U_l}+o(\e^2\de^{-2\over
n+1})\bigg] \end{equation}  
in $B_\eta(\xi_l)$, $l=1,\dots,m$, and
\begin{equation}\label{eps5m}
\frac{4\pi
N\e^2B(W)}{(1+\sqrt{1-\e^2B(W)})^2}\left(\frac{e^{u_0+W}}{\int_\Omega
e^{u_0+W}}-\frac{e^{2u_0+2W}}{\int_\Omega
e^{2u_0+2W}}\right)=O(\epsilon^2 \delta^{\frac{2n}{n+1}})
\end{equation} 
in $\Omega \setminus \cup_{l=1}^m B_\eta(\xi_l)$. Therefore, we conclude that
$\|R\|_*=O(\de^{2-\gamma}+\|a\|^2+\e^2\de^{-{2\over n+1}})$ and
\eqref{erem} follows. \qed
\end{proof}

\medskip \noindent As mentioned in section $4$, when we look for a solution of \eqref{3} in the
form $w=W+\phi$, we are led to study \eqref{ephi}. In order to state the invertibility of the linear
operator $L$ in a suitable functional setting, for $l=1,\dots,m$ let us introduce
the functions:
$$Z_{0l}(z)=\frac{\delta^2-|\sigma_l(z)-a_l|^2}{\delta^2+|\sigma_l(z)-a_l|^2},\quad Z_l(z)=
\frac{\delta(\sigma_l(z)-a_l)}{\delta^2+|\sigma_l(z)-a_l|^2}\qquad z\in B_{2\eta}(\xi_l).$$
Also, let $PZ_{0l}$ and $PZ_l$ be the unique solutions with zero average of
$$\Delta PZ_{0l} =\chi_l \Delta
Z_{0l}-\frac{1}{|\Om|}\int_\Om \chi_l \Delta Z_{0l},\qquad \Delta
PZ_l =\chi_l \Delta Z_l-\frac{1}{|\Om|}\int_\Om \chi_l \Delta Z_l$$
where $\chi_l(z):=\chi(|z-\xi_l|)$, and set $PZ_0=\displaystyle \sum_{l=1}^m
PZ_{0l}$. As in Propositions \ref{prop4.1}-\ref{nlp}, it is possible to prove:
\begin{prop} Let $M_0>0$. There exists $\eta_0>0$ small such that for any $0<\de\leq\eta_0$,
$|\log\de|^2\e^2\leq \eta_0 \de^{2\over n+1}$ and $\|a\|\leq M_0 \de$ there is a unique solution
$\phi=\phi(\de,a)$, $d_0=d_0(\de,a)\in\R$
and $d_l=d_l(\de,a)\in\C$, $l=1,\dots,m$, to
$$\left\{\begin{array}{ll}
L(\phi) =-[R+N(\phi)]  + d_0 \Delta PZ_{0}+\displaystyle \sum_{l=1}^m\re[d_l \lap PZ_l] &\text{in }\Om\\
\int_\Om\phi=\int_{\Omega } \phi  \Delta PZ_l= 0&l=0,\dots,m.
\end{array} \right. $$
Moreover, the map $(\de,a)\mapsto \phi(\de,a)$ is $C^1$ with
\begin{equation}\label{estphim}
\|\phi\|_\infty\le C \delta^{2-\sigma}|\log \delta|.
\end{equation}
\end{prop}

\medskip \noindent The function $W+\phi$ is a solution of (\ref{3}) if we adjust $\delta$ and $a$ so to have
$d_l(\de,a)=0$ for all $l=0,1,\dots,m$. Similarly to Lemma \ref{1039}, we have that
\begin{lem}
There exists $\eta_0>0$ \st if $0<\de\leq \eta_0$, $\|a\|\leq \eta_0 \delta$ and
\begin{equation}  \label{solvem}
\int_\Omega (L(\phi)+N(\phi)+R) PZ_l=0
\end{equation}
does hold for all $l=0,\dots,m$, then $W+\phi$ is a solution of \eqref{3}, i.e.
$d_l(\de,a)=0$ for all $l=0,\dots,m$.
\end{lem}
\noindent 
Since there hold the expansions
\begin{equation*}%\label{pzijm}
PZ_{0}=\sum_{l=1}^m\bigg[\chi_l(Z_{0l}+1)-{1\over|\Om|}\int_\Om
\chi_l(Z_{0l}+1)\bigg]+O(\de^2)
\:,\quad PZ_l=\chi_l Z_l-{1\over|\Om|}\int_\Om \chi_l Z_l+O(\de)\:\:l=1,\dots,m \end{equation*}
in $C(\bar \Om)$, arguing as in Proposition \ref{1219}, by \eqref{hhh} and \eqref{impm}-\eqref{estphim} we can deduce the following expansion for \eqref{solvem}:
\begin{lem}
Assume $c_0^l=0$ for all $l=1,\dots,m$ and $\|a\|\leq \eta_0 \de$. The following
expansions do hold as $\epsilon \to 0$
\begin{eqnarray*}
\int_\Omega (L(\phi)+N(\phi)+R) PZ_0&=& -8\pi D_0 \de^2
+64(n+1)^{\frac{3n+5}{n+1}} \e^2 \de^{-{2 \over n+1}}
\sum_{l=1}^{m'} |\mathcal{H}^l(\xi_l)|^{-\frac{2}{n+1}}
\int_{\mathbb{R}^2}
\frac{(|y|^2-1)|y+\frac{a_l}{\delta}|^{\frac{2n}{n+1}}}{(1+|y|^2)^5} dy\\
&&+ o(\de^2+\e^2\de^{-{1\over n+1}})+O(\epsilon^4
\delta^{-\frac{2}{n+1}}|\log \delta|^2+\e^8 \delta^{-\frac{4}{n+1}}|\log \delta|^2 )\end{eqnarray*}
and
\begin{eqnarray*}
\int_\Omega (R+L(\phi)+N(\phi)) PZ_l &=& 4 \pi \delta
\sum_{l'=1}^m (\overline{\Upsilon^{ll'}} a_{l'}+ \overline{\Gamma^{ll'}} \bar
a_{l'})-64 (n+1)^{\frac{3n+5}{n+1}} \e^2 \de^{-{2\over n+1}} |\mathcal{H}^l(\xi_l)|^{-\frac{2}{n+1}}
\chi_M(l) \int_{\mathbb{R}^2}
\frac{|y+\frac{a_l}{\delta}|^{\frac{2n}{n+1}}y}{(1+|y|^2)^5} dy \\
&&+o(\de^2+\epsilon^2 \delta^{-\frac{2}{n+1}})+O(\epsilon^4
\delta^{-\frac{2}{n+1}}|\log \delta|^2+\e^8 \delta^{-\frac{4}{n+1}}|\log \delta|^2 ),
\end{eqnarray*}
where $D_0$ is defined in \eqref{ggg} and $\chi_M$ is the
characteristic function of the set $M=\{1,\dots,m'\}$.
\end{lem}
\noindent Finally, arguing as in the proof of Theorem \ref{main}, we can establish Theorem \ref{mainbb} thanks to $D_0<0$ and the invertibility of the matrix $A$. 

\medskip \noindent Let us now discuss some examples with $m\geq 2$. As already explained at the beginning of section \ref{examples}, we can consider the case $\xi_1,\dots,\xi_m \in \Omega$ and $p_j \in \bar \Omega$ for all $j$. In general, it is very difficult to establish the sign of $D_0$ as required in \eqref{ggg}. The key idea is to start from a configuration of the vortex points $\{p_1,\dots,p_N\}$ which is obtained in a periodic way by a simpler configuration having just one concentration point. In this case, \eqref{ggg} easily follows but Theorem \ref{mainbb} is not really needed. One can use Theorem \ref{main} to obtain a solution with such a simpler configuration and then repeat it periodically. We then slightly move some of the vortex points in order to:
\begin{itemize}
\item keep zero residue of the corresponding $\mathcal{H}_0$ at each concentration point;
\item break down the periodicity of the configuration.
\end{itemize}
In this way, assumption \eqref{ggg} is still valid but Theorem \ref{main} is no longer applicable in the trivial way we explained above. We now really need to resort to Theorem \ref{mainbb}. To exhibit some concrete examples, let us focus for simplicity on the case $m=2$ but the general situation can be dealt in the same way. Let $\Omega$ be a rectangle generated by $\omega_1=a$ and $\omega_2=ib$, $a,b>0$, and let $p_1,p_2,p_3$ be the three half-periods. Assume that the vortex set is $\{-\frac{p_1}{2}, \frac{p_1}{2},0,p_1,p_2,p_3\}$, and the concentration points are $\xi_1=-\frac{p_1}{2}$, $\xi_2=\frac{p_1}{2}$ with multiplicity $n$. Supposing that $0$, $p_1$ have even multiplicity $n_1$ and $p_2,p_3$ have even multiplicity $n_2$ with $n_1+n_2=n+2$, we have that such a configuration is not only $\omega_1=2p_1$ periodic but also $p_1$ periodic: it can be tought as a double repetition (in a $p_1$-periodic way) of the vortex configuration $\{-\frac{p_1}{2}, 0,p_2\}$ in $\Omega_-:=[-\frac{a}{2},0]\times [-\frac{b}{2},\frac{b}{2}]$ with corresponding multiplicities $n$, $n_1$ and $n_2$. If $n$ is even, it is easy to see that $\frac{d^{n+1} \mathcal{H}^i}{d z^{n+1}}(\xi_i)=0$ for $i=1,2$ since the given vortex configuration is even with respect to $\xi_1$ and $\xi_2$. Notice that this is still true if we replace $0$ and $p_1$ by $-it$ and $p_1+it$, respectively, for $t \in \mathbb{R}$, provided they keep the same multiplicity $n_1$. Arguing as in \eqref{1846}, notice that $D_0$ can be written as
$$\pi D_0=\hbox{Area } \left[ \frac{1}{\sigma_0}\left(\Omega_- \setminus \sigma_0^{-1} (B_\rho(0)) \right) \right]+\hbox{Area } \left[ \frac{1}{\sigma_0}\left(\Omega_+ \setminus \sigma_0^{-1} (B_\rho(0)) \right) \right]  - 2(n+1) \hbox{Area} \left( B_{\frac{1}{\rho}}(0) \right),$$
where $\Omega_+:=[0,\frac{a}{2}]\times [-\frac{b}{2},\frac{b}{2}]$. Since 
$$u_0+8\pi(n+1)G(z,\xi_1)+8\pi(n+1)G(z,\xi_2)=-4\pi n_1 \tilde G(z,0)-4\pi n_2 \tilde G(z,p_2)+4\pi(n+2) \tilde G(z,\xi_1)$$
in $\Omega_-$, where $\tilde G(z,p)$ is the Green function in the torus $\Omega_-$ with pole at $p$, the function $\mathcal{H}_0$ can be expressed as in \eqref{explicitH}  in terms of the Weierstrass function of $\Omega_+$ and the points  $-\frac{p_1}{2}$, $0$ and $p_2$. Arguing exactly as in section \ref{examples}, we have that
$$\hbox{Area } \left[ \frac{1}{\sigma_0}\left(\Omega_- \setminus \sigma_0^{-1} (B_\rho(0)) \right) \right]-(n+1) \hbox{Area} \left( B_{\frac{1}{\rho}}(0) \right)<0$$
provided the multiplicity $n_2$ for the corner of $\Omega_-$ is so that $\frac{n_2}{2}$ is odd. Arguing similarly in $\Omega_+$, we get that $D_0<0$ as soon as $\frac{n_2}{2}$ is an odd number. The example then follows by replacing $0$, $p_1$ with $-it$, $p_1+it$ with $t$ small for the corresponding $D_{0,t} \to D_0$ as $t \to 0$.

%%%%%%%%%%%%%%%%%%%%%%%%%%%%%%%%%%%%%%%%%%%%%%%%%%%%%%%%%%%%%%%%%%%%%%%%%%%%%%%%%%%%%%%%%%%%%%%%%%%%%%%%%%%%%%%%%%%%%%%%%%%%%%%%%%%%%%%
%%%%%%%%%%%%%%%%%%%%%%%%%%%%%%%%%%%%%%%%%%%%%%%%%%%%%%%%%%%%%%%%%%%%%%%%%%%%%%%%%%%%%%%%%%%%%%%%%%%%%%%%%%%%%%%%%%%%%%%%%%%%%%%%%%%%%%%

\begin{appendices}
\section{\hspace{-0.5cm}: The construction of $\sigma_a$}
Letting $\sigma_0$ be the solution of \eqref{eq sigma0} of the form \eqref{sigma0}, where $c_0$ is given by \eqref{c0}, we have that $Q_0(z)=\frac{\sigma_0(z)}{z^{n+1}}$ is an holomorphic function near $z=0$ so that $Q_0(0)=\frac{n+1}{\mathcal{H}(0)}$ (see \eqref{0942}). Since $Q_0(0)\not=0$, the $(n+1)-$root $Q_0^{\frac{1}{n+1}}$ of $Q_0$ is a well-defined holomorhpic function locally at $z=0$, and it makes sense to define $q_0(z)=z Q_0^{\frac{1}{n+1}}(z)$ near $z=0$.

\medskip \noindent For $\sigma \in \mathcal{B}_r$, where $\mathcal{B}_r$ is given in \eqref{setB}, in a similar way we have that $Q(z)=\frac{\sigma(z)}{z^{n+1}}$ is an holomorphic function near $z=0$ with $|\frac{Q(z)}{Q_0(z)}-1| \leq r$ for all $z$. Since in particular $|Q(z)-\frac{n+1}{\mathcal{H}(0)}|\leq r|Q_0(z)|+|Q_0(z)-\frac{n+1}{\mathcal{H}(0)}|$, we can find $r$ and $\eta>0$ small so that $q(z)=z Q^{\frac{1}{n+1}}(z)$ is a well-defined holomorphic function in $B_{3\eta}(0)$ for all $\sigma \in \mathcal{B}_r$, with $\sigma(z)=q^{n+1}(z)$ for all $z \in B_{3\eta}(0)$. Since $q'(0)=Q^{\frac{1}{n+1}}(0)$ satisfies $|q'(0)|\geq [\frac{(1-r)(n+1)}{|\mathcal{H}(0)|}]^{\frac{1}{n+1}}>0$, then $q$ is locally bi-holomorphic at $0$. In order to have uniform invertibility of $q$ for all $\sigma \in \mathcal{B}_r$, let us evaluate the following quantity:
\begin{eqnarray*}
|1-\frac{q'(z)}{q'(0)}|&\leq& \frac{\sup_{B_{\eta}(0)}|q''|}{|q'(0)|} |z|\leq
 \frac{2}{\eta^2}[\frac{(1-r)(n+1)}{|\mathcal{H}(0)|}]^{-\frac{1}{n+1}} \left(\sup_{B_{2\eta}(0)}|q|\right) |z| \\
&\leq& \frac{2}{\eta^2} \left(\frac{|\mathcal{H}(0)|}{n+1}\right)^{\frac{1}{n+1}} (\frac{1+r}{1-r})^{\frac{1}{n+1}}  \left(\sup_{B_{2\eta}(0)}|q_0|\right) |z|
\end{eqnarray*}
for all $z \in B_\eta(0)$, in view of the Cauchy's inequality and $|\frac{\sigma(z)}{\sigma_0(z)}-1|= |\frac{q^{n+1}(z)}{q^{n+1}_0(z)}-1| \leq r$ for all $z \in B_{3\eta}(0)$. Therefore, we can find $\rho_1$ small so that $|1-\frac{q'(z)}{q'(0)}|\leq \frac{1}{2}$ for all $z \in B_{\rho_1^{\frac{1}{n+1}}}(0)$ and $2\rho_1^{\frac{1}{n+1}} |Q(0)|^{-\frac{1}{n+1}}\leq 2\rho_1^{\frac{1}{n+1}}[\frac{|\mathcal{H}(0)|}{n+1}]^{\frac{1}{n+1}} (1-r)^{-\frac{1}{n+1}}\leq 2 \eta$, uniformly for $\sigma\in \mathcal{B}_r$. Hence, the inverse map $q^{-1}$ of $q$ is defined from $B_{\rho_1^{\frac{1}{n+1}}}(0)$ into $B_{2\rho_1^{\frac{1}{n+1}} |Q(0)|^{-\frac{1}{n+1}}}(0)$: for all $y \in B_{\rho_1^{\frac{1}{n+1}}}(0)$ there exists a unique $z \in B_{2\rho_1^{\frac{1}{n+1}} |Q(0)|^{-\frac{1}{n+1}}}(0)$ so that $q(z)=y$,  given by $z=q^{-1}(y)$. Since $\sigma=q^{n+1}$ in $B_{3\eta}(0)$, we have that
$$\hbox{Card } \{z\in B_{2\rho_1^{\frac{1}{n+1}} |Q(0)|^{-\frac{1}{n+1}}}(0): \sigma(z)=y\}=n+1 \qquad \forall\: y \in B_{\rho_1}(0)\setminus \{0\},$$
for all $\sigma\in \mathcal{B}_r$. Since
$$|\sigma(z)|\geq (1-r) \inf_{\tilde \Omega \setminus B_{2\rho_1^{\frac{1}{n+1}} |Q(0)|^{-\frac{1}{n+1}}}(0)} |\sigma_0(z)|
\geq (1-r) \inf_{\tilde \Omega \setminus B_{2\rho_1^{\frac{1}{n+1}} [\frac{|\mathcal{H}(0)|}{n+1}]^{\frac{1}{n+1}} (1+r)^{-\frac{1}{n+1}}}(0)} |\sigma_0(z)|>0 $$
for all $z \in \tilde \Omega \setminus B_{2\rho_1^{\frac{1}{n+1}} |Q(0)|^{-\frac{1}{n+1}}}(0)$ we can find $\rho$ ($\leq \rho_1$) small so that
$$\hbox{Card } \{z\in \tilde \Omega: \sigma(z)=y\}=\hbox{Card } \{z\in B_{2\rho_1^{\frac{1}{n+1}} |Q(0)|^{-\frac{1}{n+1}}}(0): \sigma(z)=y\}=n+1 \qquad \forall\: y \in B_{\rho}(0)\setminus \{0\},$$
for all $\sigma\in \mathcal{B}_r$. Since
$$\sigma^{-1}(B_\rho(0)) \subset  B_{2\rho_1^{\frac{1}{n+1}} |Q(0)|^{-\frac{1}{n+1}}} (0) \subset  B_{2\rho_1^{\frac{1}{n+1}}[\frac{|\mathcal{H}(0)|}{n+1}]^{\frac{1}{n+1}} (1-r)^{-\frac{1}{n+1}} }(0) \subset B_{2\eta}(0),$$
for all $z\in \partial \sigma^{-1}(B_\rho(0))=\sigma^{-1}(\partial B_\rho(0))$ and $\sigma \in \mathcal{B}_r$ we have that
$$\frac{|z|^{n+1}}{\rho}=\frac{|z|^{n+1}}{|\sigma(z)|}=\frac{1}{|Q(z)|}\geq \frac{1}{(1+r)} \inf_{ B_{2\eta}(0)}  |Q_0(z)|^{-1}  >0$$
for $q_0$ is well-defined in $B_{3\eta}(0)$. We can summarize the above discussion as follows:
\begin{lem} \label{gomme}
There exist $r,\:\rho >0$ such that $q(z)=z Q(z)^{\frac{1}{n+1}}$ is a locally bi-holomorphic map with $\sigma=q^{n+1}$ and inverse $q^{-1}$ defined on $B_{\rho^{\frac{1}{n+1}}}(0)$, for all $\sigma\in \mathcal{B}_r$. In particular, there exists a neighborhhod $V$ of $0$ so that, for all $\sigma \in \mathcal{B}_r$, there hold $V \subset \sigma^{-1}(B_\rho(0))$ and $\sigma:\sigma^{-1}(B_\rho(0)) \to B_\rho(0)$ is a $(n+1)-1$ map in the following sense:
$$\hbox{Card } \{z\in \tilde \Omega:\sigma(z)=y\}=n+1 \qquad \forall\: y \in B_\rho(0)\setminus \{0\}.$$
\end{lem}

\medskip \noindent For $|a|<\rho$ and $\sigma\in \mathcal{B}_r$, by Lemma \ref{gomme} we have that
$$\sigma^{-1}(a)=\{z \in \tilde \Omega: \: \sigma(z)=a\}=\{a_0,\dots,a_n \},$$
where $a_k=q^{-1}(\hat a_k)$ and $\hat a_k$, $k=0,\dots,n$, are the $(n+1)-$roots of $a$, and then
$g_{a,\sigma}(z):=\displaystyle \frac{\sigma(z)-a}{\prod_{k=0}^{n}(z-a_k)} \in \mathcal{M}(\overline{\Omega})$ is a non-vanishing function. We are now in position to prove the following.
\begin{lem}\label{derivca}
Up to take $\rho$ smaller, there exists a $C^1-$map $a \in
B_\rho(0) \to \sigma_a \in \mathcal{B}_r$ so that $\sigma_a$ solves \eqref{sigmaa}-\eqref{ca}.
Moreover, the map $a \in B_\rho(0) \to c_a=c_{a,\sigma_a}$ is $C^1$
with
\begin{eqnarray*}
&& \Gamma:=\mathcal{H}(0) \partial_a c_a
\Big|_{a=0}=\frac{1}{n!}\frac{d^{n+1}}{dz^{n+1}}\bigg[
\mathcal{H}(z)f_{n+1}(z)\bigg]  (0)\\
&&\Upsilon:=\mathcal{H}(0) \partial_{\bar a} c_a \Big|_{a=0}=-{2\pi(n+1)\over |\Om|
n!}\overline{b_{n+1}}\,\frac{d^n \mathcal{H} }{dz^n}(0),
\end{eqnarray*}
where
$$f_{n+1}(z)=\frac{1}{(n+1)!} \frac{d^{n+1}}{dw^{n+1}}\left[2\log \frac{w-q_0(z)}{q_0^{-1}(w)-z}+4\pi
H^*(z-q_0^{-1}(w))\right] (0)\:,\qquad b_{n+1}=\frac{1}{(n+1)!}\frac{d^{n+1} q_0^{-1}}{dw^{n+1}}(0).$$
\end{lem}
\begin{proof}
Given $c_{a,\sigma}$ as in \eqref{ca}, equation \eqref{sigmaa}
is equivalent to find zeroes
of the map $\Lambda: (a,\sigma) \in B_\rho(0) \times \mathcal{B}_r \to
\mathcal{M}(\overline{\Omega})$ given as
$$\Lambda(a,\sigma)=
\sigma(z)+\left[ \int^z  \frac{g^2_{a, \sigma}(w)}{g^2_{0, \sigma}(w)}   \frac{\mathcal{H}_{a,\sigma}(w)}{w^{n+2}}  e^{-c_{a,\sigma}w^{n+1}} dw \right]^{-1}.$$
Observe that the zeroes $a_k=a_k(a,\sigma)=q^{-1}(\hat a_k)$ are continuously
differentiable in $\sigma$. Differentiating the relation
$\sigma(a_k)=a$ at $\sigma_0$ along a direction $R \in
\mathcal{M}'(\overline{\Omega})$, we have that $\sigma_0'(a_k(a,\sigma_0)) \partial_\sigma a_k(a,\sigma_0) [R]+R(a_k(a,\sigma_0))=0$. Since $\sigma_0'(a_k) \sim a_k^n$ and $R(a_k)\sim a_k^{n+1}$ in view of $\|R\|<\infty$, we get that $\partial_\sigma a_k(0,\sigma_0)[R]=0$ for all $R \in \mathcal{M}'(\overline{\Omega})$. For $z \not= 0$ the function $\frac{g_{a,\sigma}(z)}{g_{0,\sigma}(z)}$ is continuously differentiable in $\sigma$ with
$$\partial_\sigma \left(\frac{g_{a,\sigma}(z)}{g_{0,\sigma}(z)} \right) [R]=a \frac{z^{n+1}}{\prod_{k=0}^{n}(z-a_k)} \frac{R(z)}{\sigma^2(z)}
+ \frac{\sigma(z)-a }{\prod_{k=0}^{n}(z-a_k)}\frac{z^{n+1}}{\sigma(z)} \sum_{j=0}^n  \frac{1}{z-a_j}  \partial_\sigma a_j(a,\sigma) [R]$$ for every $R \in
\mathcal{M}'(\overline{\Omega})$. In particular, we get that
$\partial_\sigma \left(\frac{g_{a,\sigma}(z)}{g_{0,\sigma}(z)} \right)\Big|_{a=0} [R]=0$ for every $z \not= 0$ and $R \in
\mathcal{M}'(\overline{\Omega})$. Since we can write $\frac{g_{a,\sigma}(z)}{g_{0,\sigma}(z)}$ as
\begin{equation} \label{1310}
\frac{g_{a,\sigma}(z)}{g_{0,\sigma}(z)}=\frac{z^{n+1}}{\sigma(z)} \prod_{k=0}^{n} \frac{q(z)-q(a_k)}{z-a_k}=
\frac{z^{n+1}}{\sigma(z)} \prod_{k=0}^{n} \int_0^1 q'(a_k+t(z-a_k))dt
\end{equation}
for $z$ small in view of $\sigma=q^{n+1}$, we get that $\frac{g_{a,\sigma}(z)}{g_{0,\sigma}(z)}$ is continuously differentiable in $\sigma$ and
the linear operator $\partial_\sigma \left( \frac{g_{a,\sigma}(z)}{g_{0,\sigma}(z)}\right)$ is continuous at
$z=0$. In particular, we get that $\partial_\sigma\left(\frac{g_{a,\sigma_0}(z)}{g_{0,\sigma_0}(z)}\right)\Big|_{a=0}[R]=0$ for
every $z$ and $R \in \mathcal{M}'(\overline{\Omega})$. By \eqref{Hasigma} we have that $\mathcal{H}_{a,\sigma}$ is continuously differentiable in $\sigma$ with $\partial_\sigma \mathcal{H}_{0,\sigma}[R]=0$ for every $R \in
\mathcal{M}'(\overline{\Omega})$. We have that  $c_{a,\sigma}$ is also continuosuly differentiable in $\sigma$ with $\partial_\sigma c_{0,\sigma_0}[R]=0$ for every $R \in
\mathcal{M}'(\overline{\Omega})$, and so $\Lambda (a,\sigma)$ is with $\partial_\sigma \Lambda(0,\sigma_0)=\hbox{Id}$.

\medskip \noindent Since $a_k \sim |a|^{\frac{1}{n+1}}$, the smooth dependence in $a$ is much more delicate, and will be true just for symmetric expressions of the $a_k$'s thanks to the symmetries of $\hat a_k=q(a_k)$. To fully exploit the symmetries, it is crucial that the expression \eqref{Hasigma} of $\mathcal{H}_{a,\sigma}$ is in terms of an holomorphic function $H^*$. Indeed, we have that
\begin{eqnarray*}
2\sum_{k=0}^n H^*(z-a_k)-\frac{z}{|\Omega|} \overline{\sum_{k=0}^n a_k}
&=&2 \sum_{l=0}^{\infty} g_l(z) \sum_{k=0}^n \hat a_k^l
-{z \over
|\Om|}\overline{\sum_{l=1}^\infty b_l \sum_{k=0}^n \hat a_k^l}\\
&=& 2 (n+1)\sum_{l=0}^{\infty} g_{(n+1)l}(z) a^l -{n+1 \over
|\Om|}z\overline{\sum_{l=1}^\infty b_{(n+1)l} a^l}
\end{eqnarray*}
in view of $\displaystyle \sum_{k=0}^n \hat a_k^l=0$ for all $l \notin (n+1)\mathbb{N}$, where $g_l(z)=\frac{1}{l!}\frac{d^l}{dw^l} [H^*(z-q^{-1}(w))](0)$ and $b_l=\frac{1}{l!}\frac{d^l q^{-1}}{dw^l}(0)$ (recall that $b_0=q^{-1}(0)=0$). Since for $z$ small there holds
\begin{eqnarray*}
 \sum_{k=0}^n \log \frac{q(z)-q( a_k)}{z-a_k}=\sum_{l=0}^\infty h_l(z) \sum_{k=0}^n \hat a_k^l=(n+1) \sum_{l=0}^\infty h_{(n+1)l}(z) a^l
\end{eqnarray*}
in view of $a_k=q^{-1}(\hat a_k)$, where $h_l(z)=\frac{1}{l!} \frac{d^l}{dw^l} \left[\log \frac{w-q(z)}{q^{-1}(w)-z}\right](0),$
we have that $\frac{g_{a,\sigma}(z)}{g_{0,\sigma}(z)}$  is continuously differentiable in $a,\, \bar a$ for all $z$ in view of \eqref{1310} (for $z$ far from $0$ it is obvious). Hence,  by \eqref{Hasigma} $\frac{g_{a,\sigma}^2}{g_{0,\sigma}^2} \mathcal{H}_{a,\sigma}$, $c_{a,\sigma}$ and $\Lambda(a,\sigma)$ are continuously differentiable also in $a,\, \bar a$, and then $\Lambda$ is  a $C^1-$map with $\Lambda(0,\sigma_0)=0$, $\partial_\sigma
\Lambda(0,\sigma_0)=\hbox{Id}$. Up to take $\rho$ smaller, by the Implicit
Function Theorem we find a $C^1$-map $a \in B_\rho(0) \to \sigma_a$ so
that $\Lambda(a,\sigma_a)=0$, and the function $a \to c_a=c_{a,\sigma_a}$ is $C^1$. By
$$\partial_a [\frac{g_{a,\sigma}^2(z) g_{0,\sigma}^2(0)}{g_{a,\sigma}^2(0) g_{0,\sigma}^2(z) }  \frac{\mathcal{H}_{a,\sigma}(z)}{\mathcal{H}_{a,\sigma}(0)}](0)=
\frac{g_{0,\sigma}^2(0)}{g_{0,\sigma}^2(z) }\partial_a [e^{2\log g_{a,\sigma}(z)-2\log g_{a,\sigma}(0)}\frac{\mathcal{H}_{a,\sigma}(z)}{\mathcal{H}_{a,\sigma}(0)}](0)=
(n+1)\frac{\mathcal{H}(z)}{\mathcal{H}(0)}[f_{n+1}(z)-f_{n+1}(0)]$$
and
$$\partial_{\bar a} [\frac{g_{a,\sigma}^2(z) g_{0,\sigma}^2(0)}{g_{a,\sigma}^2(0) g_{0,\sigma}^2(z) }  \frac{\mathcal{H}_{a,\sigma}(z)}{\mathcal{H}_{a,\sigma}(0)}](0)=
-\frac{2\pi (n+1)}{|\Omega|} \frac{\mathcal{H}(z)}{\mathcal{H}(0)} \overline{b_{n+1}}z$$
we deduce the desired expression for $\Gamma$ and $\Upsilon$ in view of $\partial_\sigma c_{0,\sigma_0}=0$ and \eqref{pc}. \qed \end{proof}

\section{\hspace{-0.5cm}: The linear theory}

\noindent In this section, we will prove the invertibility of the
linear operator $L$ given by (\ref{ol}) under suitable
orthogonality conditions. The operator $L$ can be described asymptotically by the following
linear operator in $\R^2$
$$L_0(\phi)=\Delta\phi+{8(n+1)^2|y|^{2n}\over (1+|y^{n+1}-\zeta_0|^2)^2}\phi,$$
where $\zeta_0=\lim \frac{a}{\delta}$. When $\zeta_0=0$, as in the case $n=0$ \cite{BaPa}, by using a Fourier decomposition
of $\phi$ it can be shown in a rather direct way that the bounded
solutions of $L_0(\phi)=0$ in $\R^2$ are precisely linear combinations of
\begin{equation*}
Y_{0}(y) = \,{1-|y|^{2n+2}\over 1+|y|^{2n+2}}
\qquad\text{and}\qquad Y_l(y) = { (y^{n+1})_l \over
1+|y|^{2n+2}} ,\:l=1,2.
\end{equation*}
Note that $L_0$ is the linearized operator at the radial solution
$U=U_{1,0}$ of $-\Delta U=|z|^{2n} e^U$.

\medskip \noindent For the linearized operator at $U_{1,\zeta_0}$ with $\zeta_0 \not=0$, the Fourier decomposition
is useless since $U_{1,\zeta_0}$ is not radial w.r.t. any point if $n \geq 1$. However, the same property is still true as recently
proved in \cite{DEM5}, and the argument below could be carried out in full generality in the range $a=O(\delta)$. Since in Theorem \ref{main} we are concerned with the case $a=o(\delta)$, for simplicity we will discuss the linear theory just in this case.

\medskip \noindent Recall that
\begin{equation*}
Z_{0}(z) = {\delta^2-|\sigma(z)-a|^2\over
\de^2+|\sigma(z)-a|^2}, \qquad\quad Z_l(z) =
{\de[\sigma(z)-a]_l \over
\de^2+|\sigma-a|^2},\qquad l=1,2,
\end{equation*}
and $PZ_l$, $l=0,1,2$, denotes the projection of $Z_l$ onto the doubly-periodic functions with
zero average:
\begin{equation*}
\left\{ \begin{array}{ll} \Delta PZ_l =\Delta
Z_l-\frac{1}{|\Om|}\int_\Om \Delta Z_l & \text{in
$\Omega$}\\
\int_\Om PZ_l=0.&
\end{array}\right.
\end{equation*}
Given $h\in L^\infty(\Om)$ with $\int_\Om h=0$, consider the
problem of finding a function $\phi$ in $\Om$ with zero average
and numbers $d_l$, $l=0,1,2$, such that
\begin{equation}\label{plco}
\left\{ \begin{array}{ll}
L(\phi) =h + \displaystyle \sum_{l=0}^{2}d_l \Delta PZ_l &\text{ in $\Om$}\\
\int_{\Omega } \Delta PZ_l \phi = 0   &\forall\: l=0,1,2.
\end{array} \right.
\end{equation}
Since $Z=Z_1+iZ_2$, observe that (\ref{plco}) is equivalent to
solve (\ref{plcobis}) with $d=d_1-id_2$. Let us stress that
the orthogonality conditions in \eqref{plco} are taken with
respect to the elements of the approximate kernel due to
translations and to an extra element which involves dilations. A
similar situation already appears in \cite{DDeMW}.

\medskip \noindent First, we will prove an a-priori estimate for problem
\eqref{plco} when $d_l=0$ for all $l=0,1,2$, w.r.t. the
$\|\cdot\|_*$-norm defined as
$$\|h\|_*=\sup_{z\in \Om}{(\de^2+|\sigma(z)-a|^2)^{1+\gamma/2}\over \de^\gamma(|\sigma'(z)|^2+\de^{2n\over n+1})}|h(z)|,$$
where $0<\gamma <1$ is a small fixed constant.
\begin{prop} \label{p1} There exist $\eta_0>0$ small and $C>0$ such that
for any $0<\de\leq \eta_0$, $\e^2\leq \eta_0 \de^{2\over n+1}$,
$|a|\leq \eta_0 \de$ and any solution $\phi$ to
\begin{equation}\label{plco1}
\left\{ \begin{array}{ll}
L(\phi)=h &\text{in }\Om\\
\int_{\Omega } \Delta PZ_l \phi = 0 &\forall\:l=0,1,2\\
\int_\Omega \phi=0,&
\end{array} \right.
\end{equation}
one has
\begin{equation}\label{est}
\|\phi \|_\infty \le C \log \frac{1}{\de} \|h\|_*.
\end{equation}
\end{prop}
\begin{proof} The proof of estimate \equ{est} consists of several steps. Assume by contradiction the existence of sequences $\de_k \to 0$, $\e_k$ with $\e_k^2=o(\de_k^{2\over n+1})$, $a_k$ with $a_k=o(\de_k)$, functions $h_k$ with $|\log \de_k| \, \|h_k\|_*=o(1)$ as $k \to +\infty$, and solutions $\phi_k$ of (\ref{plco1}) with $\|\phi_k\|_\infty=1$. Since by (\ref{ol}) the operator $L$ acts as $L(\phi) = \Delta \phi + \ml{K} \left[ \phi+ \gamma(\phi)\right]$, where $\gamma(\phi) \in \mathbb{R}$, the function $\psi_k=\phi_k+\gamma(\phi_k)$ does solve
\begin{equation*}
\left\{ \begin{array}{ll}
\Delta \psi_k+\ml{K}_k \psi_k= h_k &\text{in $\Om$}\\
\int_{\Omega } \Delta PZ_{k,l} \psi_k= 0 &\forall \: l=0,1,2,
\end{array}\right.
\end{equation*}
where $W_k$, $\ml{K}_k$, $Z_{k,l}$ denote the functions $W$, $\ml{K}$, $Z_l$,
respectively, along the given sequence.

\begin{claim}
$\displaystyle \liminf_{k \to +\infty} \|\psi_k\|_\infty
>0$ and, up to a subsequence, $\psi_k \to \ti c \in \mathbb{R}$ as $k \to+\infty$ in $C^{1,\alpha}_{\hbox{loc}}(\bar \Om\sm\{0\})$, for all $\alpha \in (0,1)$.
\end{claim}
\noindent Indeed, assume by contradiction that $\displaystyle
\liminf_{k \to +\infty} \|\psi_k\|_\infty =0$. Up to a
subsequence, assume that
$\|\psi_k\|_\infty=\lf\|\phi_k+\gamma(\phi_k)\rg\|_\infty\to 0$ as
$k\to+\infty$. Since $\e_k^2=o(\de_k^{2\over n+1})$, by (\ref{BW})
it follows that
$$\gamma(\phi_k)=-\frac{\int_\Om e^{u_0+W_k}\phi_k}{\int_\Om
e^{u_0+W_k}}+o(1)=O(1).$$ Up to a subsequence we have that
$\frac{\int_\Om e^{u_0+W_k}\phi_k}{\int_\Om e^{u_0+W_k}} \to c$, and then
$\phi_k \to c$ uniformly in $\Omega$ as $k \to +\infty$. Since
$\int_\Om\phi_k=0$, we get $c=0$ and $\phi_k \to 0$ in
$L^\infty(\Omega)$, in contradiction with $\|\phi_k\|_\infty=1$.
Moreover, since $\|\psi_k\|_\infty=O(1)$, by
(\ref{eps1})-(\ref{eps2}) we have that $\Delta\psi_k=o(1)$ in
$C_{\hbox{loc}}(\bar \Omega \setminus \{0\})$. Up to a
subsequence, we have that $\psi_k \to \psi$ as $k\to+\infty$ in
$C^{1,\alpha}_{\hbox{loc}}(\bar \Om\sm\{0\})$. Since $\|\psi_k
\|_\infty =O(1)$, $\psi$ is a bounded function which can be
extended to a harmonic doubly-periodic function in $\Om$.
Therefore, $\psi=\ti c$ in $\Om$ with $\ds\ti
c=\lim_{k\to+\infty}\gamma(\phi_k)$, since $\ds{1\over
|\Om|}\int_\Om\psi_k=\gamma(\phi_k)$.

\medskip \noindent Now, consider the function $\Psi_{k}(y)=\psi_k ( \delta_k^{1\over n+1} y)$. Then, $\Psi_k$ satisfies
$$\Delta \Psi_{k} + K_{k}(y)\Psi_{k} =\hat h_{k}(y)\qquad\text{in }
\de_k^{-\frac{1}{n+1}} \Omega ,$$ where $
K_{k}(y)=\de_k^{2\over n+1} \ml{K}_k (\de_k^{1\over n+1}y)$ and
$ \hat h_{k}(y)=\de_k^{2\over n+1} h_k (\de_k^{1\over n+1} y)$.
Also, we set $\sigma_k(y)=\delta_k^{-1} \sigma_{a_k}(\de_k^{1\over n+1}y)$ for $y$ in
compact subsets of $\R^2$.

\begin{claim}
$\Psi_{k} \to \Psi=0$ in $C_{\hbox{loc}}(\R^2)$ as $k\to+\infty$.
\end{claim}
\noindent Indeed, observe that by (\ref{BW}) and
(\ref{eps1})-(\ref{eps2}) we have the following expansions:
\begin{equation} \label{mlK}
\ml{K}(z)=|\sigma'(z)|^2 e^{U_{\de,a}}
\left[1+O(|c_a||z|^{n+1})+O(|c_a||a|+\delta^2|\log \delta|) \right]+O(\epsilon^2
|\sigma'(z)|^4 e^{2U_{\de,a}}) .
\end{equation}
Since $\e_k^2=o(\de_k^{2\over n+1})$, the first
estimate above re-writes along our sequence as
$$ K_{k}(y)=(1+o(1)+O(\delta_k|y|^{n+1})){8|\sigma_k'(y)|^2  \over \lf(1+\big|\sigma_k(y)-a_k \de_k^{-1}\big|^2\rg)^2}+o(1) {64|\sigma_k'(y)|^4 \over \lf(1+\big|\sigma_k(y)-a_k \de_k^{-1}\big|^2\rg)^4}$$
uniformly in $\de_k^{-\frac{1}{n+1}} \Omega$ as $k\to+\infty$.
Since $\sigma=z^{n+1}Q$, we have that $\sigma_k(y)= y^{n+1} Q_{a_k}( \delta_k^{\frac{1}{n+1}} y)$  and $\sigma_k'(y)=(n+1)  y^n Q_{a_k}(\delta_k^{\frac{1}{n+1}} y)+
\delta_k^{\frac{1}{n+1}}  y^{n+1} Q_{a_k}'(\delta_k^{\frac{1}{n+1}} y)$. Since $Q_{a_k}(0) \to \frac{n+1}{\mathcal{H}(0)}=:\gamma\not=0$ and $\|Q_{a_k}'\|_{\infty,\Omega}\leq
C \|Q_{a_k}\|_{\infty,\tilde \Omega}\leq C'$, we have
that
$$\sigma_k(y)=y^{n+1}[\gamma+o(1)+O(\de_k^{1\over n+1}|y|)] \:,\qquad
\sigma_k'(y)=(n+1) y^n [\gamma+o(1)+  O(\delta_k^{1 \over n+1}|y|)]$$
as $k \to +\infty$. Then we get that
\begin{equation} \label{Kk}
K_{k}(y)=\left[{8(n+1)^2 \gamma|^2 |y|^{2n} \over \lf(1+\big|\sigma_k(y)-a_k
\de_k^{-1}\big|^2\rg)^2}+o(1) {64(n+1)^4|\gamma|^4 |y|^{4n} \over
\lf(1+\big|\sigma_k(y)-a_k \de_k^{-1}\big|^2\rg)^4} \right][1+o(1)+O(\de_k^{1\over
n+1}|y|)]
\end{equation}
uniformly in $\de_k^{-\frac{1}{n+1}} \Omega$. Choose $\eta$ small so that $|\sigma_k(y)|\geq
\frac{|\gamma|}{2}|y|^{n+1}$ in $B_{\de_k^{-\frac{1}{n+1}}\eta}(0)$ for $k$ large. Since $\|\Psi_k\|_\infty=O(1)$ and $|\hat h_k(y)|\le C\|h_k \|_*
\to 0$ on compact sets, by elliptic estimates and (\ref{Kk}) we get that $\Psi_k(\gamma^{-\frac{1}{n+1}} y) \to \hat \Psi$ in $C_{\hbox{loc}}(\R^2)$ as $k \to+\infty$, where
$\hat \Psi$ is a bounded solution of $L_0(\hat \Psi) = 0$ (with $\zeta_0=0$). Then
$\hat \Psi(y)=\displaystyle \sum_{j=0}^2
b_{j}Y_j(y)$ for some $b_{j}\in\R$, $j=0,1,2$.\\
Since $\lap Z_{k,l}+|\sigma_k'|^2 e^{U_{\de_k,a_k}}Z_{k,l}=0$ for $l=0,1,2$ (where $U_{\de_k,a_k}$ stands for $U_{\de_k,a_k,\sigma_{a_k}}$), for $l=1,2$ we have that
\begin{equation*}
\begin{split}
\int_\Om \psi_k \Delta Z_{k,l}
=-\int_\Omega |\sigma_k'(z)|^2 \psi_k e^{U_{\de_k,a_k}}Z_{k,l}=-\int_{B_{\de_k^{-\frac{1}{n+1}} \eta}(0)}
{8 |\sigma_k'(z)|^2 (\sigma_k-a_k \de_k^{-1}) \Psi_k \over
(1+|\sigma_k-a_k \de_k^{-1} |^2 )^3}\,dy+O(\de_k^3).
\end{split}
\end{equation*}
Since for all $l=0,1,2$
$$0=\int_\Om \psi_k\Delta PZ_{k,l}=\int_\Om \psi_k \lf[ \Delta Z_{k,l} -
{1\over |\Om|}\int_\Om \Delta Z_{k,l}\rg]=\int_\Om \psi_k \Delta Z_{k,l}+o(1)$$
as $k \to \infty$ in view of \eqref{deltaZ0}-\eqref{deltaZ}, by dominated convergence we get that
$$\int_{\R^2} \hat \Psi(y)\,\frac{|y|^{2n}(y^{n+1})_l}{(1+|y|^{2n+2})^3}\, dy =0 \qquad
\hbox{for }l=1,2,$$ and we conclude that $b_{1}=b_{2}=0$.
Similarly, for $l=0$ we deduce that
$$\int_{\R^2} \hat \Psi(y)\,{|y|^{2n}(1-|y|^{2n+2})\over (1+|y|^{2n+2})^3}\, dy=0,$$
which implies that $b_0=0$. Thus, the claim follows.

\medskip \noindent On the other hand, from the equation of $\psi_k$ we have the following integral representation
\begin{equation}\label{irsn}
\psi_k(z)={1\over |\Om|}\int_{\Om} \psi_k +\int_\Om G(y,z)
\lf[\ml{K}_k(y) \psi_k(y)-h_k(y) \rg]\, dy.
\end{equation}
\begin{claim} $\ti c=0$
\end{claim}
\noindent Indeed, Claims 1 and 2 imply that  $\psi_k(0)=\Psi_k(0)
\to 0$ and ${1\over |\Om|}\int_{\Om} \psi_k =\gamma(\phi_k)\to \ti c$ as $k \to
+\infty$ by definition. So, by \eqref{irsn} we deduce
that
$$\int_\Om G(y,0) \lf[\ml{K}_k(y) \psi_k(y)-h_k(y) \rg]\, dy \to -\ti c $$
as $k \to +\infty$. Now, we first estimate the integral involving
$h_k$. Since $\int_{B_{\de_k}(0)}|\log|y||\,dy=O(\de_k^2  \log
\de_k),$ we get that
\begin{equation*}
\lf|\int_{B_{\de_k}(0)} G(y,0) h_k(y) dy\rg| \le {C\over
\de_k^2}\|h_k\|_* \int_{B_{\de_k}(0)} G(y,0) dy \le C
|\log\de_k| \|h_k\|_*.
\end{equation*}
By \eqref{1458} we have that
\begin{equation*}
\lf|\int_{\Om\sm B_{\de_k}(0)} G(y,0) h_k(y) dy\rg| \le C
|\log\de_k| \int_\Omega |h_k|\leq C |\log \de_k| \|h_k\|_*,
\end{equation*}
and we conclude that
$$\lf|\int_{\Om} G(y,0) h_k(y) dy\rg|\le C|\log \de_k|\|h_k\|_* \to 0$$
in view of $|\log \de_k| \, \|h_k\|_*=o(1)$ as $k \to +\infty$. By
\eqref{mlK} we have that
\begin{equation*}
\begin{split}
&\int_{\Om} G(y,0) \ml{K}_k(y) \psi_k(y) dy=\int_{B_\eta(0)}
G(y,0) \ml{K}_k(y) \psi_k(y)
dy+O(\de_k^2)\\
&=\int_{B_{ \de_k^{-\frac{1}{n+1}}\eta}(0)} \bigg[-{1\over 2\pi
}\log |y|-{1\over 2\pi(n+1)} \log \de_k+H( \de_k^{1\over n+1}
y,0)\bigg]  K_k(y) \Psi_k(y) dy+O(\de_k^2).
\end{split}
\end{equation*}
Since by (\ref{Kk}) $K_{k}=O({|y|^{2n} \over (1+|y|^{2n+2})^2
})$ does hold uniformly in $B_{\de_k^{-\frac{1}{n+1}}\eta}(0)
\setminus B_1(0)$ and $ K_{k}(y) \to {8(n+1)^2|y|^{2n} \over
(1+|y|^{2n+2})^2}$ as $k\to+\infty$, by dominated convergence we
get that
\begin{eqnarray*}
&&\int_{B_{\de_k^{-\frac{1}{n+1}}\eta}(0)} \bigg[-{1\over 2\pi }\log |y|+H( \de_k^{1\over n+1} y,0)\bigg] K_k(y) \Psi_k(y) dy\\
&& \to \int_{\mathbb{R}^2} \bigg[-{1\over 2\pi }\log |y|+H(0,0)\bigg] {8(n+1)^2|y|^{2n} \over (1+|y|^{2n+2})^2} \Psi(y) dy = 0
\end{eqnarray*}
as $k \to +\infty$. Since $\int_\Om h_k=0$, the integration of the
equation satisfied by $\psi_k$ gives that $\int_\Om \ml{K}_k
\psi_k=0$. Then, by \eqref{mlK} we get that
$$\int_{B_{\de_k^{-\frac{1}{n+1}}\eta}(0)} K_k \Psi_k dy=\int_{B_\eta(0)}  \ml{K}_k \psi_k dy=-\int_{\Om\sm
B_\eta(0)} \ml{K}_k \psi_k=O(\de_k^2),$$ which implies that
$$\log \de_k \int_{B_{\de_k^{-\frac{1}{n+1}}\eta}(0)} K_k \Psi_k dy=O(\de_k^2 \log \de_k).$$
In conclusion, we have shown that $\int_{\Om} G(y,0) \ml{K}_k(y)
\psi_k(y)dy \to 0$ as $k \to +\infty$, yielding to $\ti c=0$.

\medskip \noindent In the following Claims, we will omit the subscript $k$. Let us denote $\ti L(\psi)=\lap\psi + \ml{K} \psi$.
\begin{claim} The operator $\ti L$ satisfies the maximum principle in $
B_\eta(0) \sm B_{R\de^{1\over n+1}}(0)$ for $R$ large enough.
\end{claim}
\noindent Indeed, as already noticed in the proof of the previous
Claim in terms of $K_{k}$, there is $C_1>0$ \st
\begin{equation} \label{salsi}
\ml{K}(z)\le C_1 {(n+1)^2  \de^2|z|^{2n}\over
(\de^2+|z|^{2n+2})^2}
\end{equation}
in $B_\eta(0)\setminus B_{\de^{1\over n+1}}(0)$. The function
$$\ti Z(z)=- Y_0\lf({ \mu z\over\de^{1\over n+1}}\rg)=\frac{\mu ^{2n+2}|z|^{2n+2}-\de^2}{\mu^{2n+2}|z|^{2n+2}+\de^2}$$
satisfies
$$-\lap \ti Z(z)=16(n+1)^2 \frac{\de^2 \mu^{2n+2} |z|^{2n} (\mu^{2n+2}|z|^{2n+2}-\de^2)}{(\mu^{2n+2}|z|^{2n+2}+\de^2)^3}.$$
For $R$ large so that $\mu^{2n+2}R^{2n+2}>{5\over 3}$ we have that
\begin{equation*}
\begin{split}
-\lap \ti Z(z)&\ge 16(n+1)^2\frac{\de^2 \mu^{2n+2} |z|^{2n} }{(\mu^{2n+2}|z|^{2n+2}+\de^2)^2}\,{\mu^{2n+2}R^{2n+2}-1\over \mu^{2n+2}R^{2n+2}+1}\\
&\ge 4(n+1)^2\frac{\de^2 \mu^{2n+2}
R^{4n+4}}{(\mu^{2n+2}R^{2n+2}+1)^2}\,{1\over
|z|^{2n+4}}\ge{(n+1)^2\over \mu^{2n+2}}{\de^2\over |z|^{2n+4}}
\end{split}
\end{equation*}
in $B_\eta(0) \sm B_{R\de^{1\over n+1}}(0)$. On the other hand,
since $\ti Z \le 1$ we have that
$$\ml{K} (z)\ti Z(z)\le C_1{(n+1)^2\de^2|z|^{2n}\over
(\de^2+|z|^{2n+2})^2}\le C_1{(n+1)^2\de^2\over |z|^{2n+4}}$$ in
$B_\eta(0)\setminus B_{\de^{1\over n+1}}(0)$, and for
$0<\mu<{1\over\sqrt{C_1}}$ we then get that
$$\ti L(\ti Z)\le \lf(-{1\over \mu^{2n+2}}+C_1\rg){(n+1)^2 \de^2 \over
|z|^{2n+4}}<0$$ in $B_\eta(0) \sm B_{R\de^{1\over n+1}}(0)$. Since
$$\ti Z(x)\ge {\mu^{2n+2}R^{2n+2}-1\over \mu^{2n+2}R^{2n+2}+1}>{1\over
4}$$ for $|z|\geq R\de^{1\over n+1}$, we have provided the
existence of a positive super-solution for $\ti L$, a sufficient
condition to have that $\ti L$ satisfies the maximum principle.

\begin{claim}
There exists a constant $C>0$ \st
$$\|\psi\|_{\infty, B_\eta(0)\sm B_{R\de^{1\over n+1}}(0)}\le C[\|\psi\|_i+\|h\|_*],$$ where
$$\|\psi\|_i=\|\psi\|_{\infty, \partial B_{R\de^{1\over n+1}}(0)}+\|\psi\|_{\infty, \ptl B_{\eta}(0)}.$$
\end{claim}
\noindent Indeed, letting $\Phi$ be the solution of
$$
\left\{ \begin{array}{ll}
-\lap \Phi=2 \displaystyle \sum_{i=1}^2 {\de^{\sigma_i \over n+1} \over |z|^{2+\sigma_i}}&\hbox{for }R\de^{1\over n+1} \leq |z| \leq r\\
\Phi=0 &\text{for }|z|=r,\, R\de^{1\over n+1}
\end{array}\right.$$
with $r \in (\eta,2\eta)$, $\sigma_1=\sigma (n+1)$ and
$\sigma_2=2n+\sigma (n+1)$, we construct a barrier function of the
form $\ti\Phi=4\|\psi\|_i \ti Z + \|h\|_* \Phi$. A direct
computation shows that
$$\Phi(z)=2 \sum_{i=1}^2 \de^{\sigma_i \over n+1}\lf[-\frac{1}{\sigma_i^2 |z|^{\sigma_i}} + \alpha_i \log |z|
+\beta_i \rg],$$ where
$$\alpha_i={1\over \sigma_i^2 \log {R\de^{1\over n+1}\over r}}\lf({1\over R^{\sigma_i}\de^{\sigma_i \over n+1}}-{1\over
r^{\sigma_i}}\rg)<0,\quad\quad \beta_i={1\over \sigma_i^2
r^{\sigma_i}}-{\log r\over \sigma_i^2\log {R\de^{1\over n+1}\over
r}}\lf({1\over R^{\sigma_i} \de^{\sigma_i \over n+1}}-{1\over
r^{\sigma_i}}\rg)$$ for $i=1,2$. Since
$$0\le \Phi(z)\le 2 \sum_{i=1}^2 \de^{\sigma_i \over n+1}\lf[-{1\over \sigma_i^2 r^{\sigma_i}}+\alpha_i \log
R\de^{1\over n+1}+\beta_i\rg]=2 \sum_{i=1}^2 \de^{\sigma_i \over
n+1} \alpha_i  \log {R\de^{1\over n+1}\over r}\le \sum_{i=1}^2 {2
\over \sigma_i^2 R^{\sigma_i}},$$ we get that
\begin{equation*}
\begin{split}
\ti L(\ti\Phi)&\le \| h\|_*\lf[-2 {\de^{\sigma}\over
|z|^{2+\sigma(n+1)}}-2{\de^{\sigma+{2n \over n+1}} \over
|z|^{2+2n+\sigma(n+1)}}
+C_1(n+1)^2 {\de^2|z|^{2n}\over (\de^2+|z|^{2n+2})^2}\sum_{i=1}^2 {2\over \sigma_i^2 R^{\sigma_i}} \rg]\\
&\le \| h\|_*\lf[-2 {\de^{\sigma}\over |z|^{2+\sigma(n+1)}}-{
\de^{\sigma+{2n \over n+1}} \over (\de^2+|z|^{2n+2})^{1+\sigma/2}
}
+{\de^\sigma |z|^{2n}\over(\de^2+|z|^{2n+2})^{1+\sigma/2} } \rg]\\
&\le - \| h\|_* {\de^\sigma(|z|^{2n}+\de^{2n\over n+1})\over
(\de^2+|z|^{2n+2})^{1+\sigma/2}}
\end{split}
\end{equation*}
in view of (\ref{salsi}), for $R$ large so that $C_1(n+1)^2
\displaystyle \sum_{i=1}^2 {2\over \sigma_i^2 R^{\sigma_i}} \leq
1$. Since $|\psi| \leq \ti \Phi$ on $\ptl B_{R\de^{1\over
n+1}}(0)\cup\ptl B_r(0)$ in view of $4\ti Z\ge 1$, by the maximum
principle we conclude that $|\psi|\le \ti\Phi$ in
$B_\eta(0)\setminus B_{R\de^{1\over n+1}}(0)$ and the claim
follows.

\medskip \noindent Since Claims 2 and 3 provide that $\|\psi_k \|_i \to 0$ as $k \to \infty$, by Claim 5 we conclude that
$\|\psi_k \|_\infty=o(1)$ as $k\to+\infty$, a contradiction with
$\liminf_{k \to +\infty} \|\psi_k\|_\infty>0$ according to Claim
1. This completes the proof. \qed
\end{proof}

\medskip \noindent We are now in position to solve problem \eqref{plco}.
\begin{prop} \label{p2}
There exists $\eta_0>0$ small such that for any $0<\de\leq
\eta_0$, $|\log \delta| \e^2\leq \eta_0 \de^{2\over n+1}$, $|a|\leq \eta_0 \de$
and $h \in L^\infty(\Omega)$ with $\int_\Omega h=0$ there is a
unique solution $\phi:=T(h)$, with $\int_\Omega \phi=0$, and
$d_0,d_1,d_2 \in \mathbb{R}$ of problem \eqref{plco}. Moreover,
there is a constant $C>0$ so that
\begin{equation}\label{est1}
\|\phi \|_\infty \le C\lf(\log \frac 1\de \rg)\|h\|_*,\quad
\sum_{l=0}^2 |d_{l}|\le C\|h\|_*.
\end{equation}
\end{prop}

\begin{proof} Since $-\Delta Z_l=|\sigma'(z)|^{2} e^{U_{\de,a}} Z_l$ in $\Om$ (where $U_{\de,a}$ stands for $U_{\de,a,\sigma_a}$) and $\int_\Omega \Delta Z_l=O(\de^2)$ in view of \eqref{deltaZ0}-\eqref{deltaZ}, we have that $\Delta PZ_l =O(|\sigma'(z)|^{2} e^{U_{\de,a}})+O(\de^2)$ in view of $Z_l=O(1)$, yielding to $\|\lap PZ_{l}\|_*\le C$ for all $l=0,1,2$. By Proposition \ref{p1} every solution of \equ{plco} satisfies
$$\|\phi\|_\infty\le
C\lf(\log{1\over\de}\rg)\lf[\|h\|_*+\sum_{l=0}^2|d_{l}|\rg].$$ Set
$\langle f,g\rangle=\int_\Om fg$ and notice that
\begin{equation}\label{diesis}
\ds\langle L(\phi),PZ_{j}\rangle = \ds\langle
L(\phi),PZ_{j}+t\rangle =\langle \phi+\gamma(\phi),\ti
L(PZ_{j}+t)\rangle
\end{equation}
for any $t\in \mathbb{R}$, in view of $\int_\Omega L(\phi)=0$. To
estimate the $|d_{l}|$'s, let us test equation \equ{plco} against
$PZ_{j}$, $j=0,1,2$, to get
$$\big\langle \phi+\gamma(\phi),\ti L(PZ_{j}+t_j)\big\rangle =\langle
h,PZ_{j}\rangle + \sum_{l=0}^2d_{l}\langle \lap
PZ_{l},PZ_{j}\rangle$$
where $t_j=\frac{1}{|\Omega|}\int_\Omega Z_j$, $j=0,1,2$. From the
proof of Lemma \ref{1039} we know that for $Z_0$ and $Z=Z_1+iZ_2$
there hold the following:
\begin{eqnarray*}
&& \int_\Om \lap PZ_0PZ_0=- 16 (n+1) \int_{\mathbb{R}^2} \frac{1-|y|^2}{(1+|y|^2)^4} +O(\de^2)\,,\qquad \int_\Om \lap PZPZ_0=O(\de^2)\\
&&\int_\Om \lap PZ \overline{PZ}=- 8 (n+1) \int_{\mathbb{R}^2}
\frac{|y|^2 }{(1+|y|^2)^4} +O(\de)\,, \qquad \int_\Om \lap PZ
PZ=O(\de)
\end{eqnarray*}
where $ \int_{\mathbb{R}^2} \frac{dy}{(1+|y|^2)^4}=2
\int_{\mathbb{R}^2} \frac{1-|y|^2}{(1+|y|^2)^4}=\frac{\pi}{3}$. In
terms of the $Z_l$'s we then have that
$$\langle \lap PZ_{l},PZ_{j}\rangle=-(n+1)C_{ij} \de_{lj}+O(\de^2),$$
where $\de_{lj}$ denotes the Kronecker's symbol and $c_{00}={8\pi\over 3}$, $c_{11}=c_{22}={4\pi\over 3}$. For $j=0,1,2$ let
us now estimate $\big\|\ti L(PZ_j+t_j)\big\|_*$:
\begin{equation}\label{diesisdiesis}
\big\|\ti L(PZ_j+t_j)\big\|_*=\big\|-|\sigma'(z)|^2 e^{U_{\de,a}}
Z_j+\ml{K}(PZ_j+t_j)+O(\delta^2) \big\|_*=O( \de+\epsilon^2
\delta^{-\frac{2}{n+1}}+\de|c_a|)
\end{equation}
in view of \eqref{deltaZ0}-(\ref{pzij}) and (\ref{mlK}). Since
$|\gamma(\phi)|=O(\|\phi\|_\infty)$ in view of (\ref{BW}) and
$\epsilon^2 \delta^{-\frac{2}{n+1}}=o(1)$, by \eqref{1458} we get that
$$\langle \phi+\gamma(\phi),\ti L(PZ_{j}+t_j)\rangle=O(\de+\epsilon^2 \delta^{-\frac{2}{n+1}}) \|\phi\|_\infty,$$
which along the previous estimates yields to
\begin{equation}\label{estcij}
\begin{split}
|d_j|\le
C\bigg[(\de+\e^2\de^{-{2\over n+1}})\|\phi\|_\infty+\|h\|_*+\delta \sum_{l=0}^2|d_l| \bigg]
\end{split}
\end{equation}
in view of $PZ_j=O(1)$. Since (\ref{estcij}) gives that
$\displaystyle \sum_{l=0}^2|d_{l}|=O(\de+\e^2\de^{-{2\over n+1}})\|\phi\|_\infty+O(\|h\|_*)$, we have that every
solution of \equ{plco} satisfies
$$\|\phi\|_\infty\le
C\lf(\log{1\over\de}\rg)\lf[\|h\|_*+\sum_{l=0}^2|d_{l}|\rg] \leq
C\log{1\over\de}(\de+\e^2\de^{-{2\over n+1}})\|\phi\|_\infty+C \log{1\over\de} \|h\|_*.$$ In view of
$\log{1\over\de} (\de+\e^2\de^{-{2\over n+1}})=o(1)$ as $\eta_0\to 0$, the a-priori estimates \eqref{est1} immediately follow.

\medskip \noindent To solve \equ{plco}, consider now the space
$$H=\lf\{\phi\in H^1(\Omega) \hbox{ doubly-periodic}: \: \int_\Omega \phi=0\,,\:\int_\Om\lap PZ_{l}\,\phi=0 \hbox{ for }l=0,1,2\rg\}$$
endowed with the usual inner product
$[\phi,\psi]=\int_\Om\grad\phi\grad\psi.$ Problem \eqref{plco} is
equivalent to finding $\phi\in H$ \st
$$[\phi,\psi]=\int_\Om\lf[\ml{K} \lf(\phi+\gamma(\phi)\rg)-h\rg]\psi\qquad\text{for all }\psi\in H.$$ With the aid
of Riesz's representation theorem, the equation has the form
$(\hbox{Id}-\hbox{compact operator})\phi= \ti h$. Fredholm's
alternative guarantees unique solvability of this problem for any
$h$ provided that the homogeneous equation has only the trivial
solution. This is equivalent to \eqref{plco} with $h\equiv 0$,
which has only the trivial solution by the a-priori estimates
\eqref{est1}. The proof is now complete.\qed
\end{proof}

%\bigskip The result of latter proposition implies that the unique
%solution $\phi=T(h)$ of \eqref{plco} defines a continuous linear
%map from the Banach space $\ml{C}_*$ of all functions $h$ in
%$L^\infty$ for which $\|h\|_*<+\infty$, into $L^\infty$ (with
%$\int_\Om h=\int_\Om \phi=0$).

\section{\hspace{-0.5cm}: The nonlinear problem}
We consider the following non linear problem
\begin{equation}\label{pnla}
\left\{\begin{array}{ll}
L(\phi)= -[R+N(\phi)] +\displaystyle \sum_{l=0}^{2}d_l \Delta PZ_{l} & \text{in }\Om\\
\int_{\Omega } \Delta PZ_l\phi = 0 \hbox{ for all }l=0,1,2 & \\
\int_{\Om}\phi=0,&
\end{array} \right.
\end{equation}
where $R$, $N(\phi)$ and $L$ are given by \eqref{R}, \eqref{nlt}
and \eqref{ol}, respectively. Notice that \eqref{linear} and
(\ref{pnla}) are equivalent by setting $d=d_1-id_2$.

\begin{lem}\label{lpnla}
There exists $\de_0>0$ small such that for any $0<\de<\eta_0$,
$|\log \delta|^2 \e^2\leq \eta_0 \de^{2\over n+1}$, $|a|\leq \eta_0 \de$ problem
\eqref{pnla} admits a unique solution $\phi$ and $d_l$, $l=0,1,2$.
Moreover, there exists $C>0$ so that
\begin{equation}\label{cotapsi}
\|\phi\|_\infty\le C|\log\de|\|R\|_*.
\end{equation}
\end{lem}
\begin{proof}
In terms of the operator $T$ defined in Proposition \ref{p2},
problem \eqref{pnla} reads as
$$\phi=-T\lf(R+N(\phi)\rg):=\ml{A}(\phi).$$
For a given number $M>0$, let us consider the space
$$
\ml{F}_M = \{\phi\in L^\infty(\Om) \hbox{ doubly-periodic }:\: \|
\phi \|_\infty \le M|\log\de| \,\|R\|_* \}.$$ It is a
straightforward but tedious computation to show that
\begin{equation}\label{star}
\|N(\phi_1) - N(\phi_2)\|_* \leq C_1 (\|\phi_1\|_\infty
+\|\phi_2\|_\infty) \|\phi_1-\phi_2\|_\infty.
\end{equation}
Just to give an idea on how (\ref{star}) can be proved, observe
that $0\leq \frac{e^{u_0+W+\phi}}{\int_\Omega e^{u_0+W+\phi}} \leq
e^{2\|\phi\|_\infty} \frac{e^{u_0+W}}{\int_\Omega e^{u_0+W}}$ and
$|\int_\Omega e^{u_0+W+\phi} \phi|\leq \|\phi\|_\infty \int_\Omega
e^{u_0+W+\phi}$. For $\|\phi\|_\infty\leq 1$ we can then get that
$$\|\phi\|_\infty \|D [\frac{e^{u_0+W+\phi}}{\int_\Omega e^{u_0+W+\phi}}][\phi]\|_*+\|D^2 [\frac{e^{u_0+W+\phi}}{\int_\Omega e^{u_0+W+\phi}}][\phi,\phi]\|_*=O(\|\frac{e^{u_0+W}}{\int_\Omega e^{u_0+W}}\|_*  \|\phi\|_\infty^2)
=O(\|\phi\|_\infty^2)$$ in view of $\|\frac{e^{u_0+W}}{\int_\Omega
e^{u_0+W}}\|_*=O(1)$ by (\ref{eps1}). This exactly what
we need to estimate in $\|\cdot\|_*-$norm the difference between
the first term of $N(\phi_1)$ and $N(\phi_2)$. For the other terms
we can argue in a similar way to get
$$\|\phi\|_\infty \|D [\frac{e^{2(u_0+W+\phi)}}{\int_\Omega e^{2(u_0+W+\phi)}}][\phi]\|_*+\|D^2 [\frac{e^{2(u_0+W+\phi)}}{\int_\Omega e^{2(u_0+W+\phi)}}][\phi,\phi]\|_*=O(\| \frac{e^{2(u_0+W)}}{\int_\Omega e^{2(u_0+W)}}\|_* \|\phi\|_\infty^2)=O(\|\phi\|_\infty^2)$$
in view of $\| \frac{e^{2(u_0+W)}}{\int_\Omega e^{2(u_0+W)}}\|_* =O(1)$ by
(\ref{eps2}), and
$$\|\phi\|_\infty \|D[B(W+\phi)][\phi]\|_*+\|D^2[B(W+\phi)][\phi,\phi]\|_*=O(B(W)\|\phi\|_\infty^2)=O(\delta^{-\frac{2}{n+1}}\|\phi\|_\infty^2)$$
in view of (\ref{BW}). Since $\epsilon^2
\delta^{-\frac{2}{n+1}}=o(1)$ we can deduce the validity of
(\ref{star}).

\medskip \noindent Denote by $C'$ the constant present in \eqref{est1}. By Proposition \ref{p2} and (\ref{star}) we get that
$$\|\ml{A}(\phi_1)-\ml{A}(\phi_2)\|_\infty \leq C'|\log \de| \|N(\phi_1)-N(\phi_2)\|_*\leq 2C'C_1 M \|R\|_ * \log^2 \de \|\phi_1-\phi_2\|_\infty$$ % \leq \frac{1}{2}\|\phi_1-\phi_2\|_\infty$$
for all $\phi_1,\phi_2 \in \ml{F}_M$. By Proposition \ref{p2} we also have that
\begin{equation*}
\|\ml{A}(\phi)\|_\infty \le C' | \log \de |\lf[ \|R\|_* +
\|N(\phi)\|_*\rg]\leq   C' | \log \de | \|R\|_*+C' C_1|\log
\de| \|\phi\|_\infty^2
\end{equation*}
for all $\phi\in \ml{F}_M$. Fix now $M$ as $M=2C'$, and by (\ref{ere}) take
$\eta_0$ small so that $4(C')^2  C_1 \log^2 \de \|R\|_*
< \frac{1}{2}$ in order to have that $\ml{A}$ is a contraction
mapping of $\ml{F}_M$ into itself. Therefore $\ml{A}$ has a unique
fixed point $\phi$ in $\ml{F}_M$, which satisfies (\ref{cotapsi})
with $C=M$.\qed
\end{proof}

\section{\hspace{-0.5cm}: The integral coefficients in \eqref{solve1b}-\eqref{solve2b}}
Letting $\zeta=\frac{a}{\delta}$, we aim to investigate the integral coefficients
$$I:=\int_{\mathbb{R}^2} \frac{(|y|^2-1)|y+\zeta|^{\frac{2n}{n+1}}}{(1+|y|^2)^5}\,dy\:,\qquad K:=\int_{\mathbb{R}^2} \frac{|y+\zeta|^{\frac{2n}{n+1}}y}{(1+|y|^2)^5}\,dy$$
which appear in \eqref{solve1b}-\eqref{solve2b} or \eqref{solve1}-\eqref{solve2}. We will show below that $I=f(|\zeta|)$ and $K=g(|\zeta|)\zeta$ with $f<0<g$, and the asymptotic behavior of $f$ and $g$ as $|\zeta|\to +\infty$ will be identified.

\medskip \noindent By the change of variable $y \to y+\zeta$ and the Taylor expansion
$$(1-x)^{-5}=\sum_{k=0}^{+\infty} c_k x^k \quad\hbox{for }|x|<1$$
with $c_k=\frac{(4+k)!}{24\,k!}$, we can re-write $I$ as
\begin{eqnarray*}
I&=&\int_{\mathbb{R}^2} \frac{|y|^{\frac{2n}{n+1}}(|y-\zeta|^2-1)}{(1+|y-\zeta|^2)^5}dy
=\sum_{k=0}^{+\infty} c_k \int_{\mathbb{R}^2} \frac{|y|^{\frac{2n}{n+1}}(|y|^2+|\zeta|^2-1-y\bar \zeta-\bar y \zeta)(y\bar \zeta+\bar y \zeta)^k}{(1+|y|^2+|\zeta|^2)^{5+k}}dy
\end{eqnarray*}
in view of 
$$(1+|y-\zeta|^2)^{-5}=(1+|y|^2+|\zeta|^2)^{-5}(1-\frac{y\bar \zeta+\bar y \zeta}{1+|y|^2+|\zeta|^2})^{-5}$$ 
with 
$$\frac{|y\bar \zeta+\bar y \zeta|}{1+|y|^2+|\zeta|^2} \leq \frac{|y|^2+|\zeta|^2}{1+|y|^2+|\zeta|^2}<1.$$
Since 
$$(y\bar \zeta+\bar y \zeta)^k= \sum_{j=0}^k \left(\begin{array}{l} k \\ j \end{array}\right)   y^j \bar \zeta^j \bar y^{k-j} \zeta^{k-j}=
\sum_{1\leq j <\frac{k}{2}  }  \left(\begin{array}{l} k \\ j \end{array}\right) \zeta^{k-2j} \bar y^{k-2j} |\zeta|^{2j} |y|^{2j}  
+\sum_{\frac{k}{2}<j\leq k}  \left(\begin{array}{l} k \\ j \end{array}\right) \bar \zeta^{2j-k} y^{2j-k} |\zeta|^{2k-2j} |y|^{2k-2j}  
$$
for $k$ odd and
$$(y\bar \zeta+\bar y \zeta)^k=
\sum_{1\leq j <\frac{k}{2}  } \left(\begin{array}{l} k \\ j \end{array}\right) \zeta^{k-2j} \bar y^{k-2j} |\zeta|^{2j} |y|^{2j}  
+\sum_{\frac{k}{2}<j\leq k}  \left(\begin{array}{l} k \\ j \end{array}\right) \bar \zeta^{2j-k} y^{2j-k} |\zeta|^{2k-2j} |y|^{2k-2j}  
+ \left(\begin{array}{l} k \\  \frac{k}{2} \end{array}\right) |\zeta|^k |y|^k
$$
for $k$ even, by symmetry we can simplify the expression of $I$ as follows:
\begin{eqnarray*}
I&=&
\sum_{k=0}^{+\infty} c_k \int_{\mathbb{R}^2} \frac{|y|^{\frac{2n}{n+1}}(|y|^2+|\zeta|^2-1)(y\bar \zeta+\bar y \zeta)^k}{(1+|y|^2+|\zeta|^2)^{5+k}}dy
-\sum_{k=0}^{+\infty} c_k \int_{\mathbb{R}^2} \frac{|y|^{\frac{2n}{n+1}}(y\bar \zeta+\bar y \zeta)^{k+1}}{(1+|y|^2+|\zeta|^2)^{5+k}}dy\\
&=&\sum_{k=0}^{+\infty} c_{2k} \left(\begin{array}{c} 2k \\  k \end{array}\right) |\zeta|^{2k}   \int_{\mathbb{R}^2} \frac{|y|^{\frac{2n}{n+1}+2k} (|y|^2+|\zeta|^2-1)}{(1+|y|^2+|\zeta|^2)^{5+2k}}  dy
-\sum_{k=1}^{+\infty}  c_{2k-1} \left(\begin{array}{c} 2k \\  k \end{array}\right) |\zeta|^{2k} \int_{\mathbb{R}^2} \frac{|y|^{\frac{2n}{n+1}+2k} }{(1+|y|^2+|\zeta|^2)^{4+2k}} dy
\end{eqnarray*}
Since $I^p_q=\displaystyle \int_0^{\infty} \frac{ \rho^p}{(1+\rho)^q}d\rho$, $q>p+1$, does satisfy the relations:
\begin{equation} \label{Ipq}
I^p_{q+1}=\frac{q-p-1}{q}I^p_q\:,\qquad I^{p+1}_q=\frac{p+1}{q-p-2} I^p_q,
\end{equation}
through the change of variable $\rho^2=\lambda t$, $\lambda=1+|\zeta|^2$, in polar coordinates we have that
\begin{eqnarray} \label{1748}
\int_{\mathbb{R}^2} \frac{|y|^{\frac{2n}{n+1}+2k} }{(1+|y|^2+|\zeta|^2)^{5+2k}}  dy
&=& 
\pi \lambda^{\frac{n}{n+1}-4-k} I^{\frac{n}{n+1}+k}_{5+2k}
=\pi \frac{3+k-\frac{n}{n+1}}{4+2k}\lambda^{\frac{n}{n+1}-4-k} I^{\frac{n}{n+1}+k}_{4+2k}\nonumber\\
&=&\frac{3+k-\frac{n}{n+1}}{2(2+k)(1+|\zeta|^2)}
\int_{\mathbb{R}^2} \frac{|y|^{\frac{2n}{n+1}+2k} }{(1+|y|^2+|\zeta|^2)^{4+2k}}  dy
\end{eqnarray}
and
\begin{eqnarray} \label{1818}
\int_{\mathbb{R}^2} \frac{|y|^{\frac{2n}{n+1}-2+2k} }{(1+|y|^2+|\zeta|^2)^{2+2k}}  dy
&=& 
\pi \lambda^{\frac{n}{n+1}-2-k} I^{\frac{n}{n+1}-1+k}_{2+2k}
=\pi \frac{(2+2k)(3+2k)}{(k+\frac{n}{n+1})(2+k-\frac{n}{n+1})}\lambda^{\frac{n}{n+1}-2-k} I^{\frac{n}{n+1}+k}_{4+2k}\nonumber\\
&=&\frac{(2+2k)(3+2k)}{(k+\frac{n}{n+1})(2+k-\frac{n}{n+1})}(1+|\zeta|^2)
\int_{\mathbb{R}^2} \frac{|y|^{\frac{2n}{n+1}+2k} }{(1+|y|^2+|\zeta|^2)^{4+2k}}  dy
\end{eqnarray}
Inserting \eqref{1748} and \eqref{1818} into $I$, we get that
\begin{eqnarray*}
I&=&
\sum_{k=0}^{+\infty} c_{2k} \left(1-\frac{3+k-\frac{n}{n+1}}{(2+k)(1+|\zeta|^2)}\right)\left(\begin{array}{c} 2k \\  k \end{array}\right) |\zeta|^{2k}   \int_{\mathbb{R}^2} \frac{|y|^{\frac{2n}{n+1}+2k} }{(1+|y|^2+|\zeta|^2)^{4+2k}}  dy\\
&&-\sum_{k=1}^{+\infty}  c_{2k-1} \left(\begin{array}{c} 2k \\  k \end{array}\right) |\zeta|^{2k} \int_{\mathbb{R}^2} \frac{|y|^{\frac{2n}{n+1}+2k} }{(1+|y|^2+|\zeta|^2)^{4+2k}} dy\\
&=&
\sum_{k=1}^{+\infty} \left[\frac{2(3+2k)c_{2k-2}}{k+\frac{n}{n+1}} \left(\frac{1+k}{2+k-\frac{n}{n+1}}-\frac{1}{1+|\zeta|^2}\right)\left(\begin{array}{c} 2k-2 \\  k-1 \end{array}\right)    (1+|\zeta|^2)-c_{2k-1} \left(\begin{array}{c} 2k \\  k \end{array}\right) |\zeta|^2 \right]\times\\
&&\times |\zeta|^{2k-2} \int_{\mathbb{R}^2} \frac{|y|^{\frac{2n}{n+1}+2k} }{(1+|y|^2+|\zeta|^2)^{4+2k}} dy.
\end{eqnarray*}
Since $2(3+2k)c_{2k-2} \left(\begin{array}{c} 2k-2 \\  k-1 \end{array}\right)=k
c_{2k-1} \left(\begin{array}{c} 2k \\  k \end{array}\right)$ for all $k \geq 1$, setting $\beta_k=c_{2k-1} \left(\begin{array}{c} 2k \\  k \end{array}\right) |\zeta|^{2k-2} \int_{\mathbb{R}^2} \frac{|y|^{\frac{2n}{n+1}+2k} }{(1+|y|^2+|\zeta|^2)^{4+2k}} dy$ we deduce that
\begin{eqnarray*}
I&=&
\sum_{k=1}^{+\infty} \left[\frac{k}{k+\frac{n}{n+1}} \left(\frac{1+k}{2+k-\frac{n}{n+1}}-\frac{1}{1+|\zeta|^2}\right) (1+|\zeta|^2)- |\zeta|^2 \right] \beta_k\\
&=&
\sum_{k=1}^{+\infty} \left[\frac{k}{k+\frac{n}{n+1}} \left(\frac{|\zeta|^2}{1+|\zeta|^2}-\frac{1}{(2+k)(n+1)-n} \right) (1+|\zeta|^2)- |\zeta|^2 \right] \beta_k<\sum_{k=1}^{+\infty} \left[\frac{k}{k+\frac{n}{n+1}}-1\right]  |\zeta|^2 \beta_k<0.
\end{eqnarray*}
In conclusion, we have shown that $I=f(|\zeta|)$ with $f<0$.

\medskip \noindent By the change of variable $y \to y+\zeta$ and the Taylor expansion of $(1-x)^{-5}$, arguing as before $K$ can be re-written as
\begin{eqnarray*}
K&=&\int_{\mathbb{R}^2} \frac{|y|^{\frac{2n}{n+1}}(y-\zeta)}{(1+|y-\zeta|^2)^5}dy
=\sum_{k=0}^{+\infty} c_k \int_{\mathbb{R}^2} \frac{|y|^{\frac{2n}{n+1}}(y-\zeta)(y\bar \zeta+\bar y \zeta)^k}{(1+|y|^2+|\zeta|^2)^{5+k}}dy.
\end{eqnarray*}
By the previous expansions of $(y\bar \zeta+\bar y \zeta)^k$ and
\begin{eqnarray*}
\int_{\mathbb{R}^2} \frac{|y|^{\frac{2n}{n+1}+2+2k}}{(1+|y|^2+|\zeta|^2)^{6+2k}}dy
&=&\pi \lambda^{\frac{n}{n+1}-4-k} I^{\frac{n}{n+1}+1+k}_{6+2k}=  
\pi \frac{\frac{n}{n+1}+1+k}{5+2k} \lambda^{\frac{n}{n+1}-4-k} I^{\frac{n}{n+1}+k}_{5+2k}\\
&=&
\frac{\frac{n}{n+1}+1+k}{5+2k}\int_{\mathbb{R}^2} \frac{|y|^{\frac{2n}{n+1}+2k}}{(1+|y|^2+|\zeta|^2)^{5+2k}}dy,
\end{eqnarray*}
for symmetry $K$ reduces to
\begin{eqnarray*}
K&=&
\zeta \, \sum_{k=0}^{+\infty} \left[c_{2k+1} \frac{\frac{n}{n+1}+1+k}{5+2k}\left(\begin{array}{c} 2k+1 \\  k \end{array}\right) 
-c_{2k} \left(\begin{array}{c} 2k \\  k \end{array}\right)
\right]
|\zeta|^{2k} \int_{\mathbb{R}^2} \frac{|y|^{\frac{2n}{n+1}+2k}}{(1+|y|^2+|\zeta|^2)^{5+2k}}dy.
\end{eqnarray*}
Since $(1+k) c_{2k+1} \left(\begin{array}{c} 2k+1 \\  k \end{array}\right)=(5+2k) c_{2k} \left(\begin{array}{c} 2k \\  k \end{array}\right)$ for all $k \geq 0$, we get that
\begin{eqnarray*}
K&=&
\zeta \, \sum_{k=0}^{+\infty} \frac{n}{(n+1)(1+k)} c_{2k} \left(\begin{array}{c} 2k \\  k \end{array}\right)
|\zeta|^{2k} \int_{\mathbb{R}^2} \frac{|y|^{\frac{2n}{n+1}+2k}}{(1+|y|^2+|\zeta|^2)^{5+2k}}dy.
\end{eqnarray*}
In conclusion, we have shown that $K=g(|\zeta|)\zeta$ with $g>0$.

\medskip \noindent In order to determine the asymptotic behavior of $f$ and $g$ as $|\zeta|\to +\infty$, we will use complex analysis to get some integral representation of $f$ and $g$, see \eqref{exprI-2J} and \eqref{exprK}. We split $I$ as $I=J_1-2J_2$, and we compute separately the constants
$$J_1=\int_{\mathbb{R}^2} \frac{|y+\zeta|^{\frac{2n}{n+1}}}{(1+|y|^2)^4}dy\:,\quad J_2=\int_{\mathbb{R}^2} \frac{|y+\zeta|^{\frac{2n}{n+1}}}{(1+|y|^2)^5}dy.$$
Concerning $J_1$, we re-write it in polar coordinates as
\begin{eqnarray*}
J_1&=&\int_{\mathbb{R}^2} \frac{|y|^{\frac{2n}{n+1}}}{(1+|y-\zeta|^2)^4}dy=\int_0^{+\infty} \rho^{\frac{2n}{n+1}+1} d\rho \int_0^{2\pi} \frac{d\theta}{(1+\rho^2+|\zeta|^2-\zeta\rho e^{-i\theta}-\overline{\zeta}\rho e^{i\theta})^4}\\
&=&- i \int_0^{+\infty} \rho^{\frac{2n}{n+1}+1} d\rho \int_{\partial^+ B_1(0)} \frac{w^3}{(\overline{\zeta}\rho)^4 (w^2-\frac{1+\rho^2+|\zeta|^2}{\overline{\zeta}\rho}w+\frac{\zeta^2}{|\zeta|^2})^4}dw.
\end{eqnarray*}
Since $w^2- \displaystyle \frac{1+\rho^2+|\zeta|^2}{\overline{\zeta}\rho}w+\frac{\zeta^2}{|\zeta|^2}$ vanishes only at
$$w_\pm=\frac{1+\rho^2+|\zeta|^2\pm \sqrt{(1+\rho^2+|\zeta|^2)^2-4\rho^2|\zeta|^2}}{2 \overline{\zeta} \rho}$$
with $|w_-|<1<|w_+|$, by the Residue Theorem we have that
\begin{eqnarray*}
J_1= - i \int_0^{+\infty} \rho^{\frac{2n}{n+1}+1} d\rho \int_{\partial^+ B_1(0)} \frac{w^3}{(\overline{\zeta}\rho)^4 (w-w_-)^4(w-w_+)^4}dw=2\pi
\int_0^{\infty} \frac{ \rho^{\frac{2n}{n+1}+1} }{6 (\overline{\zeta}\rho)^4}  \frac{d^3}{d w^3} \left[ \frac{w^3}{(w-w_+)^4}\right](w_-) d \rho .
\end{eqnarray*}
A straightforward computation shows that
$$ \frac{d^3}{d w^3}\left[ \frac{w^3}{(w-w_+)^4}\right]=-6\frac{w^3+w_+^3+9w w_+(w+w_+)}{(w-w_+)^7},$$
and then
$$ \frac{d^3}{d w^3}\left[ \frac{w^3}{(w-w_+)^4}\right](w_-)=6 (\overline{\zeta}\rho)^4 \frac{(1+\rho^2+|\zeta|^2)[(1+\rho^2+|\zeta|^2)^2+6\rho^2 |\zeta|^2]}{[(1+\rho^2+|\zeta|^2)^2-4\rho^2 |\zeta|^2]^{\frac{7}{2}}}.$$
Recalling that $\lambda=1+|\zeta|^2$, through the change of variable $\rho \to \rho^2$ we finally get for $J_1$ the expression
\begin{eqnarray}\label{exprI}
J_1=\pi
\int_0^{\infty} \rho^{\frac{n}{n+1}} \frac{(\lambda+\rho)[(\lambda+\rho)^2+6(\lambda-1)\rho]}{[(\lambda+\rho)^2-4(\lambda-1)\rho]^{\frac{7}{2}}}
d \rho.
\end{eqnarray}

\medskip \noindent In a similar way, we first re-write $J_2$ as
\begin{eqnarray*}
J_2=  i \int_0^{+\infty} \rho^{\frac{2n}{n+1}+1} d\rho \int_{\partial^+ B_1(0)} \frac{w^4}{(\overline{\zeta}\rho)^5 (w-w_-)^5 (w-w_+)^5}dw=
-2\pi \int_0^{+\infty} \frac{ \rho^{\frac{2n}{n+1}+1} }{24 (\overline{\zeta}\rho)^5}  \frac{d^4}{d w^4} \left[ \frac{w^4}{(w-w_+)^5}\right](w_-) d\rho
\end{eqnarray*}
in view of the Residue Theorem. Since
$$ \frac{d^4}{d w^4}\left[ \frac{w^4}{(w-w_+)^5}\right]=24\frac{w^4+w_+^4+16w w_+(w^2+w_+^2)+36 w^2 w_+^2}{(w-w_+)^9},$$
we get that
$$ \frac{d^4}{d w^4}\left[ \frac{w^4}{(w-w_+)^5}\right](w_-)=-24 (\overline{\zeta}\rho)^5 \frac{(1+\rho^2+|\zeta|^2)^4 +12 \rho^2|\zeta|^2 (1+\rho^2+|\zeta|^2)^2+42\rho^4 |\zeta|^4}{[(1+\rho^2+|\zeta|^2)^2-4\rho^2 |\zeta|^2]^{\frac{9}{2}}},$$
and then
\begin{eqnarray}\label{exprJ}
J_2=\pi
\int_0^{\infty} \rho^{\frac{n}{n+1}} \frac{(\lambda+\rho)^4+12(\lambda-1)\rho (\lambda+\rho)^2+42(\lambda-1)^2\rho^2}{[(\lambda+\rho)^2-4(\lambda-1)\rho]^{\frac{9}{2}}}
d \rho.
\end{eqnarray}

\medskip \noindent By \eqref{exprI}-\eqref{exprJ} we finally get that $f(|\zeta|)$  takes the form
\begin{eqnarray} \label{exprI-2J}
f=\pi \int_0^{\infty} \hspace{-0,3cm} \rho^{\frac{n}{n+1}} \frac{(\lambda+\rho)^5-2(\lambda+\rho)^4+2(\lambda-1) \rho (\lambda+\rho)^3 -24 \lambda (\lambda-1) \rho(\rho+1) (\lambda+\rho) -84(\lambda-1)^2 \rho^2 }{[(\lambda+\rho)^2-4(\lambda-1)\rho]^{\frac{9}{2}}}
d \rho
\end{eqnarray}
where $\lambda=1+|\zeta|^2$.

\medskip \noindent Observe that for $\zeta=0$ (i.e. $\lambda=1$) we simply have that
\begin{equation} \label{1228}
f(0)=J_1-2J_2=\pi [I^{\frac{n}{n+1}}_4-2 I^{\frac{n}{n+1}}_5]=-\frac{2\pi}{2n+3}I^{\frac{n}{n+1}}_5
\end{equation}
in view of \eqref{Ipq}. By the change of variable $\rho=\lambda+\sqrt \lambda t$ and the Lebesgue Theorem we get that
\begin{eqnarray*}
\lambda^{-\frac{n}{n+1}} J_1=\pi \int_{-\sqrt \lambda}^\infty (1+\frac{t}{\sqrt \lambda})^{\frac{n}{n+1}} \frac{(2+\frac{t}{\sqrt \lambda})^3+6\frac{\lambda-1}{\lambda}
(1+\frac{t}{\sqrt \lambda})(2+\frac{t}{\sqrt \lambda})}{(t^2+4+\frac{4t}{\sqrt \lambda})^{\frac{7}{2}}} dt \to 20 \pi \int_{\mathbb{R}} \frac{dt}{(t^2+4)^{\frac{7}{2}}}
\end{eqnarray*}
and
\begin{eqnarray*}
\lambda^{-\frac{n}{n+1}} J_2&=&\pi \int_{-\sqrt \lambda}^\infty (1+\frac{t}{\sqrt \lambda})^{\frac{n}{n+1}} \frac{(2+\frac{t}{\sqrt \lambda})^4+12 \frac{\lambda-1}{\lambda}
(1+\frac{t}{\sqrt \lambda})(2+\frac{t}{\sqrt \lambda})^2 +42 (\frac{\lambda-1}{\lambda})^2
(1+\frac{t}{\sqrt \lambda})^2 }{(t^2+4+\frac{4t}{\sqrt \lambda})^{\frac{9}{2}}} dt\\
& \to & 106 \pi \int_{\mathbb{R}} \frac{dt}{(t^2+4)^{\frac{9}{2}}}
\end{eqnarray*}
as $|\zeta|\to +\infty$ (i.e. $\lambda \to +\infty$). Since $\displaystyle \int_{\mathbb{R}} \frac{dt}{(t^2+4)^{\frac{7}{2}}}=\frac{14}{3} \displaystyle  \int_{\mathbb{R}} \frac{dt}{(t^2+4)^{\frac{9}{2}}},$ we get that
\begin{equation} \label{1902}
\frac{f(|\zeta|)}{|\zeta|^{\frac{2n}{n+1}}} \to -\frac{356}{3} \pi \int_{\mathbb{R}} \frac{dt}{(t^2+4)^{\frac{9}{2}}}
\end{equation}
as $|\zeta|\to \infty$. 

\medskip \noindent In a similar way, for $K$ we have that
\begin{eqnarray*}
K=  i \int_0^{+\infty} \rho^{\frac{2n}{n+1}+1} d\rho \int_{\partial^+ B_1(0)} \frac{w^4(\rho w-\zeta)}{(\overline{\zeta}\rho)^5 (w-w_-)^5 (w-w_+)^5}dw=
-2\pi \int_0^{+\infty}  \frac{\rho^{\frac{2n}{n+1}+1} }{24 (\overline{\zeta}\rho)^5}  \frac{d^4}{d w^4} \left[ \frac{w^4(\rho w-\zeta)}{(w-w_+)^5}\right](w_-) d\rho
\end{eqnarray*}
in view of the Residue Theorem. Since
$$ \frac{d^4}{d w^4}\left[ \frac{w^4(\rho w-\zeta)}{(w-w_+)^5}\right]=24\frac{5\rho w w_+[w^3+w_+^3+6ww_+(w+w_+)]-\zeta [w^4+w_+^4+16w w_+(w^2+w_+^2)+36 w^2 w_+^2]}{(w-w_+)^9},$$
we get that
$$ \frac{d^4}{d w^4}\left[ \frac{w^4(\rho w- \zeta)}{(w-w_+)^5}\right](w_-)=12 (\overline{\zeta}\rho)^5 \zeta \frac{(\lambda+\rho^2)^4+2\rho^2 (\lambda-6-5\rho^2) (\lambda+\rho^2)^2+6(\lambda-1)\rho^4  (2\lambda-7-5\rho^2)}{[(\lambda+\rho^2)^2-4(\lambda-1)\rho^2 ]^{\frac{9}{2}}},$$
and then
\begin{eqnarray}\label{exprK}
g(|\zeta|)=-\frac{\pi}{2}
\int_0^{\infty} \rho^{\frac{n}{n+1}} \frac{(\lambda+\rho)^4+2\rho (\lambda-6-5\rho) (\lambda+\rho)^2+6(\lambda-1)\rho^2  (2\lambda-7-5\rho)}{[(\lambda+\rho)^2-4(\lambda-1)\rho]^{\frac{9}{2}}}
d \rho.
\end{eqnarray}
So, we have that
\begin{equation} \label{1903}
g(0)=\frac{\pi}{2}(9I_5^{\frac{n}{n+1}}-10 I_6^{\frac{n}{n+1}})=\frac{3n+1}{2(n+1)} \pi I_5^{\frac{n}{n+1}}
\end{equation}
in view of \eqref{Ipq}, and, by the change of variable $\rho=\lambda+\sqrt \lambda t$ and the Lebesgue Theorem,
\begin{eqnarray} \label{1904}
\frac{g(|\zeta|)}{|\zeta|^{\frac{2n}{n+1}}} \to 17 \pi \int_{\mathbb{R}} \frac{dt}{(t^2+4)^{\frac{9}{2}}}
\end{eqnarray}
as $|\zeta|\to +\infty$, in view of 
$$\int_{-\sqrt \lambda}^\infty (1+\frac{t}{\sqrt \lambda})^{\frac{n}{n+1}} \frac{
(2+\frac{t}{\sqrt \lambda})^4-2(1+\frac{t}{\sqrt \lambda}) (4+\frac{6+5 \sqrt \lambda t}{\lambda}) (2+\frac{t}{\sqrt \lambda})^2-6\frac{\lambda-1}{\lambda} (1+\frac{t}{\sqrt \lambda})^2  (3+\frac{7+5 \sqrt \lambda t}{\lambda})}{(t^2+4+\frac{4t}{\sqrt \lambda})^{\frac{9}{2}}} dt\to -   \int_{\mathbb{R}} \frac{34\, dt}{(t^2+4)^{\frac{9}{2}}}$$
as $\lambda \to +\infty$.

\medskip \noindent {\bf Acknowledgements:} The work for this paper began while the second author was visiting the Departamento de Matem\'atica, Pontificia Universidad Cat\'olica de Chile (Santiago, Chile). Let him thank M. Musso and M. del Pino for their kind invitation and hospitality.
\end{appendices}


\begin{thebibliography}{AAA}
\bibitem{AbSte} M. Abramowitz, I.A. Stegun, {\em Handbook of mathematical functions with formulas, graphs, and mathematical tables}, National Bureau of Standards Applied Mathematics Series, 55. For sale by the Superintendent of Documents, U.S. Government Printing Office, Washington, 1964.

\bibitem{Abr}A.A. Abrikosov, {\em On the magnetic properties of superconductors of the second group}. Soviet Phys. JETP {\bf 5}
(1957), 1174--1182.

\bibitem{Ap} T.M. Apostol, {\em Modular functions and Dirichlet series in number theory. Second edition}, Graduate Texts in Mathematics, 41. Springer-Verlag, New York, 1990.

\bibitem{BaPa}S. Baraket, F. Pacard, {\em Construction of singular limits for a semilinear elliptic equation in dimension $2$}. Calc. Var. Partial Differential Equations {\bf 6} (1998), no. 1, 1--38.

\bibitem{Bog}E.B. Bogomolnyi, {\em The stability of classical solutions}. Sov. J. Nucl. Phys. {\bf 24} (1976), 449--454.

\bibitem{CY} L.A. Caffarelli, Y. Yang, \emph{Vortex condensation in
the Chern-Simons-Higgs model: an existence theorem}. Comm. Math.
Phys. {\bf 168} (1995), no. 2, 321--336.

\bibitem{ChI}D. Chae, O. Imanuvilov, \emph{The existence of non-topological multivortex solutions in
the relativistic self-dual Chern-Simons theory}. Comm. Math. Phys.
{\bf 215} (2000), no. 1, 119--142.

\bibitem{CFL}H. Chan, C.C. Fu, C.-S. Lin, \emph{Non-topological multivortex solution
to the selfdual Chern-Simons-Higgs equation}. Comm. Math. Phys.
{\bf 231} (2002), 189--221.

\bibitem{CLW}C.C. Chen, C.-S. Lin, G. Wang, \emph{Concentration phenomena of two-vortex solutions in a Chern-Simons model}, Ann. Sc. Norm. Sup. Pisa Cl. Sci. (5) {\bf 3} (2004), 367--397.

\bibitem{CHMY}X. Chen, S. Hastings, J.B. McLeod, Y. Yang, {\em A nonlinear elliptic equations arising from gauge
field theory and cosmology}. Proc. R. Soc. Lond. A {\bf 446}
(1994), no. 1928, 453--478.

\bibitem{ChO}X. Chen, Y. Oshita, \emph{An application of the modular function in nonlocal variational problems}, Arch. Ration. Mech. Anal. {\bf 186} (2007), no. 1, 109--132.

\bibitem{Ch}K. Choe, {\em Asymptotic behavior of condensate solutions in the Chern-Simons-Higgs theory}. J. Math. Phys. {\bf 48} (2007), no. 10, 103501, 17 pp.

\bibitem{DDeMW} J. Davila, M. del Pino, M. Musso, J. Wei, {\em Singular limits of a two-dimensional boundary value problem arising in corrosion modelling}. Arch. Ration. Mech. Anal. {\bf 182} (2006), no. 2, 181--221.

\bibitem{DEM4} M. del Pino, P. Esposito, M. Musso, {\em Two-dimensional Euler flows with concentrated vorticities}. Trans. Amer. Math. Soc. {\bf 362} (2012), no. 12, 6381--6395.

\bibitem{DEM5} M. del Pino, P. Esposito, M. Musso, {\em Linearized theory for entire solutions of a singular Liouvillle equation}. Proc. Amer. Math. Soc. {\bf 140} (2012), no. 2, 581--588.

\bibitem{dkm} M. del Pino, M. Kowalczyk, M. Musso, {\em Singular limits in Liouville-type equations}. Calc. Var. Partial Differential Equations {\bf 24} (2005), no. 1,  47--81.

\bibitem{DJLPW}W. Ding, J. Jost, J. Li, X. Peng, G. Wang, {\em Self duality equations for Ginburg-Landau and Seiberg-Witten type functionals with $6^{\hbox{th}}$ order potential}. Comm. Math. Phys. {\bf 217} (2001), 383--407.

\bibitem{DJLW2}W. Ding, J. Jost, J. Li, G. Wang, {\em An analysis of the two-vortex case in the Chern-Simons-Higgs model}. Calc. Var. Partial Differential Equations {\bf 7} (1998), 87--97.

\bibitem{D}G. Dunne, \emph{Selfdual Chern-Simons theories}, Lecture Notes in Physics
Monograph Series, 36. Springer-Verlag, Heidelberg, 1995.

\bibitem{EsFi}P. Esposito, P. Figueroa, {\em  Singular mean field equations on compact Riemann surfaces}, arXiv:1210.6162.

\bibitem{EGP}P. Esposito, M. Grossi, A. Pistoia, {\em On the existence of blowing-up solutions for a mean field equation}. Ann. IHP Analyse Non Lin{\'e}aire {\bf 22} (2005), no. 2 , 227--257.

\bibitem{Fi}P. Figueroa, \emph{Singular limits for
Liouville-type equations on the flat two-torus}, Calc. Var. Partial Differential Equations (2013), doi 10.1007/s00526-012-0594-0.

\bibitem{H}J. Han, \emph{Existence of topological multivortex solutions in self dual
gauge theory}. Proc. Roy. Soc. Edinburgh {\bf 130} (2000),
1293--1309.

\bibitem{HKP}J. Hong, Y. Kim, P.Y. Pac, {\em Multivortex solutions of the abelian Chern-Simons-Higgs theory}.
Phys. Rev. Lett. {\bf 64} (1990), no. 19, 2230--2233.

\bibitem{JW}R. Jackiw, E.J. Weinberg, {\em Self-dual Chern-Simons vortices}.
Phys. Rev. Lett. {\bf 64} (1990), no. 19, 2234--2237.

\bibitem{JaTa}A. Jaffe, C. Taubes, {\em Vortices and monopoles}, Progress in Physics, 2. Birkh\"auser, Boston, 1980.

\bibitem{LiWa}C.-S. Lin, C.-L. Wang, \emph{Elliptic functions, Green functions and the mean field equations on tori}. Ann. of Math. (2) {\bf 172} (2010), no. 2, 911--954.

\bibitem{LinYan}C.-S. Lin, S. Yan, {\em Bubbling solutions for relativistic abelian Chern-Simons
model on a torus}. Comm. Math. Phys. {\bf 297} (2010), 733--758.

\bibitem{LinYan1}C.-S. Lin, S. Yan, {\em Existence of Bubbling Solutions for Chern-Simons Model on a Torus}. Arch. Ration. Mech. Anal. {\bf 207} (2013), no. 2, 353-392.

\bibitem{Nol} M. Nolasco, {\em Nontopological $N$-vortex condensates for the self-dual Chern-Simons theory}. Comm. Pure Appl. Math. {\bf 56} (2003), no. 12, 1752--1780.

\bibitem{NoTa3}M. Nolasco, G. Tarantello, {\em Double vortex condensates in the Chern-Simons-Higgs theory}. Calc. Var. Partial Differential Equations {\bf 9} (1999), 31--91.

\bibitem{SY1}J. Spruck, Y. Yang, {\em The existence of non-topological solutions in the self-dual Chern-Simons theory}. Comm. Math. Phys. {\bf 149} (1992), 361--376.

\bibitem{SY2}J. Spruck, Y. Yang, {\em Topological solutions in the self-dual Chern-Simons theory: existence
and approximation}. Ann. IHP Analyse Non Lin\'eaire {\bf 12}
(1995), no. 1, 75--97.

\bibitem{T}G. Tarantello, {\em Multiple condensate solutions for the Chern-Simons-Higgs theory}. J. Math. Phys. {\bf 37} (1996), 3769--3796.

\bibitem{Tbook}G. Tarantello, {\em Self-dual gauge field theories. An analytical approach}, Progress in Nonlinear Differential Equations and their Applications, 72. Birkhäuser, Boston, 2008.

\bibitem{Taubes}C. Taubes, {\em Arbitrary N-vortex solutions for the
first order Ginzburg-Landau equations}. Comm. Math. Phys. {\bf 72}
(1980), 277--292.

\bibitem{tHo}G. 't Hooft, {\em A property of electric and magnetic flux in nonabelian gauge theory}. Nuclear Phys. B {\bf 153} (1979), 141--160.

\bibitem{RWa}R. Wang, {\em The existence of Chern-Simons vortices}. Comm. Math. Phys. {\bf 137} (1991), 587--597.

\bibitem{Ybook}Y. Yang, {\em Solitons in field theory and nonlinear analysis}, Springer Monographs in Mathematics. Springer-Verlag, New York, 2001.

\end{thebibliography}
\end{document}